\documentclass{amsart}
\usepackage{amsmath,amssymb,fullpage}
\usepackage{amsthm,tikz}
\usetikzlibrary{shapes.geometric}
\usetikzlibrary{shadings}
\usepackage{amsfonts,newlfont,url,xspace}
\usepackage{stmaryrd}
\usepackage{graphicx,verbatim}
\usepackage{float}
\usepackage{hyperref}
\usepackage{soul}
\usepackage{diagbox}
\usepackage{mathtools}

\newtheorem{thm}{Theorem}[section]
\newtheorem{cor}[thm]{Corollary}
\newtheorem{lem}[thm]{Lemma}
\newtheorem{prop}[thm]{Proposition}
\newtheorem{conj}[thm]{Conjecture}

\theoremstyle{definition}
\newtheorem{definition}[thm]{Definition}
\newtheorem{ex}[thm]{Example}
\newtheorem{remark}[thm]{Remark}

\numberwithin{equation}{section}

\usepackage{mathtools}

\DeclareMathOperator{\lcm}{lcm}

\usepackage{ifxetex,ifluatex}
\if\ifxetex T\else\ifluatex T\else F\fi\fi T%
  \usepackage{fontspec}
\else
  \usepackage[T1]{fontenc}
  \usepackage[utf8]{inputenc}
  \usepackage{lmodern}
\fi

\newcommand{\rank}{\textrm{rk}}
\newcommand{\Iex}{\{5,8,9,10,11,16\}}

\newcommand{\IC}{\mathcal{IC}}
\renewcommand{\O}{\mathcal{O}}

\newcommand{\row}{\mathrm{Row}}
\newcommand{\Max}{\mathrm{Max}}
\newcommand{\Min}{\mathrm{Min}}
\newcommand{\fl}{\mathrm{Floor}}
\newcommand{\inc}{\mathrm{Inc}}
\newcommand{\comp}{\mathrm{Comp}}
\newcommand{\f}{\nabla}
\newcommand{\oi}{\Delta}
\newcommand{\tog}{\mathfrak{T}}
\newcommand{\ceil}[1]{\mathrm{Ceil}({#1})}
\newcommand{\A}{\inc_I\big(\ceil{I}\big)}

\newcommand{\F}{\Min(I)\cap\oi\ceil{I}}
\newcommand{\arow}{\inc(I)\cup\Big(\oi\inc_{I}\big(\ceil{I}\big) -\big(I\cup\oi\ceil{I}\big)\Big)\cup\Big(\oi\ceil{I}-\oi(\F) \Big)}
\newcommand{\arowcomp}{\Big(\oi\inc_I(\ceil{I})-\big(I\cup\oi\ceil{I}\big)\Big)\cup\Big(\oi\ceil{I}-\oi\big(\F\big)\Big)}
\newcommand{\mm}{\mathfrak{M}}

\usepackage{hyperref}

\title{Toggling, rowmotion, and homomesy on interval-closed sets}

\author[1]{Jennifer Elder$^1$}
\address[1]{Missouri Western State University. \href{mailto:jelder8@missouriwestern.edu}{jelder8@missouriwestern.edu}}
\author[2]{Nadia Lafreni\`ere$^2$}
\address[2]{Dartmouth College. \href{mailto:nadia.lafreniere@dartmouth.edu}{nadia.lafreniere@dartmouth.edu}}
\author[3]{Erin McNicholas$^3$}
\address[3]{Willamette University. \href{mailto:emcnicho@willamette.edu}{emcnicho@willamette.edu}}
\author[4]{Jessica Striker$^4$}
\address[4]{North Dakota State University. \href{mailto:jessica.striker@ndsu.edu}{jessica.striker@ndsu.edu}}
\author[5]{Amanda Welch$^5$}
\address[5]{Eastern Illinois University. \href{mailto:arwelch@eiu.edu}{arwelch@eiu.edu}}

\begin{document}

\begin{abstract}
Interval-closed sets of a poset are a natural superset of order ideals. We initiate the study of interval-closed sets of finite posets from enumerative and dynamical perspectives. In particular, we use the generalized toggle group to define rowmotion on interval-closed sets as a product of these toggles. Our main theorem is an intricate global characterization of rowmotion on interval-closed sets, which we show is equivalent to the toggling definition. We also study specific posets; we enumerate interval-closed sets of ordinal sums of antichains, completely describe their rowmotion orbits, and prove a homomesy result involving the signed cardinality statistic. Finally, we study interval-closed sets of product of chains posets, proving further results about enumeration and homomesy. 
\end{abstract}

\maketitle

\tableofcontents

\section{Introduction}
The action of \textit{rowmotion}  has been studied in a variety of contexts in recent years (see e.g.\ \cite{Striker2017,DPS2017,Rsystems,DSV2019,Okada21,BSV21,HopkinsRow22,IyamaMarczinzik2022}). It came to prominence as an action on order ideals of posets (equivalently, elements of distributive lattices)~\cite{CF1995}, and serves as a motivating example of the cyclic sieving~\cite{ReStWh2004} and homomesy~\cite{PR2015} phenomena. Rowmotion has recently been extended to more general types of lattices \cite{ThomasWilliams19Slow,ThomasWilliams19Indep,Semidistrim} and other kinds of objects~\cite{Striker2018}. 

There are two equivalent characterizations of rowmotion on order ideals; the first was a global definition using convex-closure~\cite{Duchet1974,BS1974}. Then, Cameron and Fon-der-Flaass realized rowmotion on order ideals as an element of a permutation group \cite{CF1995} now called the \textit{toggle group} \cite{SW2012}. (A \emph{toggle} for a given poset element acts on an order ideal as the symmetric difference of the element and the order ideal if the result is an order ideal, and as the identity otherwise.) The fourth author with Williams \cite{SW2012} used this perspective to show {rowmotion} on order ideals is conjugate to another toggle group action,  which proved easier to analyze in many cases. This yielded  proofs of the \emph{cyclic sieving phenomenon} (which gives the orbit sizes explicitly) of order ideals under rowmotion in many posets of interest, including products of chains $[m]\times[n]$ and $[m]\times[n]\times[2]$. Propp and Roby \cite{PR2015} and Vorland \cite{Vorland2019} showed rowmotion on order ideals of these products of chains posets exhibits the \emph{homomesy phenomenon}, in which the average value of a statistic over any orbit of an action equals the global average.

In \cite{Striker2018}, the notion of the toggle group was generalized beyond order ideals to any family of subsets of a finite set (as the symmetric difference of the element and the subset if the result is in the set of subsets, and as the identity otherwise). This included structures found in posets (including chains and antichains) and graphs (including independent sets and vertex covers), as well as abstractions of these (matroids and convex geometries). For each of these families, toggle commutation lemmas and toggle group structure theorems were given, but there was no exploration of specific posets or products of toggles.
Several papers have since been written to find such results on toggling independent sets of certain graphs~\cite{JosephRoby,Joseph19,JosephRoby20,defant2023torsors}, but to our knowledge, there has been little to no subsequent work on other families. 

The present paper is the first dedicated study of enumeration and toggle dynamics of the set of \emph{interval-closed sets} of a poset, a natural superset of order ideals.  Using the definition of generalized toggles in~\cite{Striker2018}, we directly extend the notion of rowmotion from order ideals to interval-closed sets (Definition \ref{def:Row_tog}). 
(Toggling on interval-closed sets was briefly mentioned in \cite{Striker2018}, as an example of an interesting set of subsets on which one may apply generalized toggles; see  Remark~\ref{remark:gentoggle} of the present paper for details.)

Our first main  result, Theorem \ref{thm:AltRow}, gives a global construction of interval-closed set rowmotion, which is far more complex than the analogous statement for order ideals, and shows it is equivalent to the toggling definition. 
We also give a homomesy result for interval-closed sets of any poset (Proposition~\ref{toggleabilitly_max_min}).
The rest of the paper studies enumeration, rowmotion orbits, and homomesy on interval-closed sets of posets belonging to various families.

In Section \ref{sec:ordinalsum}, we discuss  chains and {ordinal sums} of antichains.
We enumerate the interval-closed sets in Proposition~\ref{thm:chains_ICS} and Theorem~\ref{thm:gen_ord_sum_ics_card} and completely characterize their rowmotion orbits in Theorems~\ref{thm:chains_orbits} and \ref{thm:gen_ord_sum_row}. We also show in Theorem~\ref{alt_def_ordinal} that the global rowmotion characterization of Theorem~\ref{thm:AltRow} simplifies nicely in the setting of ordinal sums of antichains. When the size of each antichain is equal and there are an even number of these antichains, we show in Theorem~\ref{thm:signed_card_ord_sum} the signed cardinality statistic exhibits homomesy on these orbits. 

In Section \ref{sec:products_of_chains}, we first prove Theorem~\ref{prodofchainICS},  enumerating interval-closed sets of products of chains of the form $[2]\times [n]$.  In Theorem~\ref{thm:enum_ICS_m_by_n_with_items_in_each_chain}, we show that the collection of interval-closed sets of $[m]\times [n]$ with at least one element from each chain is in bijection with order ideals of $[m]\times[n-1]\times[2]$, so they are enumerated by the Narayana numbers. We also show in Theorem~\ref{thm:homomesy_2_by_n_max-min} that the number of maximal elements minus the number of minimal elements is homomesic on interval-closed sets of $[2] \times [n]$ and conjecture a similar statement for $[m]\times[n]$ (Conjecture~\ref{cj:homomesy_prod_2_chains}). We also conjecture that the signed cardinality statistic is $0$-mesic with respect to rowmotion of interval-closed sets of $[m]\times[n]$ with $m+n-1$ even and $m=2,3$ (Conjecture~\ref{conj:signed_card}).

This paper contains several conjectures and open problems. Many conjectures were found experimentally with SageMath~\cite{sage}, and we encourage researchers interested in the problems we suggest to experiment with our code. For that purpose, we provide a worksheet containing the methods for enumerating interval-closed sets and for acting on them by rowmotion. It can be found at \url{https://nadialafreniere.github.io/ICS-rowmotion-worksheet}.

\subsection*{Acknowledgements} We would like to thank the referee for helpful comments. The genesis for this project was the Research Community in Algebraic Combinatorics workshop, hosted by ICERM and funded by the NSF.  We also thank ICERM for the opportunity to continue this work in the Collaborate@ICERM program.
We thank the developers of OEIS~\cite{OEIS} and SageMath~\cite{sage}, which were useful in this research, and the CoCalc~\cite{SMC} collaboration platform. JS was supported by a grant from the Simons Foundation/SFARI (527204, JS) and NSF grant DMS-2247089.

\section{Toggling and rowmotion on interval-closed sets}\label{sec:togglingICS}
\subsection{Poset preliminaries}
\label{sec:def}
Let $n\in \mathbb{N}$ and let $(P,\leq_P)$ be a partially ordered set (poset); we drop the subscript on the order relation when the poset is understood. All posets in this paper are finite; we denote the number of elements in a poset $P$ as $|P|$.
We follow \cite[Ch.\ 3]{Stanley2011} for standard poset terminology and list the most important definitions for our work below. We postpone the definition of interval-closed set and associated notions to Section~\ref{sec:ICS}.

For $a,b\in P$, $a<b$ means $a\leq b$ and $a\neq b$. We say $b$ \emph{covers} $a$, denoted $a\lessdot b$, if $a<b$ and
there is no $p\in P$ such that $a < p < b$. The \emph{interval} $[a,b]$ is the set $\{p\in P \ | \ a\leq p\leq b\}$. Given a subset $S$ of $P$, a \emph{minimal} element of $S$ is an element that covers no other element in $S$. A \emph{maximal} element of $S$ is an element that is covered by no other element of $S$. A \emph{linear extension} of $P$ is a total order $\leq_P^*$ that \emph{extends} $\leq_P$, in the sense that if $a\leq_P b$, then $a\leq_P^* b$. We often write a linear extention as an ordered tuple of all poset elements: $(x_1,x_2,\ldots,x_{|P|})$.
We say $P$ is \emph{ranked} if there exists a function $\rank:P\rightarrow\mathbb{Z}$ such that $a\lessdot b$ implies $\rank(b)=\rank(a)+1$. If we specify that  the smallest value attained by $\rank$ on $P$ is $0$, then for each element $p\in P$, we say $\rank(p)$ is its \emph{rank}, while the rank of the poset is the largest value attained by $\rank$.

A \textit{chain} of $P$ is a totally ordered subset of $P$. An $n$-element \textit{chain poset} (where all elements are distinct) is denoted as $[n]$.
An \textit{antichain} of $P$ is a subset of pairwise incomparable elements. An \emph{antichain poset} of $n$ distinct, pairwise incomparable elements is denoted as $\mathbf{n}$. (Note that $[1]=\mathbf{1}$.)
Given posets $P$ and $Q$, their \textit{disjoint union} $P + Q$ is the poset with elements $P\cup Q$ and  order relation $a \leq_{P + Q} b$ if and only if either  $a,b\in P$ and $a\leq_P b$, or  $a,b\in Q$ and $a\leq_Q b$. 
The \textit{ordinal sum} of posets $P \oplus Q$ is a poset with elements $P \cup Q$ and order relation $a \leq_{P\oplus Q} b$ if and only if one of the following holds:
$a,b\in P$ and $a \leq_P b$,
$a,b\in Q$ and $a \leq_Q b$, or
$a \in P$ and $b \in Q$.
The \emph{Cartesian product} $P\times Q$ is the poset with elements $\{(p,q) \ | \ p\in P, q\in Q\}$ and partial order given by $(a,b)\leq_{P\times Q} (c,d)$ if and only if componentwise comparisons $a\leq_P c$ and $b\leq_Q d$ are true.
 A subset $I\subseteq P$ is an \emph{order ideal} if whenever $b\in I$ and $a\leq b$, then $a\in I$. A subset $I$ is an \emph{order filter} if whenever $a\in I$ and $a\leq b$, then $b\in I$.
Given a subset $S$ of $P$, its \emph{complement} $\overline{S}$ is $P- S$.  Note that the complement of an order ideal is an order filter, and vice versa.

\subsection{Interval-closed sets}
\label{sec:ICS}

We begin by defining interval-closed sets and important associated concepts.
\begin{definition} Let $P$ be a poset and $I$ a subset of $P$. We say that $I$ is an \emph{interval-closed set} if for all $x, y \in I$ such that $x \leq y$, then $z \in I$ if $x \leq z \leq y$ (that is, the entire interval $[x,y]$ is in $I$). Let $\IC(P)$ denote the set of interval-closed sets of $P$.
\end{definition}

Order ideals and order filters are both examples of interval-closed sets, though there are other types, as the following example shows.

\begin{ex}\label{ex:Exfirst} Consider the poset in Figure~\ref{fig:ex12}, where we have highlighted the interval-closed set $I = \Iex$. $I$ is an example of an interval-closed set that is neither an order ideal nor an order filter. 

Note that $5$, $8$, $11$, and $16$ could each be removed from $I$ and we would still have an interval-closed set ($8$ because it is incomparable to the other elements of $I$ and $5$, $11$, and $16$ because they are minimal and maximal). However, removing $9$ or $10$ from $I$ would result in a subset of $P$ that is not an interval-closed set.
\end{ex}

\begin{figure}[htbp]
\centering
	\begin{minipage}{.4\linewidth}\centering
		\begin{tikzpicture}[scale = .5]
		\node[draw,circle](A) at (5,0) {1};
		\node[draw,circle](B) at (2,2) {2};
		\node[draw,circle](C) at (4,2) {3};
		\node[draw,circle](D) at (6,2) {4};
		\node[draw,circle,fill=red!60](E) at (8,2) {5};
		\node[draw,circle](T) at (10,2) {6};
		\node[draw,circle](F) at (2,4) {7};
		\node[draw,circle,fill=red!60](G) at (4,4) {8};
    \node[draw,circle,fill=red!60](H) at (6,4) {9};
    \node[draw,circle,fill=red!60](I) at (8,4) {10};
    \node[draw,circle,fill=red!60](J) at (10,4) {11};
    \node[draw,circle](K) at (0,6) {12};
    \node[draw,circle](L) at (2,6) {13};
    \node[draw,circle](M) at (4,6) {14};
    \node[draw,circle](N) at (6,6) {15};
    \node[draw,circle,fill=red!60](O) at (8,6) {16};
    \node[draw,circle](P) at (10,6) {17};
    \node[draw,circle](Q) at (0,8) {18};
    \node[draw,circle](R) at (3,8) {19};
    \node[draw,circle](S) at (6,8) {20};
		\draw[ultra thick] (A) --(B) --(F) --(K) --(Q);
		\draw[ultra thick] (F) --(L) --(R);
		\draw[ultra thick] (F) --(M) --(S);
		\draw[ultra thick] (B) --(G) --(M) --(R);
		\draw[ultra thick] (G) --(N) --(S);
    \draw[ultra thick] (A) --(C) --(H) --(O) --(S);
    \draw[ultra thick] (A) --(D) --(I) --(P) --(S);
    \draw[ultra thick] (A) --(E) --(I) --(O);
    \draw[ultra thick] (E) --(H) --(N);
    \draw[ultra thick] (A) --(C) --(H) --(O) --(S);
    \draw[ultra thick] (I) --(N);
    \draw[ultra thick] (L) --(S);
    \draw[ultra thick] (J) --(O);
    \draw[ultra thick] (J) --(T);
\end{tikzpicture}
\end{minipage}
	\caption{The interval-closed set from Examples~\ref{ex:Exfirst} and \ref{ex:Exsecond}} \label{fig:ex12}
\end{figure}
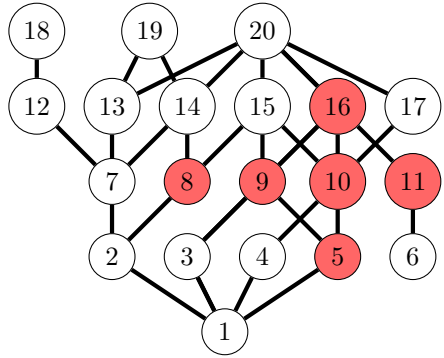

\begin{definition}\label{def:poset_stuff} Given $I\in \IC(P)$, let $\oi(I)$ denote the smallest order ideal containing  $I$, and let $\f(I)$ denote the smallest order filter containing $I$.
The set of \emph{elements of $P$ comparable to $I$} is  $\comp(I):=\f(I)\cup\oi(I)$.  We denote its complement, the set of \emph{elements of $P$ incomparable to $I$}, by $\inc(I)$.  The set of elements of a subset $J\subseteq P$ which are incomparable to $I$ is denoted $\inc_J(I):=J\cap\inc(I)$.
Denote the minimal elements of $I$ as $\Min(I)$ and the maximal elements of $I$ as $\Max(I)$.  We define the \emph{floor} of $I$ to be the set of maximal elements in $\oi(I)-I$, and denote it $\fl(I)$.
\end{definition}

Note that given an interval-closed set $I$,  $\Max(I)$,  $\Min(I)$, and  $\fl(I)$ are each antichains in $P$.

\begin{ex}\label{ex:Exsecond} Consider again the interval-closed set $I = \Iex$ in Figure~\ref{fig:ex12}. Then $\Max(I) = \{8,16\}$, $\Min(I)= \{5, 8, 11\}$, and $\fl(I) = \{2, 3, 4, 6\}$. Figure \ref{fig:Sec2_orders} shows $\inc(I)$, $\oi(I)$, and $\f(I)$.

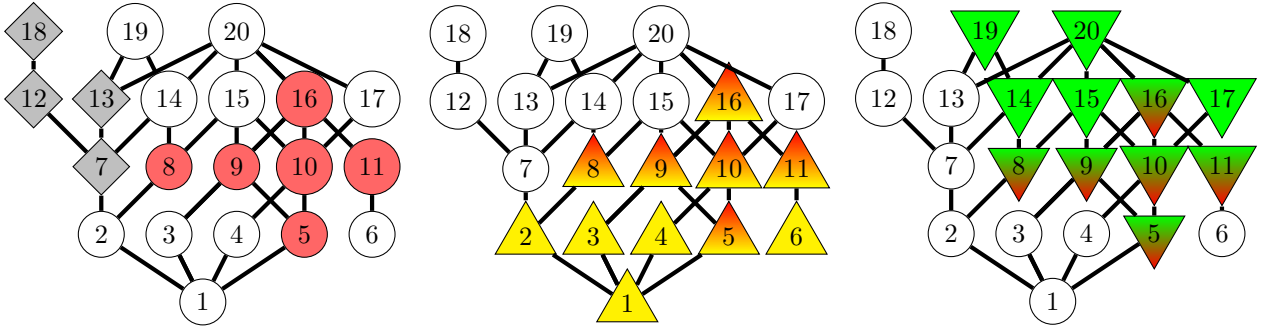
\begin{figure}\centering
	\begin{minipage}{.3\linewidth}\centering
		\begin{tikzpicture}[scale = .45]
		\node[draw,circle](A) at (5,0) {1};
		\node[draw,circle](B) at (2,2) {2};
		\node[draw,circle](C) at (4,2) {3};
		\node[draw,circle](D) at (6,2) {4};
		\node[draw,circle,fill=red!60](E) at (8,2) {5};
		\node[draw,diamond, fill=lightgray, inner sep=.9mm](F) at (2,4) {7};
		\node[draw,circle,fill=red!60](G) at (4,4) {8};
    \node[draw,circle,fill=red!60](H) at (6,4) {9};
    \node[draw,circle,fill=red!60](I) at (8,4) {10};
    \node[draw,circle, fill=red!60](J) at (10,4) {11};
		\node[draw,circle](T) at (10,2) {6};
    \node[draw,diamond, fill=lightgray, inner sep=.5mm](K) at (0,6) {12};
    \node[draw,diamond, fill=lightgray, inner sep=.5mm](L) at (2,6) {13};
    \node[draw,circle](M) at (4,6) {14};
    \node[draw,circle](N) at (6,6) {15};
    \node[draw,circle,fill=red!60](O) at (8,6) {16};
    \node[draw,circle](P) at (10,6) {17};
    \node[draw,diamond, fill=lightgray, inner sep=.5mm](Q) at (0,8) {18};
    \node[draw,circle](R) at (3,8) {19};
    \node[draw,circle](S) at (6,8) {20};
		\draw[ultra thick] (A) --(B) --(F) --(K) --(Q);
		\draw[ultra thick] (F) --(L) --(R);
		\draw[ultra thick] (F) --(M) --(S);
		\draw[ultra thick] (B) --(G) --(M) --(R);
		\draw[ultra thick] (G) --(N) --(S);
    \draw[ultra thick] (A) --(C) --(H) --(O) --(S);
    \draw[ultra thick] (A) --(D) --(I) --(P) --(S);
    \draw[ultra thick] (A) --(E) --(I) --(O);
    \draw[ultra thick] (E) --(H) --(N);
    \draw[ultra thick] (A) --(C) --(H) --(O) --(S);
    \draw[ultra thick] (I) --(N);
    \draw[ultra thick] (L) --(S);
    \draw[ultra thick] (J) --(O);
    \draw[ultra thick] (J) --(T);
\end{tikzpicture}
\end{minipage}\qquad
	\begin{minipage}{.3\linewidth}\centering
		\begin{tikzpicture}[scale = .45]
		\node[regular polygon, draw, regular polygon sides=3,  inner sep=.55mm, fill=yellow](A) at (5,0) {1};
		\node[regular polygon, draw, regular polygon sides=3,  inner sep=.55mm, fill=yellow](B) at (2,2) {2};
		\node[regular polygon, draw, regular polygon sides=3,  inner sep=.55mm, fill=yellow](C) at (4,2) {3};
		\node[regular polygon, draw, regular polygon sides=3,  inner sep=.55mm, fill=yellow](D) at (6,2) {4};
		\node[regular polygon, regular polygon sides=3,  inner sep=.55mm, top color = red, bottom color= yellow, draw](E) at (8,2) {5};
		\node[draw,circle](F) at (2,4) {7};
   \node[regular polygon, regular polygon sides=3,  inner sep=.55mm, top color = red, bottom color= yellow, draw](G) at (4,4) {8};
    \node[regular polygon, regular polygon sides=3,  inner sep=.55mm, top color = red, bottom color= yellow, draw](H) at (6,4) {9};
    \node[regular polygon, draw, regular polygon sides=3,  inner sep=.03mm, top color = red, bottom color= yellow](I) at (8,4) {10};
    \node[regular polygon, regular polygon sides=3,  inner sep=.03mm, top color = red, bottom color= yellow, draw](J) at (10,4) {11};
		\node[regular polygon, draw, regular polygon sides=3,  inner sep=.55mm, fill=yellow](T) at (10,2) {6};
    \node[draw,circle](K) at (0,6) {12};
    \node[draw,circle](L) at (2,6) {13};
    \node[draw,circle](M) at (4,6) {14};
    \node[draw,circle](N) at (6,6) {15};
    \node[regular polygon, regular polygon sides=3,  inner sep=.07mm, top color = red, bottom color= yellow,  draw](O) at (8,6) {16};
    \node[draw,circle](P) at (10,6) {17};
    \node[draw,circle](Q) at (0,8) {18};
    \node[draw,circle](R) at (3,8) {19};
    \node[draw,circle](S) at (6,8) {20};
		\draw[ultra thick] (A) --(B) --(F) --(K) --(Q);
		\draw[ultra thick] (F) --(L) --(R);
		\draw[ultra thick] (F) --(M) --(S);
		\draw[ultra thick] (B) --(G) --(M) --(R);
		\draw[ultra thick] (G) --(N) --(S);
    \draw[ultra thick] (A) --(C) --(H) --(O) --(S);
    \draw[ultra thick] (A) --(D) --(I) --(P) --(S);
    \draw[ultra thick] (A) --(E) --(I) --(O);
    \draw[ultra thick] (E) --(H) --(N);
    \draw[ultra thick] (A) --(C) --(H) --(O) --(S);
    \draw[ultra thick] (I) --(N);
    \draw[ultra thick] (L) --(S);
    \draw[ultra thick] (J) --(O);
    \draw[ultra thick] (J) --(T);
\end{tikzpicture} 
\end{minipage}\qquad
	\begin{minipage}{.3\linewidth}\centering
		\begin{tikzpicture}[scale = .45]
		\node[draw,circle](A) at (5,0) {1};
		\node[draw,circle](B) at (2,2) {2};
		\node[draw,circle](C) at (4,2) {3};
		\node[draw,circle](D) at (6,2) {4};
		\node[regular polygon, regular polygon sides=3,  inner sep=.55mm, top color = green, bottom color= red, shape border rotate=180,draw](E) at (8,2) {5};
		\node[draw,circle](F) at (2,4) {7};
		\node[regular polygon, regular polygon sides=3,  inner sep=.55mm, top color = green, bottom color= red, shape border rotate=180,draw](G) at (4,4) {8};
    \node[regular polygon, regular polygon sides=3,  inner sep=.55mm, top color = green, bottom color= red, shape border rotate=180,draw](H) at (6,4) {9};
    \node[regular polygon, regular polygon sides=3,  inner sep=.1mm, top color = green, bottom color= red, shape border rotate=180,draw](I) at (8,4) {10};
    \node[regular polygon, regular polygon sides=3,  inner sep=.1mm, top color = green, bottom color= red, shape border rotate=180,draw](J) at (10,4) {11};
		\node[draw,circle](T) at (10,2) {6};
    \node[draw,circle](K) at (0,6) {12};
    \node[draw,circle](L) at (2,6) {13};
    \node[regular polygon, regular polygon sides=3,  inner sep=.1mm, fill=green, shape border rotate=180,draw](M) at (4,6) {14};
    \node[regular polygon, regular polygon sides=3,  inner sep=.1mm, fill=green, shape border rotate=180,draw](N) at (6,6) {15};
    \node[regular polygon, regular polygon sides=3,  inner sep=.1mm, top color = green, bottom color= red, shape border rotate=180,draw](O) at (8,6) {16};
    \node[regular polygon, regular polygon sides=3,  inner sep=.1mm, fill=green, shape border rotate=180,draw](P) at (10,6) {17};
    \node[draw,circle](Q) at (0,8) {18};
    \node[regular polygon, regular polygon sides=3,  inner sep=.1mm, fill=green, shape border rotate=180,draw](R) at (3,8) {19};
    \node[regular polygon, regular polygon sides=3,  inner sep=.1mm, fill=green, shape border rotate=180,draw](S) at (6,8) {20};
		\draw[ultra thick] (A) --(B) --(F) --(K) --(Q);
		\draw[ultra thick] (F) --(L) --(R);
		\draw[ultra thick] (F) --(M) --(S);
		\draw[ultra thick] (B) --(G) --(M) --(R);
		\draw[ultra thick] (G) --(N) --(S);
    \draw[ultra thick] (A) --(C) --(H) --(O) --(S);
    \draw[ultra thick] (A) --(D) --(I) --(P) --(S);
    \draw[ultra thick] (A) --(E) --(I) --(O);
    \draw[ultra thick] (E) --(H) --(N);
    \draw[ultra thick] (A) --(C) --(H) --(O) --(S);
    \draw[ultra thick] (I) --(N);
    \draw[ultra thick] (L) --(S);
    \draw[ultra thick] (J) --(O);
    \draw[ultra thick] (J) --(T);
\end{tikzpicture} 
\end{minipage}
	\caption{The interval-closed set $I$ (shaded with red), the elements incomparable to $I$ $\inc(I)$ (diamond nodes shaded in grey), the order ideal $\oi(I)$ (triangular nodes $\triangle$ shaded with yellow), and the order filter $\f(I)$ (inverted triangular nodes $\triangledown$ shaded with green).} \label{fig:Sec2_orders}
\end{figure}

\end{ex}

The following proposition shows we can uniquely specify an interval-closed set by the  pair of antichains $(\Max(I),\fl(I))$.

\begin{prop}\label{prop:ICS_alt_def}
The map $I\mapsto (\Max(I),\fl(I))$ is a bijection between interval-closed sets of a poset $P$ and pairs of disjoint antichains $(A,B)$ of $P$ such that any element in $B$ is in the order ideal $\Delta(A)$ generated by $A$.
\end{prop}

\begin{proof}
Let $I$ be an interval-closed set and $f$ the function $f(I) = (\Max(I), \fl(I))$.
By definition, $\fl(I)$ is strictly below all elements of $I$, so in particular, $\fl(I)$ is strictly below $\Max(I)$. Hence, the image of $f$ is a pair of disjoint antichains $\Max(I)$ and $\fl(I)$ of $P$ such that any element in $\fl(I)$ is in the order ideal generated by $\Max(I)$. 

Conversely, let $A$ and $B$ be any two disjoint antichains of $P$ with all elements of $B$ in $\oi(A)$, the order ideal generated by $A$. 
The difference of order ideals $g(A,B) := \oi(A) - \oi(B)$ is an interval-closed set. 

This process is
such that $g\circ f(I) = I$ and $f \circ g(A, B) = (A, B)$, thus $f$ is a bijection.
\end{proof}

\subsection{Rowmotion as a composition of toggles}\label{sec:togrow}
Following \cite[Section 3.4]{Striker2018}, we define toggling on interval-closed sets. We then define \emph{rowmotion} on interval-closed sets as a product of these toggles and prove several lemmas and corollaries on the action of rowmotion on interval-closed sets of general posets.

\begin{definition} Let $x\in P$ and $I\in\IC(P)$ an interval-closed set of $P$. Define the \emph{toggle} $t_x:\IC(P)\rightarrow\IC(P)$ as follows:

\begin{itemize}
\item If $x \in I$,
$t_x(I) = 
\begin{cases} I-\{x\} &\text{if } I-\{x\}\in\IC(P) \\
I &\text{otherwise.}
\end{cases}$
\item If $x \notin I$,
$t_x(I) = 
\begin{cases} I \cup \{x\} &\text{if } I \cup\{x\} \in\IC(P)\\
I &\text{otherwise.}
\end{cases}$
\end{itemize}
That is, $x$ is toggled in/out of $I$ if doing so results in another interval-closed set.  
\end{definition}

\begin{remark}\label{remark:gentoggle}
The paper \cite{Striker2018} defined the \emph{interval-closed set toggle group} as the subgroup (of the symmetric group $\mathfrak{S}_{\IC(P)}$ on all interval-closed sets of $P$) generated by these toggles.
It also proved the following toggle commutation lemma and partially classified the group structure of the interval-closed set toggle group, leaving an open question about strengthening that classification.
\end{remark}

\begin{lem}[\protect{\cite[Lemma 3.17]{Striker2018}}] \label{lem:commute}
Given $x,y\in P$,
$(t_x t_y)^2 = 1$ if and only if $x$ and $y$ are
incomparable in $P$ or (in the case $x < y$) if $y$ covers $x$ where $y$ is maximal and $x$ is minimal.
\end{lem}

We can compose toggles to create interesting actions. If we compose toggles from top to bottom, we name this action rowmotion (as in the case of order ideals~\cite{SW2012}). 

\begin{definition} \label{def:Row_tog}
Given an interval-closed set $I\in\IC(P)$, the \emph{rowmotion} of $I$, $\row(I)$, is given by applying all toggles in the reverse order of any linear extension.
\end{definition}

For example, Figure~\ref{fig:Sec2_rowI} shows the interval-closed set $I = \Iex$ and its rowmotion $\row(I)=\{3,4,6,7,9,10,12,13,14,15,17,18\}$.

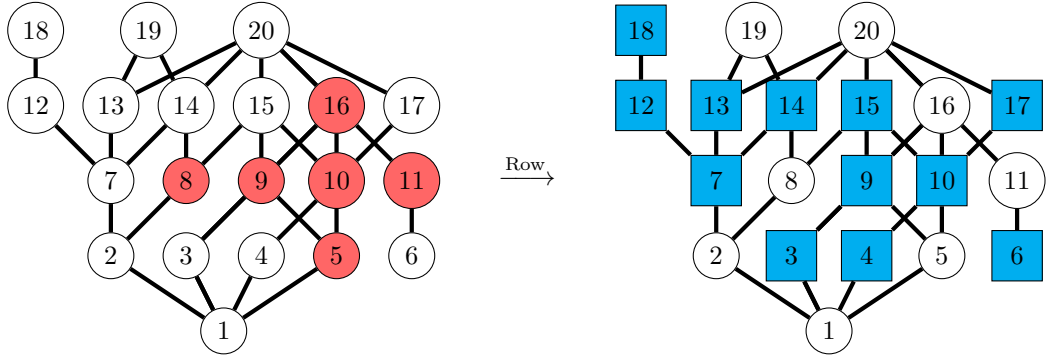
\begin{figure}\centering
	\begin{minipage}{.4\linewidth}\centering
		\begin{tikzpicture}[scale = .5]
		\node[draw,circle](A) at (5,0) {1};
		\node[draw,circle](B) at (2,2) {2};
		\node[draw,circle](C) at (4,2) {3};
		\node[draw,circle](D) at (6,2) {4};
		\node[draw,circle,fill=red!60](E) at (8,2) {5};
		\node[draw,circle](F) at (2,4) {7};
		\node[draw,circle,fill=red!60](G) at (4,4) {8};
    \node[draw,circle,fill=red!60](H) at (6,4) {9};
    \node[draw,circle,fill=red!60](I) at (8,4) {10};
    \node[draw,circle, fill=red!60](J) at (10,4) {11};
    \node[draw,circle](T) at (10,2) {6};
    \node[draw,circle](K) at (0,6) {12};
    \node[draw,circle](L) at (2,6) {13};
    \node[draw,circle](M) at (4,6) {14};
    \node[draw,circle](N) at (6,6) {15};
    \node[draw,circle,fill=red!60](O) at (8,6) {16};
    \node[draw,circle](P) at (10,6) {17};
    \node[draw,circle](Q) at (0,8) {18};
    \node[draw,circle](R) at (3,8) {19};
    \node[draw,circle](S) at (6,8) {20};
		\draw[ultra thick] (A) --(B) --(F) --(K) --(Q);
		\draw[ultra thick] (F) --(L) --(R);
		\draw[ultra thick] (F) --(M) --(S);
		\draw[ultra thick] (B) --(G) --(M) --(R);
		\draw[ultra thick] (G) --(N) --(S);
    \draw[ultra thick] (A) --(C) --(H) --(O) --(S);
    \draw[ultra thick] (A) --(D) --(I) --(P) --(S);
    \draw[ultra thick] (A) --(E) --(I) --(O);
    \draw[ultra thick] (E) --(H) --(N);
    \draw[ultra thick] (A) --(C) --(H) --(O) --(S);
    \draw[ultra thick] (I) --(N);
    \draw[ultra thick] (L) --(S);
    \draw[ultra thick] (J) --(O);
    \draw[ultra thick] (J) --(T);
\end{tikzpicture} 
\end{minipage}\quad $\xlongrightarrow{\text{Row}}$\quad
	\begin{minipage}{.4\linewidth}\centering
		\begin{tikzpicture}[scale = .5]
		\node[draw,circle](A) at (5,0) {1};
		\node[draw,circle](B) at (2,2) {2};
		\node[draw,rectangle, minimum size=6.7mm, fill=cyan](C) at (4,2) {3};
		\node[draw,rectangle, , minimum size=6.7mm, fill=cyan](D) at (6,2) {4};
		\node[draw,circle](E) at (8,2) {5};
		\node[draw,rectangle, fill=cyan, minimum size=6.7mm](F) at (2,4) {7};
		\node[draw,circle](G) at (4,4) {8};
    \node[draw,rectangle, minimum size=6.7mm,fill=cyan](H) at (6,4) {9};
    \node[draw,rectangle, minimum size=6.7mm,fill=cyan](I) at (8,4) {10};
    \node[draw,circle](J) at (10,4) {11};
		\node[draw,rectangle, minimum size=6.7mm,fill=cyan](T) at (10,2) {6};
    \node[draw,rectangle, fill=cyan, minimum size=6.7mm](K) at (0,6) {12};
    \node[draw,rectangle, fill=cyan, minimum size=6.7mm](L) at (2,6) {13};
    \node[draw,rectangle, , minimum size=6.7mm, fill=cyan](M) at (4,6) {14};
    \node[draw,rectangle, minimum size=6.7mm, fill=cyan](N) at (6,6) {15};
    \node[draw,circle](O) at (8,6) {16};
    \node[draw,rectangle, minimum size=6.7mm, fill=cyan](P) at (10,6) {17};
    \node[draw,rectangle, fill=cyan, minimum size=6.7mm](Q) at (0,8) {18};
    \node[draw,circle](R) at (3,8) {19};
    \node[draw,circle](S) at (6,8) {20};
		\draw[ultra thick] (A) --(B) --(F) --(K) --(Q);
		\draw[ultra thick] (F) --(L) --(R);
		\draw[ultra thick] (F) --(M) --(S);
		\draw[ultra thick] (B) --(G) --(M) --(R);
		\draw[ultra thick] (G) --(N) --(S);
    \draw[ultra thick] (A) --(C) --(H) --(O) --(S);
    \draw[ultra thick] (A) --(D) --(I) --(P) --(S);
    \draw[ultra thick] (A) --(E) --(I) --(O);
    \draw[ultra thick] (E) --(H) --(N);
    \draw[ultra thick] (A) --(C) --(H) --(O) --(S);
    \draw[ultra thick] (I) --(N);
    \draw[ultra thick] (L) --(S);
    \draw[ultra thick] (J) --(O);
    \draw[ultra thick] (J) --(T);
\end{tikzpicture} 
\end{minipage}
	\caption{The interval-closed set $I$ shown in red shaded circular nodes, and $\row(I)$ shown in blue shaded square nodes} \label{fig:Sec2_rowI}
\end{figure}

Rowmotion is well-defined by the commutativity of toggles for incomparable elements given in Lemma~\ref{lem:commute}.
As a bijective action on a finite set, rowmotion partitions the set into finite, disjoint subsets called \emph{orbits} given by:
$\O(I)=\{I, \row(I),\ldots, \row^{-1}(I)\}$. The \emph{order} of an action is the least common multiple of all orbit cardinalities.

The rest of this subsection discusses specifics about rowmotion on  interval-closed sets.

\begin{lem}
\label{lem:max_elt_comp}
Let $I\in \IC(P)$ be an interval-closed set containing all the maximal elements of $P$. Then $\row(I) = \overline{I}$.
\end{lem}

\begin{proof}
Let $I\in \IC(P)$ be an interval-closed set containing all the maximal elements of $P$. Thus $I$ is an order filter.
Construct a reverse linear extension $\mathcal{L}$ as follows: first add all maximal elements of $P$, then all elements in $I$ covered by those maximal elements, etc., until all elements of $I$ are reached. Then add the maximal element of $\overline{I}$, then the elements of $\overline{I}$ covered by those maximal elements, etc., until all elements of $P$ are included. We compute rowmotion by applying all toggles in the order of this reverse  linear extension, so each maximal element of $P$ is toggled out, then all elements of $I$ covered by these elements, etc. Rowmotion continues untoggling all elements in $I$ until all are removed. At this point, we arrive at the empty interval-closed set. After this, all remaining elements of $P$ are successively toggled in, producing the complement of~$I$.
\end{proof}

We have the following as a corollary.
\begin{cor}\label{cor:empty ICS}
Rowmotion on $\IC(P)$ has an orbit of size $2$, namely, $\{\emptyset, P \}$.
\end{cor}
\begin{proof}
By Lemma~\ref{lem:max_elt_comp}, $\row(P)=\overline{P}=\emptyset$. By definition, $\row(\emptyset)=P$, as all toggles act nontrivially.
\end{proof}

The following lemma shows that for posets of rank at most $1$, rowmotion on interval-closed sets acts as complementation, and thus has order $2$.
\begin{lem}\label{lem:at_most_1}
Let $P$ be a poset of rank at most one, and $I\in \IC(P)$. Then $\row(I) = \overline{I}$, thus the order of rowmotion is $2$.
\end{lem}

\begin{proof}
Let $P$ be a ranked poset of rank at most one, and let $S$ be any subset of $P$. Then, there exists no $z \in S$ such that $x,y \in S$ and $x<z<y$. Thus, the interval-closed condition is trivially satisfied, so $S \in \IC(P)$. Consequently, the action of toggling becomes, for any interval-closed set $I$,
\[ t_x(I) = \begin{cases} I - \{x\}& \text{if } x\in I\\
I\cup\{x\} & \text{otherwise.}\end{cases} \]
Since rowmotion toggles all elements of $P$, $\row(I) = \overline{I}$.
\end{proof}

Let $P^*$ denote the dual of a poset $P$ (the same elements and the reversed partial order), and let $t^*_x$ denote the toggle of $x$ acting on $\IC(P^*)$. Let $\row_P$ denote the action of rowmotion on interval-closed sets of the poset $P$.
\begin{prop} \label{Prop: DualInverse}
If $I\in \IC(P)$, then $I\in \IC(P^*)$ and $t_x(I)=t^*_x(I)$. Moreover, $\row^{-1}_P(I)=\row_{P^*}(I)$.
\end{prop}
\begin{proof}
If $I\in \IC(P)$, $I$ is  interval-closed under the relation~$\leq_P$, thus it is also interval-closed under the relation~$\geq_P$, which equals the relation $\leq_{P^*}$. Consequently, the toggle $t^*_x$ acts identically to the toggle $t_x$.

Let $|P|=n$
and $(x_1,x_2,\ldots,x_{n})$ be a linear extension of $P$. Then given $I\in\IC(P)$, $\row_P(I)=t_{x_1}\circ t_{x_2}\circ\cdots t_{x_n}(I)$, while $\row_P^{-1}(I)=t_{x_n}\circ t_{x_{n-1}}\circ\cdots\circ t_{x_1}(I)$. But $(x_n,x_{n-1},\ldots,x_{1})$ is a linear extension of $P^*$, so $\row_{P^*}(I)=t^*_{x_n}\circ t^*_{x_{n-1}}\circ\cdots\circ t^*_{x_1}(I)=t_{x_n}\circ t_{x_{n-1}}\circ\cdots\circ t_{x_1}(I)=\row_P^{-1}(I)$.
\end{proof}
The above proposition shows that as in the case of order ideals, toggling from bottom to top is inverse rowmotion.

\subsection{Rowmotion as a global action}
In this subsection, we state and prove Theorem~\ref{thm:AltRow}: an alternate, global description of rowmotion on interval-closed sets.

Let $P$ be a poset and $I\in\IC(P)$.
We recall the notation from Definition~\ref{def:poset_stuff}, 
which will be used in Theorem~\ref{thm:AltRow}.

We also need the following new definitions. 
\begin{definition}\label{defn: ceiling} 
Given an interval-closed set $I\in\IC(P)$, we define:
\begin{itemize}
\item the \emph{ceiling of $I$}, denoted $\ceil{I}$, as $\ceil{I}=\Min(\f(I) - I)$, and 
\item  the \emph{minimal elements of $I$ under the ceiling}, $\F,$ which is equivalent to the set $\{x\in \Min(I) \ | \ \f(\{x\})\cap \ceil{I}\not=\emptyset\}$. 
\end{itemize}
\end{definition}

\begin{ex} Figure \ref{fig:Sec2_Ceil_MinUnder} shows the interval-closed set $I=\Iex$, the ceiling of $I$, $\ceil{I}=\{14, 15, 17\}$, and the set of minimum elements under the ceiling of $I$, $\Min(I)\cap\ceil{I}=\{5, 8\}$. 
\end{ex}

\begin{figure}[htbp]\centering
\begin{minipage}{.45\linewidth}\centering
		\begin{tikzpicture}[scale = .5]
		\node[regular polygon, draw, regular polygon sides=3, color=blue, very thick, inner sep=.55mm](A) at (5,0) {1};
		\node[regular polygon, draw, regular polygon sides=3, color=blue, very thick, inner sep=.55mm](B) at (2,2) {2};
		\node[regular polygon, draw, regular polygon sides=3, color=blue, very thick, inner sep=.55mm](C) at (4,2) {3};
		\node[regular polygon, draw, regular polygon sides=3, color=blue, very thick, inner sep=.55mm](D) at (6,2) {4};
		\node[regular polygon, draw, regular polygon sides=3, color=blue, very thick, inner sep=.55mm, fill=red!60](E) at (8,2) {5};
		\node[regular polygon, draw, regular polygon sides=3, color=blue, very thick, inner sep=.55mm](F) at (2,4) {7};
   \node[regular polygon, draw, regular polygon sides=3, color=blue, very thick, inner sep=.55mm, fill=red!60](G) at (4,4) {8};
    \node[regular polygon, draw, regular polygon sides=3, color=blue, very thick, inner sep=.55mm, fill=red!60](H) at (6,4) {9};
    \node[regular polygon, draw, regular polygon sides=3, color=blue, very thick, inner sep=.07mm, fill=red!60](I) at (8,4) {10};
    \node[draw,circle,fill=red!60](J) at (10,4) {11};
		\node[draw, circle](T) at (10,2) {6};
    \node[draw,circle](K) at (0,6) {12};
    \node[draw,circle](L) at (2,6) {13};
    \node[regular polygon, draw, regular polygon sides=3, very thick, inner sep=.08mm, fill=blue!60](M) at (4,6) {14};
    \node[regular polygon, draw, regular polygon sides=3, very thick, inner sep=.08mm, fill=blue!60](N) at (6,6) {15};
    \node[draw,circle,fill=red!60](O) at (8,6) {16};
    \node[regular polygon, draw, regular polygon sides=3, very thick, inner sep=.08mm, fill=blue!60](P) at (10,6) {17};
    \node[draw,circle](Q) at (0,8) {18};
    \node[draw,circle](R) at (3,8) {19};
    \node[draw,circle](S) at (6,8) {20};
		\draw[ultra thick] (A) --(B) --(F) --(K) --(Q);
		\draw[ultra thick] (F) --(L) --(R);
		\draw[ultra thick] (F) --(M) --(S);
		\draw[ultra thick] (B) --(G) --(M) --(R);
		\draw[ultra thick] (G) --(N) --(S);
    \draw[ultra thick] (A) --(C) --(H) --(O) --(S);
    \draw[ultra thick] (A) --(D) --(I) --(P) --(S);
    \draw[ultra thick] (A) --(E) --(I) --(O);
    \draw[ultra thick] (E) --(H) --(N);
    \draw[ultra thick] (A) --(C) --(H) --(O) --(S);
    \draw[ultra thick] (I) --(N);
    \draw[ultra thick] (L) --(S);
    \draw[ultra thick] (J) --(O);
    \draw[ultra thick] (J) --(T);
\end{tikzpicture} 
\end{minipage}\qquad
\begin{minipage}{.45\linewidth}\centering
		\begin{tikzpicture}[scale = .5]
\node[draw,circle](A) at (5,0) {1};
		\node[draw,circle](B) at (2,2) {2};
		\node[draw,circle](C) at (4,2) {3};
		\node[draw,circle](D) at (6,2) {4};
		\node[regular polygon, draw, regular polygon sides=5, color=blue, very thick, inner sep=.8mm, fill=red!60](E) at (8,2) {5};
		\node[draw,circle](F) at (2,4) {7};
		\node[regular polygon, draw, regular polygon sides=5, color=blue, very thick, inner sep=.8mm, fill=red!60](G) at (4,4) {8};
    \node[draw,circle,fill=red!60](H) at (6,4) {9};
    \node[draw,circle,fill=red!60](I) at (8,4) {10};
    \node[draw,circle, fill=red!60](J) at (10,4) {11};
    \node[draw,circle](T) at (10,2) {6};
    \node[draw,circle](K) at (0,6) {12};
    \node[draw,circle](L) at (2,6) {13};
    \node[draw,circle](M) at (4,6) {14};
    \node[draw,circle](N) at (6,6) {15};
    \node[draw,circle,fill=red!60](O) at (8,6) {16};
    \node[draw,circle](P) at (10,6) {17};
    \node[draw,circle](Q) at (0,8) {18};
    \node[draw,circle](R) at (3,8) {19};
    \node[draw,circle](S) at (6,8) {20};
		\draw[ultra thick] (A) --(B) --(F) --(K) --(Q);
		\draw[ultra thick] (F) --(L) --(R);
		\draw[ultra thick] (F) --(M) --(S);
		\draw[ultra thick] (B) --(G) --(M) --(R);
		\draw[ultra thick] (G) --(N) --(S);
    \draw[ultra thick] (A) --(C) --(H) --(O) --(S);
    \draw[ultra thick] (A) --(D) --(I) --(P) --(S);
    \draw[ultra thick] (A) --(E) --(I) --(O);
    \draw[ultra thick] (E) --(H) --(N);
    \draw[ultra thick] (A) --(C) --(H) --(O) --(S);
    \draw[ultra thick] (I) --(N);
    \draw[ultra thick] (L) --(S);
    \draw[ultra thick] (J) --(O);
    \draw[ultra thick] (J) --(T);
\end{tikzpicture} 
\end{minipage}
	\caption{Left: the $\ceil{I}$ (shaded in blue), the order ideal $\oi(\ceil{I})$ (triangular nodes). Right: the minimal elements of $I$ under the ceiling $\Min(I)\cap\oi\ceil{I}$ (pentagonal nodes). The interval-closed set $I$ is shaded in red in both diagrams.} \label{fig:Sec2_Ceil_MinUnder} 
\end{figure}
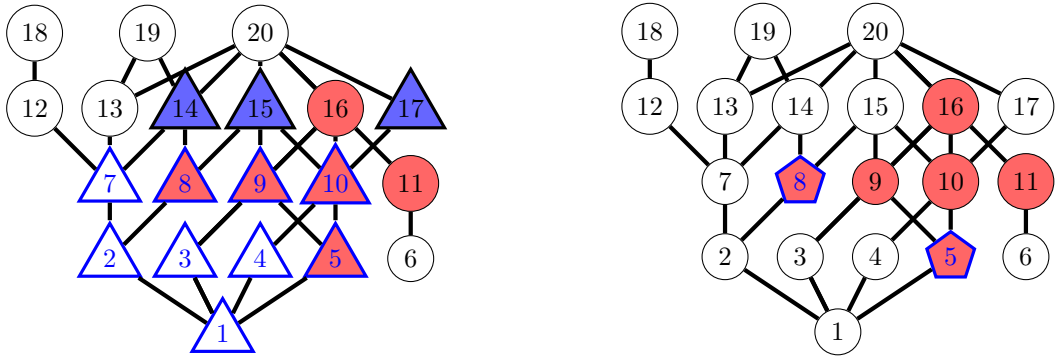

The process of rowmotion on $I$ for some arbitrary $I\in\IC(P)$ proceeds iteratively on the elements of $P$ starting with the maximal elements and moving down.  At the point where we are determining if an element $x\in P$ is toggled in or out, i.e. whether $x\in \row(I)$ or not, the decision has already been made for all $w> x$ in $P$. Thus, in asking whether $x\in \row(I)$ we can unambiguously reference whether or not $w\in \row(I)$ for all $w> x.$  This allows for the following lemma. 

\begin{lem}\label{Lem: togobs}
Given an interval-closed set $I\in\IC(P)$, consider an element $x\in P-I$. Rowmotion toggles $x$ into the interval-closed set unless doing so results in a subset of $P$ that is not an interval-closed set, i.e. unless 
\begin{itemize}
\item[(1)] there exists $y\in I$ and $z\in P-I$ such that $y< z < x$, or
\item[(2)] there exists $y,z\in P$ such that $x< y < z$ with $y\not\in\row(I)$ and $z\in\row(I).$
\end{itemize}
Similarly, for $x\in I$, rowmotion toggles $x$ out of the interval-closed set unless
\begin{itemize}
\item[(3)] there exist $y,z\in P$ such that $y< x< z$, $y\in I$ and $z\in \row(I)$.
\end{itemize}
\end{lem}

Lemmas \ref{Lemma: RowMaxMin}, \ref{Lemma: RowCeil} and \ref{Lemma: RowInc} make use of Lemma \ref{Lem: togobs} to show certain classes of elements are always swept into or out of $\row(I).$  This approach is extended in Theorem \ref{thm:AltRow} to show that for interval-closed sets, the element-toggle definition of rowmotion is equivalent to a global definition via set operations.

\begin{lem}\label{Lemma: RowMaxMin}
Given $I\in\IC(P)$, and a maximal (resp.\ minimal) element $x$ of $P$, if $x\in I$ then $x\notin\row(I)$. 
\end{lem}

\begin{proof}
Let $x$ be in $I$ for some interval-closed set $I$ in $\IC(P)$.  Suppose $x$ is maximal in $P$.  By definition of maximal, there does not exist an element $z$ such that $x<z.$ Thus $x$ is not in scenario (3) of Lemma \ref{Lem: togobs}, and $x$ is toggled out under rowmotion, i.e. $x\not\in\row(I)$.

Suppose instead is a minimal element of $P$.  Then by definition of minimal there does not exist an element $y<x$ and again $x$ is not in scenario (3).  Thus $x\not\in\row(I).$ 
\end{proof}

\begin{lem}\label{Lemma: RowCeil}
Given $I\in\IC(P)$, if $x\in\f(I)-I$ (i.e.\ $x$ is comparable to $I$ from above), then $x\in \row(I)$ if and only if $x\in \ceil{I}$. 
\end{lem}

\begin{proof}
Let $I\in\IC(P)$. Recall from Definition~\ref{defn: ceiling} that $\ceil{I}=\Min(\f(I)-I)$. Consider an element $x\in \f(I)-I-\ceil{I}.$  By definition of the ceiling and order filter of $I$, there must exist $z\in\ceil{I}\subseteq P-I$ and $y\in I$ such that $y< z< x.$ Thus, $x$ is in scenario (1) of Lemma \ref{Lem: togobs} and therefore $x\not\in\row(I).$

Suppose $x\in\ceil{I}.$  There does not exist a $y\in I$ and a $z\in P-I$ such that $y< z< x$, as this would imply $z\in\f(I)-I$ with $z< x$, but by definition of ceiling, $x$ is a minimal element in $\f(I)-I.$  Thus $x$ is not in scenario (1).  Also by definition of ceiling, $x\not\in I$ and for all $z> x, z\in\f(I)-I-\ceil{I}$.  Thus, as we just showed, $z\not\in\row(I),$  and $x$ cannot be in scenario (2). Therefore, for all $x\in\ceil{I}$, $x\in\row(I).$  
\end{proof}

\begin{lem}\label{Lemma: RowInc}
Given $I\in\IC(P)$, if an element $x\in P$ is incomparable to $I$ then $x\in \row(I).$
\end{lem}

\begin{proof}
 Let $x$ be an element of $P$ incomparable to $I$, i.e.\ $x\in\inc(I)$.  Thus for all $y< x$, $y\not\in I$ and $x$ can not be in scenario (1) of Lemma \ref{Lem: togobs}.  To show $x$ is not in scenario (2) we proceed by way of contradiction.  Suppose there exist $y,z\in P$ such that $x< y< z$ with $y\not\in\row(I)$ and $z\in\row(I)$.  
Since $x$ is incomparable to $I$, it follows that $y,z\not\in I.$  If $y$ were in scenario (1), that would imply $z$ is too, which can't be the case since $z$ was toggled in during the rowmotion process.  Thus, since $y\not\in\row(I)$ it must be in scenario (2), i.e. there must be elements $y_1,z_1\in P$ such that $y< y_1< z_1$, $y_1\not\in\row(I)$, and $z_1\in\row(I)$.  
As $x< y$, the same reasoning applied above to $y$ and $z$ applies to $y_1, z_1$, and thus there must be elements $y_2,z_2$ such that $y_1< y_2< z_2$, $y_2\not\in\row(I)$, and $z_2\in\row(I)$, and so on.  
Thus the existence of even one such $y,z$ pair with $x< y< z$ and $y\not\in\row(I)$ but $z\in\row(I)$ implies the existence of infinitely many.  Thus, since $P$ is finite (by our assumption of all posets in this paper), there does not exist any such $y,z$ pair and $x$ is not in scenario (2).  Therefore, for all $x\in\inc(I)$, $x\in\row(I).$ 
\end{proof}

The following theorem is the main result of Section~\ref{sec:togglingICS}.  One may wish to consult Example~\ref{Ex:Sec2_ThmAltRow} before and while reading the proof. Theorem~\ref{alt_def_ordinal} gives a simplification of Theorem~\ref{thm:AltRow} in the case that the poset is an ordinal sum of antichains.
See Remark~\ref{remark:comp_oi} for discussion of this result as compared to the analogous result for  order ideal rowmotion. 

\begin{thm}\label{thm:AltRow}
Given an interval-closed set $I\in \IC(P)$, rowmotion on $I$ is given by

\begin{equation}
\label{Eqn:arow formula}
\row(I)=\arow.
\end{equation}
\end{thm}

\begin{proof}

By Lemma \ref{Lemma: RowInc}, $\inc(I)\subseteq \row(I)$.  In fact, by Corollary \ref{cor:empty ICS}, when $I=\emptyset$ then $\row(I)=P=\inc(I)$. 
All that remains to show is that for elements of $P$ comparable to $I$, i.e.\ $x\in\comp(I)$, $x\in \row(I)$ if and only if $x\in\arowcomp$.   We consider separately the set of elements comparable to $I$ from above the ceiling (Case 1), the set of elements comparable to $I$ at or below the ceiling (Case 2), and the set of elements which are in or below $I$ but not below any ceiling element (Case 3). These cases are illustrated in Example \ref{Ex:Sec2_ThmAltRow}.

\medskip
\noindent Case 1: The set of elements comparable to $I$ from above the ceiling, i.e. $x\in\f(I)-I-\ceil{I}$
\smallskip

By Lemma \ref{Lemma: RowCeil}, for all $x\in\f(I)-I-\ceil{I},$ $x$ is not in $\row(I).$  We will show $x$ is similarly not an element of $\arowcomp$. 

As $x\in\f(I)$ but not in $I$, it follows that $y< x$ for some $y\in I$ and that for all $t\in I$, $x\not< t$.  Otherwise, we would have $i,t\in I$ such that $i< x< t$ which violates the definition of an interval-closed set. Thus $x\not\in \oi(I)$.  Since $\oi\A-(I\cup\oi\ceil{I})$ is a subset of $\oi(I)$, it follows that $x\not\in\oi\A-(I\cup\oi\ceil{I}).$

Similarly, as the ceiling elements are those in the set $\Min(\f(I)-I),$ it follows that for all $x\in\f(I)-I-\ceil{I}$, $x\not< t$ for all $t\in\ceil{I}.$  Thus $x\not\in\oi\ceil{I}$ and therefore $x\not\in \oi\ceil{I}-\oi(\Min(I)\cap\oi\ceil{I}).$
Thus, for all $x\in\f(I)-I-\ceil{I},$ $x\notin\row(I)$ and $x\not\in\arowcomp.$

\medskip
\noindent Case 2: The set of elements comparable to $I$ at or below the ceiling, i.e. $x\in\oi\ceil{I}\cap\comp(I)$
\smallskip

We will show in this case that $x\in\row(I)$ unless $x\in \oi\big(\F\big).$  
First consider $x\in\ceil{I}$. By Lemma \ref{Lemma: RowCeil} $x\in \row(I).$  By definition of ceiling, for all $z\geq x$, $z\not\in I$.  Thus $x\not\in \oi\Min(I)$ and therefore not in $\oi\big(\Min(I)\cap\oi\ceil{I}\big).$

Consider $x\in\oi\ceil{I}\cap\comp(I)-\ceil{I}.$ 
As the ceiling is the set of minimal elements in $\f(I)-I$, it follows that $x\not\in\f(I)-I$.  As $x$ is comparable to $I$, it must be in $\oi(I),$ and therefore either in the set $x\in\oi\ceil{I}\cap I-\Min(I)$ or the set $x\in\oi\big(\F\big).$

If $x\in\oi\ceil{I}\cap I-\Min(I),$ then by definition of order ideal there exists an element $z\in\ceil{I}\subseteq\row(I)$ such that $x< z$.  As $x$ is not a minimal element of $I$, there also exists an element $y\in I$ such that $y< x$ and we are in scenario (3) of Lemma \ref{Lem: togobs}, implying $x$ does not get toggled out of the interval-closed set and therefore $x\in\row(I).$ 

If $x\in\oi\big(\F\big),$ then by definition of order ideal, there exists an element $z\in\ceil{I}$ such that $x< z$ and $z\in\row(I)$.  If $x\in\Min(I)$ then $x\in I$ but for all $y< x,$ $y\not\in I$.  Thus, scenario (3) is avoided and $x$ gets toggled out.  If $x\in\oi\big(\F\big)-\Min(I)$ then $x\not\in I$ but there exists an element $y\in\Min(I)$ such that $x< y< z$ with $y\not\in\row(I)$ and $z\in\row(I)$, thus we are in scenario (2) of Lemma \ref{Lem: togobs} and $x$ does not toggle in.  In either case, $x\not\in\row(I).$

Therefore, for all $x\in\oi\ceil{I}\cap\comp(I)$, $x\in\row(I)$ if and only if $x \notin \oi\big(\F\big),$ i.e.\ $x\in\oi\ceil{I}-\oi\big(\F\big).$ 

\medskip
\noindent Case 3: The set of elements which are in or below $I$ but not below any ceiling element, i.e. $x\in\oi\A-\oi\ceil{I}$

\smallskip For this case we will consider two subcases, $x$ in $I$ and $x$ not in $I$. 
Let $x$ be an arbitrary element of $\Big(\oi\A-\oi\ceil{I}\Big)\cap I$, i.e.\ $x\in\A.$  It follows by the definition of ceiling and order filter that $\f(\{x\})\subseteq \f(I)-\ceil{I}$.  We proceed by induction, considering first a maximal element $m$ of $\f(\{x\}).$ By definition of order filter, since $m$ is maximal in $\f(\{x\})$ there is no element $z\in P$ greater than $m$.  Thus, if $m\in I$, it is not in scenario (3) of Lemma \ref{Lem: togobs}, and therefore $m\not\in\row(I)$.  If $m\not\in I$, then $m$ is in $\f(I)-I-\ceil{I},$ and  Lemma \ref{Lemma: RowCeil} implies $m\not\in\row(I)$. Now suppose by way of induction that for some $t\in \f(\{x\})$ it is the case that for all $z\in\f(\{x\})$ such that $z>t$, $z\not\in\row(I)$.  As before, there are two possible cases: $t\in\f(I)-I-\ceil{I},$ or $t\in I$.  If $t\in\f(I)-I-\ceil{I},$ then again by Lemma \ref{Lemma: RowCeil}, $t\not\in\row(I)$.  We assumed that for all $z\in\f(\{x\})$ such that $z>t,$ $z\not\in\row(I).$  However, by definition of order filter, the set of $z\in\f(\{x\})$ such that $z>t$ is the same as the set of all $z\in P$ such that $z>t.$  Thus, if $t\in I$, it is not in scenario (3) of Lemma \ref{Lem: togobs}, and therefore $t\not\in\row(I)$. 
Either way, $t\not\in\row(I)$.  Taking any linear extension of $P$, it follows by induction on the rank of elements in $\f(\{x\})$ that for all $t\in\f(\{x\})$, $t\not\in\row(I)$, and thus $x\not\in\row(I)$.

Now let $x$ be an arbitrary element of $\Big(\oi\A-\oi\ceil{I}\Big)-I.$  It follows that  $x$ is not in $I$, but by definition of $\oi\A$, there is some $z\in\A$ such that $x< z.$ Thus, for all $y< x$, $y\not\in I$ or we would contradict our assumption that $I$ is an interval-closed set.  Thus $x$ is not in scenario (1) of Lemma \ref{Lem: togobs}.    

Suppose by way of contradiction that $x$ is in scenario (2) of Lemma \ref{Lem: togobs}, i.e.\ that there are elements $y,z$ such that $x< y< z$, $y\not\in\row(I)$, and $z\in\row(I).$  If such $y$ and $z$ exist, then there must be a pair of elements $y_1$ and $z_1$ such that $x< y_1< z_1$,  $y_1\not\in\row(I)$, $z_1\in\row(I)$, and $y_1\lessdot z_1$.  If $y_1$ is in $I$, it must be in $\A$ as otherwise $x< y_1$ would contradiction our assumption that $x$ is incomparable to $\ceil{I}$.  Similarly, if $z_1\in I$ it must be in $\A.$  However, since we just showed at the start of Case 3, that for all $w\in \A$, $w$ and all elements of $\f(\{w\})$ are not in $\row(I)$, and as $z_1\in\row(I)$ and $z_1\in \f(\{y_1\})$ it follows that both $z_1$ and $y_1$ are not in $\A$ and thus not in $I$.  Thus we have elements $y_1< z_1$ in $P-I$ such that $z_1$ was toggled in during rowmotion but $y_1$ was not.  As $z_1$ could be toggled in, it follows that for all $t< y_1< z_1,$  $t\not\in I.$  Otherwise $z_1$ would be in scenario (1) of Lemma \ref{Lem: togobs} and would not have been toggled in.  This implies $y_1$ is not in scenario (1) itself, and therefore must be in scenario (2) of Lemma \ref{Lem: togobs}.  This implies the existence of a pair of elements $y_2$ and $z_2$, with such that $y_1< y_2< z_2$, $y_2\not\in\row(I),$ $z_2\in\row(I),$ and $z_2$ covers $y_2.$  The covering relationships and assumptions about inclusion in $\row(I)$ imply that $y_1, y_2, z_1,$ and $z_2$ are all distinct elements.  Moreover, as $x< y_2< z_2$, the same argument applied to $y_1$ and $z_1$ applies to $y_2$ and $z_2$, implying the existence of another pair of distinct elements $y_3$ and $z_3$.  Thus, the existence of one pair of elements $y, z$ such that $x< y < z$, $y\not\in\row(I)$, and $z\in\row(I)$ implies the existence of infinitely many covering pairs $y_i,z_i$.  Since $P$ is finite, this leads to a contradiction and thus $x$ is not in scenario (2) of Lemma \ref{Lem: togobs}.  Therefore, for all $x\in\oi\A-\oi\ceil{I}-I$, $x\in\row(I).$

Thus, for all $x\in \oi\A-\oi\ceil{I}$, $x\in\row(I)$ if an only if $x\not\in I,$ or equivalently if $x\in\oi\A-(I\cup\oi\ceil{I}).$ 

\smallskip
Having covered all possible cases, it follows that $x\in\row(I)$ if and only if $x\in\arow.$
\end{proof}

\begin{remark}\label{remark:comp_oi}
This result should be thought of as analogous to the global definition of rowmotion on order ideals \cite{CF1995,SW2012}. Using the notation of this paper, if $I$ is an order ideal of $P$, the result of applying order ideal rowmotion would be $\oi(\ceil{I})$.
The comparative level of complexity illustrates that though the definition of interval-closed sets seems to be a mild generalization of order ideals, the dynamics is, in many cases, much more complicated. 
\end{remark}

\begin{ex}\label{Ex:Sec2_ThmAltRow}
We again turn to our example interval-closed set $I=\Iex$ to illustrate the cases and conclusions from the proof of Theorem \ref{thm:AltRow}.  Figure \ref{fig:thm_cases} shows how the cases considered in Theorem~\ref{thm:AltRow}, along with the elements incomparable to $I$, partition the poset. 

Recall that $\row(I)=\{3,4,6,7,9,10,12,13,14,15,17,18\}.$ Examining Figure \ref{fig:thm_cases}, we observe that for all $x\in \inc(I)=\{7, 12, 13, 18\}$, $x$ is in $\row(I)$, as shown in Lemma \ref{Lemma: RowInc}.  Furthermore, the elements comparable to $I$ from above the ceiling  (i.e. $x\in\f(I)-I-\ceil{I}=\{19,20\},$ considered in Case 1) are comparable to $I$ but strictly above both $I$ and the ceiling $\ceil{I}$.  Thus these elements are not in $\arow$, but also not in $\row(I)$.  

Recall that for this particular interval-closed set, the minimum elements of $I$ under the ceiling, $\Min(I)\cap\oi\ceil{I}$, are $5$ and $8$ (as shown in Figure \ref{fig:Sec2_Ceil_MinUnder}).  Thus $\oi(\Min(I)\cap\oi\ceil{I})=\{1, 2, 5, 8\}.$  In Figure \ref{fig:thm_cases}, we see that the elements comparable to $I$ at or below the ceiling (i.e.\ $x\in\oi\ceil{I}\cap\comp(I)$, considered in Case 2 and indicated with circular nodes in Figure \ref{fig:thm_cases}) are in $\row(I)$ if an only if they are not in $\oi(\Min(I)\cap\oi\ceil{I}).$

Finally, for those elements in or below $I$ but not below any ceiling element (i.e.\ $x\in \oi\inc_I(\ceil{I})-\oi\ceil{I}=\{6, 11, 16\}$, considered in Case 3), we observe that $x\in\row(I)$ if and only if $x\not\in I$.
\end{ex}

\begin{figure}[htbp]\centering
	\begin{minipage}{.36\linewidth}\centering
		\begin{tikzpicture}[scale = .5]
		\node[draw,circle](A) at (5,0) {1};
		\node[draw,circle](B) at (2,2) {2};
		\node[draw,circle, fill=cyan](C) at (4,2) {3};
		\node[draw,circle, fill=cyan](D) at (6,2) {4};
		\node[draw,circle,fill=red!60](E) at (8,2) {5};
		\node[draw,diamond, fill=cyan, inner sep=.8mm](F) at (2,4) {7};
		\node[draw,circle,fill=red!60](G) at (4,4) {8};
    \node[draw,circle,left color=red!70, right color=cyan](H) at (6,4) {9};
    \node[draw,circle,left color=red!70, right color=cyan](I) at (8,4) {10};
    \node[draw,rectangle, minimum size=6.7mm,fill=red!60](J) at (10,4) {11};
		\node[draw,rectangle, minimum size=6.7mm,fill=cyan](T) at (10,2) {6};
    \node[draw,diamond, fill=cyan, inner sep=.5mm](K) at (0,6) {12};
    \node[draw,diamond, fill=cyan, inner sep=.5mm](L) at (2,6) {13};
    \node[draw,circle, fill=cyan](M) at (4,6) {14};
    \node[draw,circle, fill=cyan](N) at (6,6) {15};
    \node[draw,rectangle, minimum size=6.7mm,fill=red!60](O) at (8,6) {16};
    \node[draw,circle, fill=cyan](P) at (10,6) {17};
    \node[draw,diamond, fill=cyan, inner sep=.5mm](Q) at (0,8) {18};
    \node[regular polygon, draw, regular polygon sides=5, inner sep=.6mm](R) at (3,8) {19};
    \node[regular polygon, draw, regular polygon sides=5, inner sep=.6mm](S) at (6,8) {20};
		\draw[ultra thick] (A) --(B) --(F) --(K) --(Q);
		\draw[ultra thick] (F) --(L) --(R);
		\draw[ultra thick] (F) --(M) --(S);
		\draw[ultra thick] (B) --(G) --(M) --(R);
		\draw[ultra thick] (G) --(N) --(S);
    \draw[ultra thick] (A) --(C) --(H) --(O) --(S);
    \draw[ultra thick] (A) --(D) --(I) --(P) --(S);
    \draw[ultra thick] (A) --(E) --(I) --(O);
    \draw[ultra thick] (E) --(H) --(N);
    \draw[ultra thick] (A) --(C) --(H) --(O) --(S);
    \draw[ultra thick] (I) --(N);
    \draw[ultra thick] (L) --(S);
    \draw[ultra thick] (J) --(O);
    \draw[ultra thick] (J) --(T);
\end{tikzpicture}
\end{minipage}
	\caption{The interval-closed set $I$ is shaded in red, while elements of $\row(I)$ are shaded in blue.  Elements incomparable to $I$, $\inc(I)$, are indicated with diamond nodes; elements comparable to $I$ from above the ceiling (Case 1: $\f(I)-I-\ceil{I}$), are indicated with pentagonal nodes; elements comparable to $I$ at or below the ceiling (Case 2: $\oi\ceil{I}\cap\comp(I)$), are indicated with circular nodes; and elements which are in or below $I$ but not below any ceiling element (Case 3: $\oi\inc_I(\ceil{I})-\oi\ceil{I}$), are indicated with square nodes.} \label{fig:thm_cases}
\end{figure}
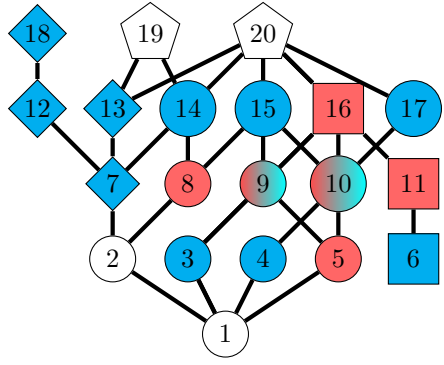

Before studying rowmotion on interval-closed sets of specific posets in Sections~\ref{sec:ordinalsum} and \ref{sec:products_of_chains}, we prove in the next subsection a homomesy result that holds for all posets.

\subsection{Toggleability homomesy}
We define an analogue of the toggleability statistic of  \cite[Def 6.1]{Striker2015} for interval-closed sets and prove it is homomesic for extremal elements.

\begin{definition}
Fix a finite poset $P$. For each $x\in P$, define the \emph{toggleability statistic} $\tog_x: \IC(P)\rightarrow \{-1, 0, 1\}$  as follows:
\[\mathfrak{T}_x(I) =
\begin{cases}
1 &\text{if } x \text{ may be toggled in to } I, \\
-1 &\text{if } x \text{ may be toggled out of } I, \\
0 &\text{otherwise}.
\end{cases}\]
\end{definition}

\begin{definition}[\cite{PR2015}]\label{def:homomesy}
We say that a statistic exhibits \emph{homomesy} under some action when every orbit of that action has the same average when the statistic is calculated over the orbit. That is, 
\[\frac{1}{|\mathcal{O}|}\sum_{I \in \mathcal{O}} \textrm{stat}(I) =c\]
for all orbits $\mathcal{O}$. In this case, we say that the statistic is $c$-mesic.
\end{definition}

\begin{ex}
Let $P$ be the diamond poset (the product of chains $[2]\times [2]$) as in Figure \ref{fig:homomesy_toggleability_max}. 
In this figure, each interval-closed set is listed, and the toggleability statistic of each element appears as the label of the element. The interval-closed sets are grouped by their rowmotion orbits. We highlight in the figure the toggleability statistic for the unique maximal element. Note that the average value of this statistic is $0$ on each orbit, thus this is an instance of homomesy. We show in Proposition~\ref{toggleabilitly_max_min} that this statistic is $0$-mesic with respect to rowmotion on any poset.
\end{ex}

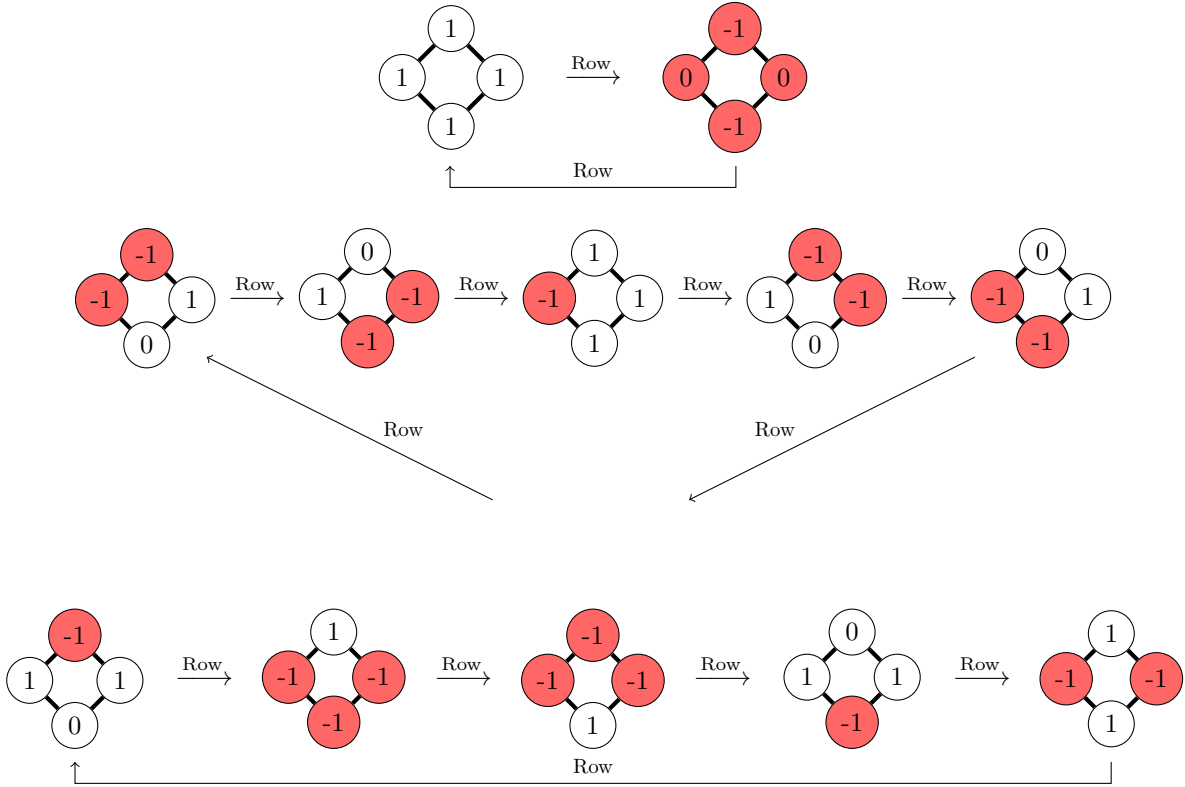
\begin{figure}[hbt]
\centering
	\begin{minipage}{.14\linewidth}\centering
		\begin{tikzpicture}[scale = .65]
		\node[draw,circle](A) at (1,0) {1};
		\node[draw,circle](B) at (0,1) {1};
		\node[draw,circle](C) at (2,1) {1};
		\node[draw,circle](D) at (1,2) {1};
		\draw[ultra thick] (A) --(B) --(D);
		\draw[ultra thick] (A) --(C) --(D);
	\end{tikzpicture}
	\end{minipage}\quad $\xlongrightarrow{\text{Row}}$\quad
	\begin{minipage}{.14\linewidth}\centering
		\begin{tikzpicture}[scale = .65]
		\node[draw,circle,fill=red!60](A) at (1,0) {-1};
		\node[draw,circle,fill=red!60](B) at (0,1) {0};
		\node[draw,circle,fill=red!60](C) at (2,1) {0};
		\node[draw,circle,fill=red!60](D) at (1,2) {-1};
		\draw[ultra thick] (A) --(B) --(D);
		\draw[ultra thick] (A) --(C) --(D);
		\end{tikzpicture}
	\end{minipage}
  \begin{minipage}{1\linewidth}\centering
\begin{tikzpicture}[scale=.95]
\node[anchor=center, scale=0.78] at (2.25,.25) {Row};
\draw[->] (4.25,.3)--(4.25,0)--(0.25,0)--(0.25,.3);
\end{tikzpicture}
\end{minipage}\\
\vspace{2.5mm}
\vspace{2.5mm}

\begin{minipage}{.12\linewidth}
	\centering
	\begin{tikzpicture}[scale = .6]
		\node[draw,circle](A) at (1,0) {0};
		\node[draw,circle,fill=red!60](B) at (0,1) {-1};
		\node[draw,circle](C) at (2,1) {1};
		\node[draw,circle,fill=red!60](D) at (1,2) {-1};
		\draw[ultra thick] (A) --(B) --(D);
		\draw[ultra thick] (A) --(C) --(D);
	\end{tikzpicture}
\end{minipage}\ $\xlongrightarrow{\text{Row}}$\ 
\begin{minipage}{.12\linewidth} \centering
\begin{tikzpicture}[scale = .6]
	\node[draw,circle,fill=red!60](A) at (1,0) {-1};
	\node[draw,circle](B) at (0,1) {1};
	\node[draw,circle,fill=red!60](C) at (2,1) {-1};
	\node[draw,circle](D) at (1,2) {0};
	\draw[ultra thick] (A) --(B) --(D);
	\draw[ultra thick] (A) --(C) --(D);
\end{tikzpicture}
\end{minipage}\ $\xlongrightarrow{\text{Row}}$\ 
\begin{minipage}{.12\linewidth} \centering
\begin{tikzpicture}[scale = .6]
	\node[draw,circle](A) at (1,0) {1};
	\node[draw,circle,fill=red!60](B) at (0,1) {-1};
	\node[draw,circle](C) at (2,1) {1};
	\node[draw,circle](D) at (1,2) {1};
	\draw[ultra thick] (A) --(B) --(D);
	\draw[ultra thick] (A) --(C) --(D);
\end{tikzpicture}
\end{minipage}\ $\xlongrightarrow{\text{Row}}$\ 
\begin{minipage}{.12\linewidth} \centering
\begin{tikzpicture}[scale = .6]
	\node[draw,circle](A) at (1,0) {0};
	\node[draw,circle](B) at (0,1) {1};
	\node[draw,circle,fill=red!60](C) at (2,1) {-1};
	\node[draw,circle,fill=red!60](D) at (1,2) {-1};
	\draw[ultra thick] (A) --(B) --(D);
	\draw[ultra thick] (A) --(C) --(D);
\end{tikzpicture}
\end{minipage}\ $\xlongrightarrow{\text{Row}}$\ 
\begin{minipage}{.12\linewidth} \centering
\begin{tikzpicture}[scale = .6]
	\node[draw,circle,fill=red!60](A) at (1,0) {-1};
	\node[draw,circle,fill=red!60](B) at (0,1) {-1};
	\node[draw,circle](C) at (2,1) {1};
	\node[draw,circle](D) at (1,2) {0};
	\draw[ultra thick] (A) --(B) --(D);
	\draw[ultra thick] (A) --(C) --(D);
\end{tikzpicture}
\end{minipage} 
\\
\vspace{2mm}

\begin{minipage}{1\linewidth}\centering
\begin{tikzpicture}[scale=.95]
\node[anchor=center, scale=0.78] at (3.2,1) {Row};
\node[anchor=center, scale=0.78] at (-2,1) {Row};
\draw[->] (6,2)--(2,0);
\draw[<-] (-4.75,2)--(-.75,0);
\hspace{-.375cm}
\begin{minipage}{.12\linewidth} \centering
\begin{tikzpicture}[scale = .6]
	\node[draw,circle](A) at (1,0) {1};
	\node[draw,circle](B) at (0,1) {1};
	\node[draw,circle,fill=red!60](C) at (2,1) {-1};
	\node[draw,circle](D) at (1,2) {1};
	\draw[ultra thick] (A) --(B) --(D);
	\draw[ultra thick] (A) --(C) --(D);
\end{tikzpicture}
\end{minipage}
\end{tikzpicture}
\end{minipage}
\\
\vspace{5.5mm}
\vspace{5.5mm}

\begin{minipage}{.12\linewidth} \centering
\begin{tikzpicture}[scale = .6]
	\node[draw,circle](A) at (1,0) {0};
	\node[draw,circle](B) at (0,1) {1};
	\node[draw,circle](C) at (2,1) {1};
	\node[draw,circle,fill=red!60](D) at (1,2) {-1};
	\draw[ultra thick] (A) --(B) --(D);
	\draw[ultra thick] (A) --(C) --(D);
\end{tikzpicture}
\end{minipage}\quad $\xlongrightarrow{\text{Row}}$\quad
\begin{minipage}{.12\linewidth} \centering
\begin{tikzpicture}[scale = .6]
	\node[draw,circle,fill=red!60](A) at (1,0) {-1};
	\node[draw,circle,fill=red!60](B) at (0,1) {-1};
	\node[draw,circle,fill=red!60](C) at (2,1) {-1};
	\node[draw,circle](D) at (1,2) {1};
	\draw[ultra thick] (A) --(B) --(D);
	\draw[ultra thick] (A) --(C) --(D);
\end{tikzpicture}
\end{minipage}\quad $\xlongrightarrow{\text{Row}}$\quad
\begin{minipage}{.12\linewidth} \centering
\begin{tikzpicture}[scale = .6]
	\node[draw,circle](A) at (1,0) {1};
	\node[draw,circle,fill=red!60](B) at (0,1) {-1};
	\node[draw,circle,fill=red!60](C) at (2,1) {-1};
	\node[draw,circle,fill=red!60](D) at (1,2) {-1};
	\draw[ultra thick] (A) --(B) --(D);
	\draw[ultra thick] (A) --(C) --(D);
\end{tikzpicture}
\end{minipage}\quad $\xlongrightarrow{\text{Row}}$\quad
\begin{minipage}{.12\linewidth} \centering
\begin{tikzpicture}[scale = .6]
	\node[draw,circle,fill=red!60](A) at (1,0) {-1};
	\node[draw,circle](B) at (0,1) {1};
	\node[draw,circle](C) at (2,1) {1};
	\node[draw,circle](D) at (1,2) {0};
	\draw[ultra thick] (A) --(B) --(D);
	\draw[ultra thick] (A) --(C) --(D);
\end{tikzpicture}
\end{minipage}\quad $\xlongrightarrow{\text{Row}}$\quad
\begin{minipage}{.12\linewidth} \centering
\begin{tikzpicture}[scale = .6]
	\node[draw,circle](A) at (1,0) {1};
	\node[draw,circle,fill=red!60](B) at (0,1) {-1};
	\node[draw,circle,fill=red!60](C) at (2,1) {-1};
	\node[draw,circle](D) at (1,2) {1};
	\draw[ultra thick] (A) --(B) --(D);
	\draw[ultra thick] (A) --(C) --(D);
\end{tikzpicture}
\end{minipage}
\begin{minipage}{1\linewidth}\centering
\begin{tikzpicture}[scale=.95]
\node[anchor=center, scale=0.78] at (.75,.25) {Row};
\draw[->] (8,.3)--(8,0)--(-6.5,0)--(-6.5,.3);
\end{tikzpicture}
\end{minipage}\\
	\caption{Orbits of the diamond poset under rowmotion. The items in each interval-closed set are in red, and the label of each item is the value of the toggleability statistic for this item. Adding the values of the toggleability of the maximal (resp. minimal) element for all the interval-closed sets of a given orbit always gives $0$, showing that the statistic is $0$-mesic. The same is not true for the elements that are neither minima nor maxima.} \label{fig:homomesy_toggleability_max}
\end{figure}

\begin{prop}\label{toggleabilitly_max_min}
For any poset $P$, the toggleability statistic of any maximal (resp.\ minimal) element  is $0$-mesic under rowmotion on $\IC(P)$. 
\end{prop}

\begin{proof} 
Let $x$ be a maximal element of a poset $P$.  Consider the toggleability of $x$ within an orbit of some interval-closed set $I$.  By Lemma \ref{Lemma: RowMaxMin}, the maximal element of $P$ is always toggled \emph{out} under rowmotion.  Thus, if $x$ is an element of $\row^j(I)$, then $x$ will be toggled out and $\tog_x(\row^j(I))=-1.$  Moreover, as the maximal element is always toggled out under rowmotion, if $x\in\row^j(I)$, it could not have been an element of $\row^{j-1}(I)$ and therefore must have been toggled in, implying $\tog_x(\row^{j-1}(I))=+1.$  Thus, within an orbit, every $-1$ value of the toggleability statistic for a maximal element of $P$ is immediately preceded by at $+1$ value.  

Since rowmotion is the composition of toggles from top to bottom, a maximal element $x$ of $P$ will be toggled in every time the toggleability statistic equals $+1$.  And as just argued, every time a maximal value is toggled in, it is immediately toggled out under the next application of rowmotion.  Thus every $+1$ value of the toggleability statistic for a maximal element is immediately followed by a $-1$ value.  

Having shown every $+1$ value is followed by a $-1$ value in the following element of the orbit, and every $-1$ value is preceded by at $+1$ value, it follows that these values cancel out over the orbit and the toggleability statistic for a maximal element of $P$ is $0$-mesic under rowmotion.

The previous argument applies to a maximal element of the dual poset $P^*$, and since a minimal element of $P$ is a maximal element of $P^*$, and $\row_{P^*}=\row^{-1}_P$ by Proposition \ref{Prop: DualInverse}, it follows that the toggleability statistic for the minimal element of $P$ under rowmotion is also $0$-mesic.  In this case, each $-1$ vaule of the toggleability statistic is immediately followed by a $+1$ value within the orbit.
\end{proof}

\begin{remark}\label{rmk:toggleability_in_general}
Note that, unlike what happens with order ideals~\cite[Thm 6.7]{Striker2015}, the toggleability statistic by non-extremal elements is not, in general, homomesic. For example, the toggleability statistic for any element in the middle rank of the diamond poset has average $\frac{1}{2}$ on the first orbit of Figure \ref{fig:homomesy_toggleability_max} and $0$ on the second orbit.
\end{remark}

See Sections~\ref{sec:ord_sum_homomesy} and \ref{sec:prod_chains_homomesy} for more homomesy results on interval-closed sets.

\section{Rowmotion on interval-closed sets of ordinal sums of antichains}\label{sec:ordinalsum}
In this section, we examine posets that are ordinal sums of antichains. First, we classify the interval-closed sets of these posets and fully describe their orbits under rowmotion. Then, we highlight specific examples, applying our results to those specific cases. We end the section with results on homomesy.

\subsection{Chain posets}\label{ssec:chains}

Recall that $[n]$ denotes a chain poset of $n$ elements, or equivalently the ordinal sum of $n$ antichains $\mathbf{1}\oplus\mathbf{1}\oplus\cdots\oplus\mathbf{1}$.

\begin{prop}\label{thm:chains_ICS}
The cardinality of $\IC([n])$ is $\binom{n}{2}+n+1$.
\end{prop}

\begin{proof}
By Proposition \ref{prop:ICS_alt_def}, the interval-closed sets of $[n]$ are the empty set or intervals of the form $[x_i, x_j]$ with $x_i \leq x_j.$ Thus there are $\binom{n}{2}$ interval-closed sets with more than a single element, $n$ interval-closed sets consisting of a single element, and finally the empty set.
\end{proof}

\begin{thm}\label{thm:chains_orbits}
The set $\IC([n])$ with $n \geq 0$ has rowmotion order dividing $2(n+2)$ when $n$ is odd and $n+2$ when $n$ is even.
Moreover, its rowmotion orbit structure is described below:
\begin{itemize}
\item a single orbit of size $2$ corresponding to $\O(\emptyset)=\{\emptyset,[n]\}$, 
\item $\lfloor \frac{n-1}{2}\rfloor$ orbits of size $n + 2$. These orbits have representatives $[1,k]$ where $1\leq k<\frac{n}{2}$, and are of the form
\[
\O([1,k]) = \{ [1,k], [2,k+1], \ldots , [n-k+1, n], [1,n-k],\ldots [k+1,n]\},
\]
\item and when $n$ is even, a single orbit of size $\frac{n + 2}{2}$ of the form
\[
\O\left (\left [1,\frac{n}{2}\right ]\right)=\left \{\left [1,\frac{n}{2}\right ],\left [2,\frac{n}{2}+1\right ],\ldots,\left [\frac{n}{2}+1,n\right ]\right\}.
\]
\end{itemize}
\end{thm}

\begin{proof} 
Fix $n \geq 0$. Label the elements of the chain poset $[n]$ as $1,2, \dotsc, n$. 

By Corollary \ref{cor:empty ICS}, there is one rowmotion orbit of size 2, consisting of the empty set and the whole chain poset $[n]$. 

Let $I = [1, k]$ be a $k$-element order ideal with $k<\frac{n}{2}$. Then $\row(I) = [2, k + 1]$ is the $k$-element interval-closed set with maximal element of rank one greater than that of $I$. Successive applications of $\row$ continue to shift the $k$-elements up one rank until the interval-closed set contains the maximal element of the poset. That is, the interval-closed set is $[n - k+1, n]$. Call this set $J$. Moving from $I$ to $J$ contributes $n-k+1$ elements to the orbit.  By Lemma \ref{lem:max_elt_comp}, applying $\row$ to $J$ returns the complement $\overline{J} = [1, n - k]$, which is an $(n-k)$-element order ideal. The process repeats on this order ideal, contributing an additional $k+1$ elements to the orbit until we again reach an interval-closed set containing the maximal element of the poset. This set is $[k + 1, n]$, whose complement is $I$. Thus, the orbit has size $n+2.$ As $[1, k]$  and $[1, n - k]$ are contained in the same orbit, we note that there are $\lfloor\frac{n-1}{2}\rfloor$ choices for $k$ when selecting a representative of the form $[1, k]$ that will produce these orbits of size $n+2$.

When $n$ is even, consider $I = [1, \frac{n}{2}]$, the $\frac{n}{2}$-element order ideal.  The orbit of this interval-closed set under $\row$ is similar to the case above, except $J = [\frac{n}{2} + 1, n]$, so the complement $\overline{J} = \row(J)$ is the original interval-closed set $I$. In this case, the orbit has size $\frac{n}{2}+1.$

Using Theorem \ref{thm:chains_ICS}, we see that all $I\in\IC([n])$ are contained in one of the orbits described above. If $n$ is odd, the number of elements in these orbits is $2 + \frac{n - 1}{2}(n + 2) = \binom{n}{2} + n + 1$, and if $n$ is even, the number of elements in these orbits is $2 + \frac{n-2}{2}(n+2) + \frac{n+2}{2} = \binom{n}{2} + n + 1.$

Thus, the order of rowmotion on chains $[n]$ is $2$ when $n=1$ or $2$, $n+2$ when $n\geq 4$ is even, and $2(n+2)$ when $n\geq 3$ is odd. 
\end{proof}

\subsection{Ordinal sums of antichains}
In this subsection, we study a generalization of the chain poset in which all elements in each rank are comparable with all elements of every other rank. This is the ordinal sum of antichains $\mathbf{a}_1\oplus\mathbf{a}_2\oplus\cdots\oplus\mathbf{a}_n$, which includes $a_i$ elements in the $i^{th}$ antichain. Our main result is Theorem~\ref{thm:gen_ord_sum_row}, which completely describes the rowmotion orbits of this family of posets and shows that the order of rowmotion depends only on $n$, not the number of elements in each rank.

We begin by counting the interval-closed sets of $\mathbf{a}_1\oplus\mathbf{a}_2\oplus\cdots\oplus\mathbf{a}_n$.

\begin{thm}\label{thm:gen_ord_sum_ics_card}
The cardinality of $\IC(\mathbf{a}_1\oplus\mathbf{a}_2\oplus\cdots\oplus\mathbf{a}_n)$ is  $1+\sum_{1\leq i\leq n}(2^{a_i}-1)+\sum_{1\leq i<j\leq n}(2^{a_i}-1)(2^{a_j}-1)$. This consists of the following:
\begin{itemize}
\item the empty set $\emptyset$, 
\item subsets of size $k$ within a single antichain $\mathbf{a}_i$, 
\item and subsets consisting of some non-empty subset of size $k$ of $\mathbf{a}_i$, some non-empty subset of size $l$ of $\mathbf{a}_j$ with $i<j$, and all elements of $\mathbf{a}_r$ with $i<r<j$. 
\end{itemize}
\end{thm}

\begin{proof}
Consider $P=\mathbf{a}_1\oplus\mathbf{a}_2\oplus\cdots\oplus\mathbf{a}_n$. If $I\in\IC(P)$ contains any elements of $\mathbf{a}_i$ and any elements of $\mathbf{a}_j$ with $i<j$, then it must contain all elements of $\mathbf{a}_r$ for all $i<r<j$.  Thus elements of $\IC(P)$ are either the empty set, non-empty subsets of a single $\mathbf{a}_i$, of which there are $\sum_{1\leq i\leq n}(2^{a_i}-1)$, or non-empty subsets of two distinct $a_i$ and $a_j$ along with all elements of $\mathbf{a}_r$ for $i<r<j$, of which there are $\sum_{1\leq i<j\leq n}(2^{a_i}-1)(2^{a_j}-1)$.
\end{proof}

We have the following as a corollary of Lemma~\ref{lem:at_most_1}.
\begin{cor}
Consider $P = \mathbf{a}_1$ or $P = \mathbf{a}_1 \oplus \mathbf{a}_2$. Then, any subset of $P$ is an interval-closed set. Furthermore, $\row(I) = \overline{I}$ for any $I \in \IC(P)$, and rowmotion has order $2$.
\end{cor}

We now determine the rowmotion orbits of $\IC(\mathbf{a}_1\oplus\mathbf{a}_2\oplus\cdots\oplus\mathbf{a}_n)$ for $n\geq 3$. We  use the following definition and lemma.

\begin{definition}\label{def: ordsumICS}
  For $I\in\IC(\mathbf{a}_1\oplus\mathbf{a}_2\oplus\cdots\oplus\mathbf{a}_n)$, we denote $I$ as $\binom{\mathbf{a}_i}{k}$ if $I$ is equal to a non-empty proper subset of $\mathbf{a}_i$ of size $k$, and $[\binom{\mathbf{a}_i}{k},\binom{\mathbf{a}_j}{l}]$ if $I$ consists of some non-empty proper subset of size $k$ of $\mathbf{a}_i$, some non-empty proper subset of size $l$ of $\mathbf{a}_j$ with $i<j$, and all elements of $\mathbf{a}_r$ for all $i<r<j$. We use $\mathbf{a}_i$ if $I$ is a single antichain within the ordinal sum, $[\mathbf{a}_i,\binom{\mathbf{a}_j}{l}]$ if $I$ is a non-empty proper subset of some $\mathbf{a}_j$ and all elements of $\mathbf{a}_r$ for $i\leq r<j$, $[\binom{\mathbf{a}_i}{k},\mathbf{a}_j]$ if I is a non-empty proper subset of some $\mathbf{a}_i$ and all elements of $\mathbf{a}_r$ for $i<r\leq j$, and $[\mathbf{a}_i,\mathbf{a}_j]$ if $I$ consists of all elements $\mathbf{a}_r$ for $i\leq r\leq j$.  We use $\overline{\binom{\mathbf{a}_i}{k}}$ to denote the set complement $\mathbf{a}_i-\binom{\mathbf{a}_i}{k}$ within an antichain.  Thus  $\left [\overline{\binom{\mathbf{a}_i}{k}},\overline{\binom{\mathbf{a}_j}{l}}\right ]$ denotes the interval-closed set consisting of the elements $\mathbf{a}_i-\binom{\mathbf{a}_i}{k}$, $\mathbf{a}_j-\binom{\mathbf{a}_j}{l},$ and $\mathbf{a}_r$ for all $i<r<j.$ 
\end{definition}

\begin{lem}\label{lem:ints_commute_map}
Let $f$ be the map that takes $\mathbf{a}_i$ to the single element $i$ and $[\mathbf{a}_i,\mathbf{a}_j]$ to the interval $[i,j]$ in the chain poset $[n]$. Rowmotion commutes with this map, that is, $f(\row(I)) = \row(f(I))$.  
\end{lem}

\begin{proof}
Each element in the rowmotion orbit of $[\mathbf{a}_i,\mathbf{a}_j]$ where $i\leq j$ contains only completely full ranks forming an interval-closed set, or more precisely $\mathcal{O}\left([\mathbf{a}_i,\mathbf{a}_j]\right)=\{[\mathbf{a}_i,\mathbf{a}_j], [\mathbf{a}_{i+1}, \mathbf{a}_{j+1}],\ldots,[\mathbf{a}_{i+n-j},\mathbf{a}_n],\ldots[\mathbf{a}_{i-1},\mathbf{a}_{j-1}]\}.$ Therefore, we can either perform rowmotion on $[\mathbf{a}_i,\mathbf{a}_j]$, and then look at the related interval-closed set in $[n]$, or we can first map $[\mathbf{a}_i,\mathbf{a}_j]$ to $[i,j]$, and then perform rowmotion.
\end{proof}

\begin{thm}
\label{thm:gen_ord_sum_row} Let $n \geq 3$. 
The set $\IC(\mathbf{a}_1\oplus\mathbf{a}_2\oplus\cdots\oplus\mathbf{a}_n)$ has rowmotion order  $2n(n+2)$ when $n$ is odd and $n(n+2)/2$ when $n$ is even.
Moreover, a complete description of its rowmotion orbit structure is below:
\begin{itemize}
\item $1+\frac{1}{2}\sum_{1\leq i<j\leq n}(2^{a_i}-2)(2^{a_j}-2)$ orbits of size $2$, corresponding to $\{\emptyset, P\}$, and orbits with representatives $[\binom{\mathbf{a}_i}{k},\binom{\mathbf{a}_j}{l}]$, which are of the form,
\[
\O \left (  \left [\binom{\mathbf{a}_i}{k},\binom{\mathbf{a}_j}{l}\right ]\right )=\left\{ \left [\binom{\mathbf{a}_i}{k},\binom{\mathbf{a}_j}{l}\right ], \left [\overline{\binom{\mathbf{a}_i}{k}},\overline{\binom{\mathbf{a}_j}{l}}\right ]\right\},
\]
\item $\lfloor \frac{n-1}{2}\rfloor$ orbits of size $n+2$,  with representatives $[\mathbf{a}_1,\mathbf{a}_j]$ with $j < \frac{n }{2}$, and are of the form 
\[
\O([\mathbf{a}_1,\mathbf{a}_k]) = \{ [\mathbf{a}_1,\mathbf{a}_k], [\mathbf{a}_2,\mathbf{a}_{k+1}], \ldots , [\mathbf{a}_{n-k+1},\mathbf{a}_{n}], [\mathbf{a}_{1},\mathbf{a}_{n-k}],\ldots [\mathbf{a}_{k+1},\mathbf{a}_{n}]\},
\]
\item $1$ orbit of size $\frac{n+2}{2}$ when $n$ is even, with representative $[\mathbf{a}_1,\mathbf{a}_{n/2}]$, of the form \[
\O\left ( [\mathbf{a}_1,\mathbf{a}_{n/2}]\right)=\left \{ [\mathbf{a}_1,\mathbf{a}_{n/2}], [\mathbf{a}_{n/2+1},\mathbf{a}_{n}] \right\},
\]
\item  $\sum_{1\leq i\leq n}(2^{a_i-1}-1)$ orbits of size $2n$ when $n$ is odd and $\sum_{1\leq i\leq n}(2^{a_i}-2)$ orbits of size $n$ when $n$ is even, with representatives $\binom{\mathbf{a}_i}{k}$ in both cases, of the form 
\[\O\left(\binom{\mathbf{a}_i}{k}\right)=\left\{\binom{\mathbf{a}_i}{k},\left[\overline{\binom{\mathbf{a}_i}{k}},\mathbf{a}_{i+1}\right],\left[\binom{\mathbf{a}_i}{k},\mathbf{a}_{i+2}\right],\ldots, \left[x,\mathbf{a}_n],[\mathbf{a}_1,\overline{x}\right],\ldots,\left[\mathbf{a}_{i-1},\overline{\binom{\mathbf{a}_i}{k}}\right]\right\},\]
where $x=\binom{\mathbf{a}_i}{k}$ if $n-i$ is even and $x=\overline{\binom{\mathbf{a}_i}{k}}$ if $n-i$ is odd.
\end{itemize}
\end{thm}

\begin{proof}
Let $P=\mathbf{a}_1\oplus\mathbf{a}_2\oplus\cdots\oplus\mathbf{a}_n$ and $n\geq 3$. If an interval-closed set in $P$ contains elements from $3$ or more antichains, these antichains must be adjacent in the poset, and all elements of the interior antichains must be completely included. We classify the orbit structure of rowmotion on $\IC(P)$ by cases according to the number of partial antichains included in the interval-closed sets.

\medskip 
\noindent Case 1: Interval-closed sets containing only full antichains 

\smallskip
Consider the case of interval-closed sets $I$ of the form $\mathbf{a}_i$ or $[\mathbf{a}_i,\mathbf{a}_j]$. That is, $I$ contains all elements of each $\mathbf{a}_k$ where $i\leq k\leq j$. Let $f$ be the map that takes $\mathbf{a}_i$ to the single element $i$ and $[\mathbf{a}_i,\mathbf{a}_j]$ to the interval $[i,j]$ in the chain poset $[n]$. By Lemma \ref{lem:ints_commute_map}, rowmotion commutes with this map. 

Using Theorem \ref{thm:chains_orbits} on $[n]$ allows us to find the full orbit $\O(I)$ for each $I\in\IC(P)$ considered in this case. 
That is, we have an orbit of size $2$ containing $\{\emptyset, P \}$, $\lfloor \frac{n-1}{2}\rfloor$ orbits of size $n+2$, and if $n$ is even, an additional orbit of size $\frac{n+2}{2}$. 

\medskip
\noindent Case 2: Interval-closed sets containing one partial antichain and possibly other complete antichains 

\smallskip
Consider $I=\binom{\mathbf{a}_i}{k}$. Then, $\row(I)=\overline{\binom{\mathbf{a}_i}{k}}\oplus \mathbf{a}_{i+1},$ where $\overline{\binom{\mathbf{a}_i}{k}}$ denotes the complement $\mathbf{a}_i-\binom{\mathbf{a}_i}{k}$, and repeated applications of rowmotion continue to add all elements of the subsequent rank while alternating between the elements $\binom{\mathbf{a}_i}{k}$ and $\overline{\binom{\mathbf{a}_i}{k}}$.  Thus $\row^2(I)=[\binom{\mathbf{a}_i}{k},\mathbf{a}_{i+2}]$, $\row^3(I)=[\overline{\binom{\mathbf{a}_i}{k}},\mathbf{a}_{i+3}]$ and so forth until we reach $\row^{n-i}(I)=[x,\mathbf{a}_{n}]$, where $x=\binom{\mathbf{a}_i}{k}$ if $n-i$ is even and $x=\overline{\binom{\mathbf{a}_i}{k}}$ if $n-i$ is odd. At this point, applying rowmotion gives us the complement $\row^{n-i+1}(I)=[\mathbf{a}_1,\overline{x}].$  From here, subsequent applications of $\row$ remove elements of minimal rank while alternating between $x$ and $\overline{x}$, thus $\row^{n-i+2}(I)=[\mathbf{a}_2,x], \row^{n-i+3}(I)=[\mathbf{a}_3,\overline{x}],$ and so on.  After $n$ iterations we have $\row^n(I)=\binom{\mathbf{a}_i}{k}=I$ if $n$ is even, and $\row^n(I)=\overline{\binom{\mathbf{a}_i}{k}}$ if $n$ is odd.  In the case when $n$ is odd, an additional $n$ applications of $\row$ means returning to $I$. Thus $\O(\binom{\mathbf{a}_i}{k})$ has size $n$ when $n$ is even and $2n$ when $n$ is odd.  
      
Moreover, we observe that all interval-closed sets containing one partial antichain are included in these orbits. When $n$ is even, each of these orbits has a unique representative of the form $I=\binom{\mathbf{a}_i}{k}$, while when $n$ is odd, both $\binom{\mathbf{a}_i}{k}$ and $\overline{\binom{\mathbf{a}_i}{k}}$ appear in the same orbit.  Given there are $\sum_{1\leq i\leq n}(2^{a_i}-2)$ interval-closed sets of the form $I=\binom{\mathbf{a}_i}{k}$, when $n$ is even we have $\sum_{1\leq i\leq n}(2^{a_i}-2)$ orbits of size $n$, and when $n$ is odd we have half this many, $\sum_{1\leq i\leq n}(2^{a_i-1}-1)$, orbits of size $2n.$  

\medskip 
\noindent Case 3: Interval-closed sets containing two partial antichains and possibly other complete antichains

\smallskip
Consider a fixed interval-closed set $I = [\binom{\mathbf{a}_i}{k},\binom{\mathbf{a}_j}{l}]$, where $i\neq j$. Any interval-closed set of $\mathbf{a}_1\oplus \mathbf{a}_2\oplus \cdots \oplus \mathbf{a}_n$  with two partial antichains has this form. Then, $I$ is bounded above (resp. below) by a non-empty proper subset of the antichain $\mathbf{a}_j$ (resp. $\mathbf{a}_i$). Under application of rowmotion, all elements greater than those in $\mathbf{a}_j$ are not toggled into $I$, as there are elements of $\mathbf{a}_j$ smaller than them not in $I$ that are covers of elements of $I$, in other words they in are in scenario (1) of Lemma \ref{Lem: togobs}. Similarly, elements smaller than those in $\mathbf{a}_i$ are in scenario (2) and are not toggled into $I$, as there are elements of $\mathbf{a}_i$ greater than them not in $I$ that are covered by elements of $I$. Thus, $\row(I)$ remains bounded above by $\mathbf{a}_j$ and below by $\mathbf{a}_i$. In each of these ranks, rowmotion acts by returning the complements $\overline{\binom{\mathbf{a}_j}{l}}$ and $\overline{\binom{\mathbf{a}_i}{k}}$. Finally, the elements strictly between $\mathbf{a}_i$ and $\mathbf{a}_j$ are not toggled out of I, since there are larger elements in $\row(I)$, the elements in $\overline{\binom{\mathbf{a}_j}{l}}$ added under rowmotion, and smaller elements in $I$, the elements of $\binom{\mathbf{a}_i}{k}$ not yet removed by rowmotion. Thus, the orbit is $\{ [\binom{\mathbf{a}_i}{k},\binom{\mathbf{a}_j}{l}], [\overline{\binom{\mathbf{a}_i}{k}},\overline{\binom{\mathbf{a}_j}{l}]}\}$. Hence, $I$ is in an orbit of size $2$.
There is a total of $\sum_{1\leq i<j\leq n}(2^{a_i}-2)(2^{a_j}-2)$ interval-closed sets of the form $[\binom{\mathbf{a}_i}{k},\binom{\mathbf{a}_j}{l}]$ with $i\neq j$, which means that the number of orbits that contain them is $\frac{1}{2}\sum_{1\leq i<j\leq n}(2^{a_i}-2)(2^{a_j}-2)$.

By Theorem~\ref{thm:gen_ord_sum_ics_card} and counting elements in these various orbits, all $I\in\IC(P)$ are contained in one of the orbits described above.
So the order of rowmotion on ordinal sums of $n$ antichains is $\lcm(2,n+2,2n)=2n(n+2)$ when $n$ is odd, and $\lcm(2,n+2,\frac{n+2}{2},n)=n(n+2)/2$ when $n$ is even.
\end{proof}

\begin{ex} We end this subsection by illustrating some examples of the orbits of $P = \mathbf{2} \oplus \mathbf{3} \oplus \mathbf{1} \oplus \mathbf{4}$. 

There are three different types of orbits.

\begin{enumerate} 

\item The orbits of size two, which include $\{ \emptyset, P\}$ and non-trivial orbits consisting of interval-closed sets with both a minimal and (different) maximal rank with partially full antichains.

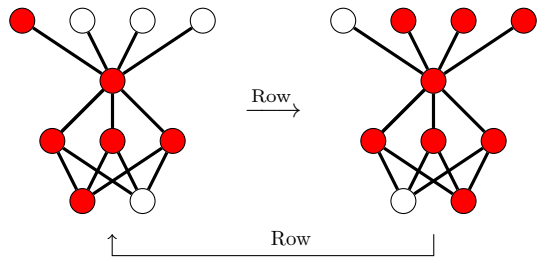
\begin{figure}[H]
\centering
	\begin{minipage}{.17\linewidth}\centering
		\begin{tikzpicture}[scale = .4]
		\node[draw,circle, fill=red](A) at (-1,-1) {};
		\node[draw,circle](B) at (1,-1) {};
		\node[draw,circle, fill=red](C) at (-2,1) {};
		\node[draw,circle, fill=red](D) at (0,1) {};
		\node[draw,circle, fill=red](E) at (2,1) {};    
		\node[draw,circle, fill=red](F) at (0,3) {};    
		\node[draw,circle, fill=red](G) at (-3,5) {};
		\node[draw,circle](H) at (-1,5) {};
		\node[draw,circle](I) at (1,5) {};
		\node[draw,circle](J) at (3,5) {};    
		\draw[very thick] (A) --(C) --(F) --(G);
		\draw[very thick] (A) --(D) --(F) --(H);
		\draw[very thick] (A) --(E) --(F) --(I);
		\draw[very thick] (B) --(C) --(F) --(J);
		\draw[very thick] (B) --(D);
		\draw[very thick] (B) --(E);    
	\end{tikzpicture}
	\end{minipage}\quad $\xlongrightarrow{\text{Row}}$\quad
	\begin{minipage}{.17\linewidth}\centering
		\begin{tikzpicture}[scale = .4]
		\node[draw,circle](A) at (-1,-1) {};
		\node[draw,circle, fill=red](B) at (1,-1) {};
		\node[draw,circle, fill=red](C) at (-2,1) {};
		\node[draw,circle, fill=red](D) at (0,1) {};
		\node[draw,circle, fill=red](E) at (2,1) {};    
		\node[draw,circle, fill=red](F) at (0,3) {};    
		\node[draw,circle](G) at (-3,5) {};
		\node[draw,circle, fill=red](H) at (-1,5) {};
		\node[draw,circle, fill=red](I) at (1,5) {};
		\node[draw,circle, fill=red](J) at (3,5) {};    
		\draw[very thick] (A) --(C) --(F) --(G);
		\draw[very thick] (A) --(D) --(F) --(H);
		\draw[very thick] (A) --(E) --(F) --(I);
		\draw[very thick] (B) --(C) --(F) --(J);
		\draw[very thick] (B) --(D);
		\draw[very thick] (B) --(E);    
		\end{tikzpicture}
	\end{minipage}\\
  \vspace{1mm}
\begin{minipage}{1\linewidth}\centering
\begin{tikzpicture}[scale=.95]
\node[anchor=center, scale=0.78] at (2.5,.25) {Row};
\draw[->] (4.5,.3)--(4.5,0)--(0,0)--(0,.3);
\end{tikzpicture}
\end{minipage}
\end{figure}

 \vspace{2.5mm}

\item The orbits of size $n+2$ and a single orbit of size $\frac{n+2}{2}$ (since $n$ is even), which consist of interval-closed sets that are made up of consecutive full antichains, excluding the full poset.

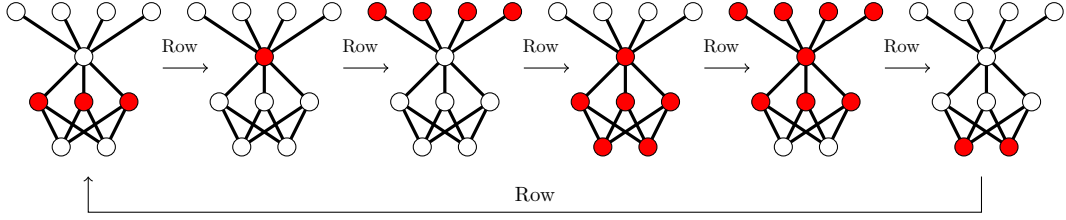
\begin{figure}[H]
\begin{center}
\begin{tikzpicture}[scale = .30]

 \draw [->] (3.5,2.5) -- (5.5,2.5);
 \node[anchor=center,scale=.7] at (4.25,3.5) {Row};
 \draw [->] (11.5,2.5) -- (13.5,2.5);
  \node[anchor=center,scale=.7] at (12.25,3.5) {Row};
  \draw [->] (19.5,2.5) -- (21.5,2.5);
   \node[anchor=center,scale=.7] at (20.25,3.5) {Row};
   \draw [->] (27.5,2.5) -- (29.5,2.5);
    \node[anchor=center,scale=.7] at (28.25,3.5) {Row};
    \draw [->] (35.5,2.5) -- (37.5,2.5);
     \node[anchor=center,scale=.7] at (36.25,3.5) {Row};

\draw [-, very thick] (-1,-1) -- (-2,1);
\draw [-, very thick] (-1,-1) -- (0,1);
\draw [-, very thick] (-1,-1) -- (2,1);
\draw [-, very thick] (1,-1) -- (-2,1);
\draw [-, very thick] (1,-1) -- (0,1);
\draw [-, very thick] (1,-1) -- (2,1);

\draw [-, very thick] (7,-1) -- (6,1);
\draw [-, very thick] (7,-1) -- (8,1);
\draw [-, very thick] (7,-1) -- (10,1);
\draw [-, very thick] (9,-1) -- (6,1);
\draw [-, very thick] (9,-1) -- (8,1);
\draw [-, very thick] (9,-1) -- (10,1);

\draw [-, very thick] (15,-1) -- (14,1);
\draw [-, very thick] (15,-1) -- (16,1);
\draw [-, very thick] (15,-1) -- (18,1);
\draw [-, very thick] (17,-1) -- (14,1);
\draw [-, very thick] (17,-1) -- (16,1);
\draw [-, very thick] (17,-1) -- (18,1);

\draw [-,very thick] (23,-1) -- (22,1);
\draw [-, very thick] (23,-1) -- (24,1);
\draw [-, very thick] (23,-1) -- (26,1);
\draw [-, very thick] (25,-1) -- (22,1);
\draw [-, very thick] (25,-1) -- (24,1);
\draw [-, very thick] (25,-1) -- (26,1);

\draw [-, very thick] (31,-1) -- (30,1);
\draw [-, very thick] (33,-1) -- (32,1);
\draw [-, very thick] (31,-1) -- (34,1);
\draw [-, very thick] (33,-1) -- (30,1);
\draw [-, very thick] (31,-1) -- (32,1);
\draw [-, very thick] (33,-1) -- (34,1);

\draw [-, very thick] (39,-1) -- (38,1);
\draw [-, very thick] (41,-1) -- (40,1);
\draw [-, very thick] (39,-1) -- (42,1);
\draw [-, very thick] (41,-1) -- (38,1);
\draw [-, very thick] (39,-1) -- (40,1);
\draw [-, very thick] (41,-1) -- (42,1);

\draw [-, very thick] (-2,1) -- (0,3);
\draw [-, very thick] (0,1) -- (0,3);
\draw [-, very thick] (2,1) -- (0,3);

\draw [-, very thick] (6,1) -- (8,3);
\draw [-, very thick] (8,1) -- (8,3);
\draw [-, very thick] (10,1) -- (8,3);

\draw [-, very thick] (14,1) -- (16,3);
\draw [-, very thick] (16,1) -- (16,3);
\draw [-, very thick] (18,1) -- (16,3);

\draw [-, very thick] (22,1) -- (24,3);
\draw [-, very thick] (24,1) -- (24,3);
\draw [-, very thick] (26,1) -- (24,3);

\draw [-, very thick] (30,1) -- (32,3);
\draw [-, very thick] (32,1) -- (32,3);
\draw [-, very thick] (34,1) -- (32,3);

\draw [-, very thick] (38,1) -- (40,3);
\draw [-, very thick] (40,1) -- (40,3);
\draw [-, very thick] (42,1) -- (40,3);

\draw [-, very thick] (0,3) -- (-3,5);
\draw [-, very thick] (0,3) -- (-1,5);
\draw [-, very thick] (0,3) -- (1,5);
\draw [-, very thick] (0,3) -- (3,5);

\draw [-, very thick] (8,3) -- (5,5);
\draw [-, very thick] (8,3) -- (7,5);
\draw [-, very thick] (8,3) -- (9,5);
\draw [-, very thick] (8,3) -- (11,5);

\draw [-, very thick] (16,3) -- (13,5);
\draw [-, very thick] (16,3) -- (15,5);
\draw [-, very thick] (16,3) -- (17,5);
\draw [-, very thick] (16,3) -- (19,5);

\draw [-, very thick] (24,3) -- (21,5);
\draw [-, very thick] (24,3) -- (23,5);
\draw [-, very thick] (24,3) -- (25,5);
\draw [-, very thick] (24,3) -- (27,5);

\draw [-, very thick] (32,3) -- (29,5);
\draw [-, very thick] (32,3) -- (31,5);
\draw [-, very thick] (32,3) -- (33,5);
\draw [-, very thick] (32,3) -- (35,5);

\draw [-, very thick] (40,3) -- (37,5);
\draw [-, very thick] (40,3) -- (39,5);
\draw [-, very thick] (40,3) -- (41,5);
\draw [-, very thick] (40,3) -- (43,5);

\draw[fill=white, radius = .2] (-1,-1) circle [radius = 0.4];
\draw[fill=white, radius = .2] (1,-1) circle [radius = 0.4];

\draw[fill=white, radius = .2] (7,-1) circle [radius = 0.4];
\draw[fill=white, radius = .2] (9,-1) circle [radius = 0.4];

\draw[fill=white, radius = .2] (15,-1) circle [radius = 0.4];
\draw[fill=white, radius = .2] (17,-1) circle [radius = 0.4];

\draw[fill=red, radius = .2] (23,-1) circle [radius = 0.4];
\draw[fill=red, radius = .2] (25,-1) circle [radius = 0.4];

\draw[fill=white, radius = .2] (31,-1) circle [radius = 0.4];
\draw[fill=white, radius = .2] (33,-1) circle [radius = 0.4];

\draw[fill=red, radius = .2] (39,-1) circle [radius = 0.4];
\draw[fill=red, radius = .2] (41,-1) circle [radius = 0.4];
\draw[fill=red, radius = .2] (-2,1) circle [radius = 0.4];
\draw[fill=red, radius = .2] (-0,1) circle [radius = 0.4];
\draw[fill=red, radius = .2] (2,1) circle [radius = 0.4];

\draw[fill=white, radius = .2] (6,1) circle [radius = 0.4];
\draw[fill=white, radius = .2] (8,1) circle [radius = 0.4];
\draw[fill=white, radius = .2] (10,1) circle [radius = 0.4];

\draw[fill=white, radius = .2] (14,1) circle [radius = 0.4];
\draw[fill=white, radius = .2] (16,1) circle [radius = 0.4];
\draw[fill=white, radius = .2] (18,1) circle [radius = 0.4];

\draw[fill=red, radius = .2] (22,1) circle [radius = 0.4];
\draw[fill=red, radius = .2] (24,1) circle [radius = 0.4];
\draw[fill=red, radius = .2] (26,1) circle [radius = 0.4];

\draw[fill=red, radius = .2] (30,1) circle [radius = 0.4];
\draw[fill=red, radius = .2] (32,1) circle [radius = 0.4];
\draw[fill=red, radius = .2] (34,1) circle [radius = 0.4];

\draw[fill=white, radius = .2] (38,1) circle [radius = 0.4];
\draw[fill=white, radius = .2] (40,1) circle [radius = 0.4];
\draw[fill=white, radius = .2] (42,1) circle [radius = 0.4];

\draw[fill=white, radius = .2] (0,3) circle [radius = 0.4];

\draw[fill=red, radius = .2] (8,3) circle [radius = 0.4];

\draw[fill=white, radius = .2] (16,3) circle [radius = 0.4];

\draw[fill=red, radius = .2] (24,3) circle [radius = 0.4];

\draw[fill=red, radius = .2] (32,3) circle [radius = 0.4];

\draw[fill=white, radius = .2] (40,3) circle [radius = 0.4];

\draw[fill=white, radius = .2] (-3,5) circle [radius = 0.4];
\draw[fill=white, radius = .2] (-1,5) circle [radius = 0.4];
\draw[fill=white, radius = .2] (1,5) circle [radius = 0.4];
\draw[fill=white, radius = .2] (3,5) circle [radius = 0.4];

\draw[fill=white, radius = .2] (5,5) circle [radius = 0.4];
\draw[fill=white, radius = .2] (7,5) circle [radius = 0.4];
\draw[fill=white, radius = .2] (9,5) circle [radius = 0.4];
\draw[fill=white, radius = .2] (11,5) circle [radius = 0.4];

\draw[fill=red, radius = .2] (13,5) circle [radius = 0.4];
\draw[fill=red, radius = .2] (15,5) circle [radius = 0.4];
\draw[fill=red, radius = .2] (17,5) circle [radius = 0.4];
\draw[fill=red, radius = .2] (19,5) circle [radius = 0.4];

\draw[fill=white, radius = .2] (21,5) circle [radius = 0.4];
\draw[fill=white, radius = .2] (23,5) circle [radius = 0.4];
\draw[fill=white, radius = .2] (25,5) circle [radius = 0.4];
\draw[fill=white, radius = .2] (27,5) circle [radius = 0.4];

\draw[fill=red, radius = .2] (29,5) circle [radius = 0.4];
\draw[fill=red, radius = .2] (31,5) circle [radius = 0.4];
\draw[fill=red, radius = .2] (33,5) circle [radius = 0.4];
\draw[fill=red, radius = .2] (35,5) circle [radius = 0.4];

\draw[fill=white, radius = .2] (37,5) circle [radius = 0.4];
\draw[fill=white, radius = .2] (39,5) circle [radius = 0.4];
\draw[fill=white, radius = .2] (41,5) circle [radius = 0.4];
\draw[fill=white, radius = .2] (43,5) circle [radius = 0.4];
\end{tikzpicture}
\end{center}
  \vspace{1mm}
\begin{minipage}{1\linewidth}\centering
\begin{tikzpicture}[scale=.95]
\node[anchor=center, scale=0.78] at (6.25,.25) {Row};
\draw[->] (12.5,.5)--(12.5,0)--(0,0)--(0,.5);
\end{tikzpicture}
\end{minipage}
\end{figure}

\item Lastly, the orbits of size $n$, which consist of interval-closed sets with a single partially full antichain and possibly other complete antichains.

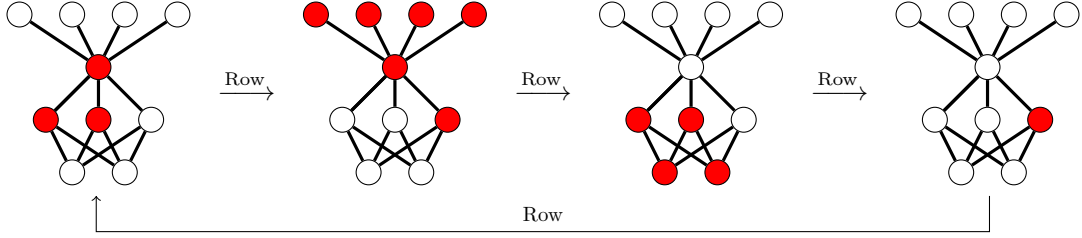
\begin{figure}[H]
\centering
	\begin{minipage}{.15\linewidth}\centering
		\begin{tikzpicture}[scale = .35]
		\node[draw,circle](A) at (-1,-1) {};
		\node[draw,circle](B) at (1,-1) {};
		\node[draw,circle, fill=red](C) at (-2,1) {};
		\node[draw,circle, fill=red](D) at (0,1) {};
		\node[draw,circle](E) at (2,1) {};    
		\node[draw,circle, fill=red](F) at (0,3) {};    
		\node[draw,circle](G) at (-3,5) {};
		\node[draw,circle](H) at (-1,5) {};
		\node[draw,circle](I) at (1,5) {};
		\node[draw,circle](J) at (3,5) {};    
		\draw[very thick] (A) --(C) --(F) --(G);
		\draw[very thick] (A) --(D) --(F) --(H);
		\draw[very thick] (A) --(E) --(F) --(I);
		\draw[very thick] (B) --(C) --(F) --(J);
		\draw[very thick] (B) --(D);
		\draw[very thick] (B) --(E);    
	\end{tikzpicture}
	\end{minipage}\quad $\xlongrightarrow{\text{Row}}$\quad
	\begin{minipage}{.15\linewidth}\centering
		\begin{tikzpicture}[scale = .35]
		\node[draw,circle](A) at (-1,-1) {};
		\node[draw,circle](B) at (1,-1) {};
		\node[draw,circle](C) at (-2,1) {};
		\node[draw,circle](D) at (0,1) {};
		\node[draw,circle, fill=red](E) at (2,1) {};    
		\node[draw,circle, fill=red](F) at (0,3) {};    
		\node[draw,circle, fill=red](G) at (-3,5) {};
		\node[draw,circle, fill=red](H) at (-1,5) {};
		\node[draw,circle, fill=red](I) at (1,5) {};
		\node[draw,circle, fill=red](J) at (3,5) {};    
		\draw[very thick] (A) --(C) --(F) --(G);
		\draw[very thick] (A) --(D) --(F) --(H);
		\draw[very thick] (A) --(E) --(F) --(I);
		\draw[very thick] (B) --(C) --(F) --(J);
		\draw[very thick] (B) --(D);
		\draw[very thick] (B) --(E);    
		\end{tikzpicture}
	\end{minipage}\quad $\xlongrightarrow{\text{Row}}$\quad
	\begin{minipage}{.15\linewidth}\centering
		\begin{tikzpicture}[scale = .35]
		\node[draw,circle, fill=red](A) at (-1,-1) {};
		\node[draw,circle, fill=red](B) at (1,-1) {};
		\node[draw,circle, fill=red](C) at (-2,1) {};
		\node[draw,circle, fill=red](D) at (0,1) {};
		\node[draw,circle](E) at (2,1) {};    
		\node[draw,circle](F) at (0,3) {};    
		\node[draw,circle](G) at (-3,5) {};
		\node[draw,circle](H) at (-1,5) {};
		\node[draw,circle](I) at (1,5) {};
		\node[draw,circle](J) at (3,5) {};    
		\draw[very thick] (A) --(C) --(F) --(G);
		\draw[very thick] (A) --(D) --(F) --(H);
		\draw[very thick] (A) --(E) --(F) --(I);
		\draw[very thick] (B) --(C) --(F) --(J);
		\draw[very thick] (B) --(D);
		\draw[very thick] (B) --(E);    
		\end{tikzpicture}
	\end{minipage}\quad $\xlongrightarrow{\text{Row}}$\quad
	\begin{minipage}{.15\linewidth}\centering
		\begin{tikzpicture}[scale = .35]
		\node[draw,circle](A) at (-1,-1) {};
		\node[draw,circle](B) at (1,-1) {};
		\node[draw,circle](C) at (-2,1) {};
		\node[draw,circle](D) at (0,1) {};
		\node[draw,circle, fill=red](E) at (2,1) {};    
		\node[draw,circle](F) at (0,3) {};    
		\node[draw,circle](G) at (-3,5) {};
		\node[draw,circle](H) at (-1,5) {};
		\node[draw,circle](I) at (1,5) {};
		\node[draw,circle](J) at (3,5) {};    
		\draw[very thick] (A) --(C) --(F) --(G);
		\draw[very thick] (A) --(D) --(F) --(H);
		\draw[very thick] (A) --(E) --(F) --(I);
		\draw[very thick] (B) --(C) --(F) --(J);
		\draw[very thick] (B) --(D);
		\draw[very thick] (B) --(E);    
		\end{tikzpicture}
	\end{minipage}\\
  \vspace{1mm}
\begin{minipage}{1\linewidth}\centering
\begin{tikzpicture}[scale=.95]
\node[anchor=center, scale=0.78] at (6.25,.25) {Row};
\draw[->] (12.5,.5)--(12.5,0)--(0,0)--(0,.5);
\end{tikzpicture}
\end{minipage}
\end{figure}
 \vspace{2.5mm}
\end{enumerate}

\end{ex}


\subsection{Stacks of altered diamonds}

In this subsection, we focus on examples of specific ordinal sum posets. Specifically, the family we call \emph{stacked altered diamond posets}. These posets are generated using the construction of \textit{series-parallel posets} from \cite{Stanley2011}, which are posets created from smaller posets through disjoint unions and ordinal sums. A diamond poset is the ordinal sum $\mathbf{1} \oplus \mathbf{2} \oplus \mathbf{1}$, and a more general altered diamond is the ordinal sum $\mathbf{1} \oplus \mathbf{m} \oplus \mathbf{1}$. 

Let $D(n,m) = \mathbf{1}\oplus\mathbf{m}\oplus \mathbf{1}\oplus\cdots\oplus\mathbf{m}\oplus\mathbf{1}$, where $n$ is the number of summands. We can view this as stacking $\frac{n-1}{2}$ copies of an altered diamond $\mathbf{1} \oplus \mathbf{m} \oplus \mathbf{1}$ where $m\geq 2$. Alternatively, we can write $D(n, m) = \mathbf{a}_1\oplus\mathbf{a}_2\oplus\cdots\oplus\mathbf{a}_n$ with $a_i = 1$ for $i$ odd, and $a_i = m$ for $i$ even.

For example, we have $D(5,2)$ below.

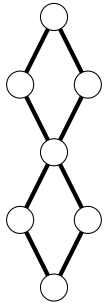
\begin{figure}[H]
\begin{center}
\begin{tikzpicture}[scale = .45]

\draw [-, ultra thick] (0,0) -- (-1,2);
\draw [-, ultra thick] (0,0) -- (1,2);

\draw [-, ultra thick] (-1,2) -- (0,4);
\draw [-, ultra thick] (1,2) -- (0,4);

\draw [-, ultra thick] (0,4) -- (-1,6);
\draw [-, ultra thick] (0,4) -- (1,6);

\draw [-, ultra thick] (0,8) -- (-1,6);
\draw [-, ultra thick] (0,8) -- (1,6);

\draw[fill=white, radius = .2] (0,0) circle [radius = 0.4];

\draw[fill=white, radius = .2] (-1,2) circle [radius = 0.4];
\draw[fill=white, radius = .2] (1,2) circle [radius = 0.4];

\draw[fill=white, radius = .2] (0,4) circle [radius = 0.4];

\draw[fill=white, radius = .2] (-1,6) circle [radius = 0.4];
\draw[fill=white, radius = .2] (1,6) circle [radius = 0.4];

\draw[fill=white, radius = .2] (0,8) circle [radius = 0.4];
\end{tikzpicture}
\end{center}
\end{figure}

For this family of ordinal sums, we have the following corollaries of Theorems \ref{thm:gen_ord_sum_ics_card} and \ref{thm:gen_ord_sum_row} respectively.

\begin{cor}
Let $n\geq 3$ be odd, and $m\geq 2$. The cardinality of $\IC(D(n,m))$ is 
\[
\frac{8+(n+1)(n+3)}{8}+\frac{(n-1)(n+3)}{4}(2^m-1) +\frac{(n-3)(n-1)}{8}(2^m-1)^2
\]

This consists of the following:
\begin{itemize}
\item the empty set $\emptyset$, 
\item non-empty subsets of size $k$ within a single antichain $\mathbf{a}_i,$
\item and subsets consisting of some non-empty subset of size $k$ of $\mathbf{a}_i$, some non-empty subset of size $l$ of $\mathbf{a}_j,$ with $i<j$, and all elements of $\mathbf{a}_r$ for all $i<r<j$. 
\end{itemize}
\end{cor}

\begin{proof}
Consider $D(n,m)= \mathbf{1}\oplus\mathbf{m}\oplus \mathbf{1}\oplus\cdots\oplus\mathbf{m}\oplus\mathbf{1}$. From Theorem \ref{thm:gen_ord_sum_ics_card}, we have a formula for the cardinality of $\IC(D(n,m))$ split into cases based on the number of antichains contributing to the interval-closed set: no antichains, a single antichains, or more than one antichain.

\begin{itemize}
\item No antichains: There is one interval-closed set of this type, the empty set.
\item A single antichain: There are $\frac{n+1}{2}$ interval-closed sets consisting of one full $\mathbf{a}_i$ with $i$ odd, and $\frac{n-1}{2}(2^m - 1)$ interval-closed sets consisting of a non-empty subset of $\mathbf{a}_i$ with $i$ even.
\item More than one antichain: Let $\mathbf{a}_i$ be the antichain of smallest rank contributing elements and $\mathbf{a}_j$ the antichain of largest rank. If both $i$ and $j$ are even, then we have $\binom{\frac{n-1}{2}}{2}(2^m-1)^2$ possibilities for interval-closed sets of this type. If both are odd, we have $\binom{\frac{n+1}{2}}{2}$ possibilities, and if they differ in parity we have $\frac{n-1}{2}\frac{n+1}{2}(2^m-1)$ possibilities.
\end{itemize}
All together, the cardinality of $\IC(D(n,m))$ is given by
\[
\displaystyle 1 + \frac{n+1}{2} + \frac{n-1}{2}(2^m - 1) +\binom{\frac{n-1}{2}}{2}(2^m-1)^2 + \binom{\frac{n+1}{2}}{2} + \frac{n-1}{2}\frac{n+1}{2}(2^m-1).
\]
Simplifying this calculation, we have
\[
\left[1+\frac{n+1}{2}+  \binom{\frac{n+1}{2}}{2} \right ]+ \left [\frac{n-1}{2}(2^m - 1)+ \frac{n-1}{2}\frac{n+1}{2}(2^m-1) \right]+\binom{\frac{n-1}{2}}{2}(2^m-1)^2
\]
\[
= \frac{8+(n+1)(n+3)}{8}+\frac{(n-1)(n+3)}{4}(2^m-1) +\frac{(n-3)(n-1)}{8}(2^m-1)^2
\]
as desired.
\end{proof}

\begin{cor}
 Let $n \geq 3$ be odd, and $m\geq 2$. 
The set $\IC(D(n,m))$ has rowmotion order  $2n(n+2)$.
Moreover, a complete description of its rowmotion orbit structure is below:
\begin{itemize}
\item $1 + \frac{(n-3)(n-1)}{8}(2^{m-1}-1)(2^{m}-2)$ orbits of size $2$, corresponding to the orbit of $\emptyset$, and orbits with representatives of the form $[\binom{\mathbf{a}_i}{k},\binom{\mathbf{a}_j}{l}]$ with $i \neq j$ even,
\item $\frac{n-1}{2}$ orbits of size $n+2$,  with representatives of the form $[\mathbf{a}_1, \mathbf{a}_j]$ with $j < \frac{n }{2}$,
\item  $\frac{n-1}{2}(2^{m-1}-1)$ orbits of size $2n$, with representatives of the form $\binom{\mathbf{a}_i}{k}$ with $i$ even.
\end{itemize}
\end{cor}

\begin{proof} Consider $D(n,m)= \mathbf{1}\oplus\mathbf{m}\oplus \mathbf{1}\oplus\cdots\oplus\mathbf{m}\oplus\mathbf{1}$, with $n\geq 3$ odd. We can use the formulas provided in Theorem \ref{thm:gen_ord_sum_row} to find the complete descriptions of the rowmotion orbit structure. As before, there are only two sizes for the antichains $\mathbf{a}_i$: $a_i = 1$ for $i$ odd, and $a_ i = m$ for $i$ even. This fact allows us to simplify the calculations significantly.

We have $$1+\displaystyle\sum_{1\leq i<j\leq n}(2^{a_i-1}-1)(2^{a_j}-2) = 1+\displaystyle\sum_{\substack{1\leq i<j\leq n\\ i, j \ \textrm{even}}} (2^{m-1}-1)(2^{m}-2).$$ Thus, there are $$1 + \displaystyle\binom{\frac{n-1}{2}}{2}(2^{m-1}-1)(2^{m}-2)= 1 + \frac{(n-3)(n-1)}{8}(2^{m-1}-1)(2^{m}-2)$$ orbits of size $2$.

Similarly, $$\displaystyle \sum_{1\leq i\leq n}(2^{a_i-1}-1) = \displaystyle \sum_{\substack{1\leq i\leq n\\ i \ \textrm{even}}}(2^{m-1}-1),$$ so there are $\frac{n-1}{2}(2^{m-1}-1)$ orbits of size $2n$

As $n$ is always odd, the only orbits with representatives of the form $[\mathbf{a}_i,\mathbf{a}_j]$ are of size $n+2$ and there remain $\frac{n-1}{2}$ of them as these orbits depend only on the rank of the poset and not the sizes of the antichains.

Finally, the order of rowmotion is inherited directly from Theorem \ref{thm:gen_ord_sum_row}.
\end{proof}

\begin{ex} Consider the three different types of orbits for $D(7,3)$.

\begin{enumerate} 

\item The orbits of size two, which include the orbit of $\emptyset$ and non-trivial orbits consisting of interval-closed sets with both a minimal and (different) maximal rank with partially full antichains.

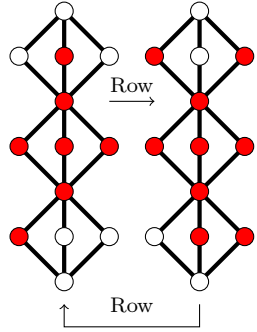
\begin{figure}[H]
\begin{center}
\begin{tikzpicture}[scale = .30]
 \draw [->] (2,8) -- (4,8);
 \node[anchor=center] at (3,8.75) {\footnotesize Row};
 \draw [->] (6,-1) -- (6,-2) -- (0,-2) -- (0,-1);
  \node[anchor=center] at (3,-1) {\footnotesize Row};

\draw [-, ultra thick] (0,0) -- (-2,2);
\draw [-, ultra thick] (0,0) -- (0,2);
\draw [-, ultra thick] (0,0) -- (2,2);

\draw [-, ultra thick] (6,0) -- (4,2);
\draw [-, ultra thick] (6,0) -- (6,2);
\draw [-, ultra thick] (6,0) -- (8,2);

\draw [-, ultra thick] (0,4) -- (-2,2);
\draw [-, ultra thick] (0,4) -- (0,2);
\draw [-, ultra thick] (0,4) -- (2,2);

\draw [-, ultra thick] (6,4) -- (4,2);
\draw [-, ultra thick] (6,4) -- (6,2);
\draw [-, ultra thick] (6,4) -- (8,2);

\draw [-, ultra thick] (0,4) -- (-2,6);
\draw [-, ultra thick] (0,4) -- (0,6);
\draw [-, ultra thick] (0,4) -- (2,6);

\draw [-, ultra thick] (6,4) -- (4,6);
\draw [-, ultra thick] (6,4) -- (6,6);
\draw [-, ultra thick] (6,4) -- (8,6);

\draw [-, ultra thick] (0,8) -- (-2,6);
\draw [-, ultra thick] (0,8) -- (0,6);
\draw [-, ultra thick] (0,8) -- (2,6);

\draw [-, ultra thick] (6,8) -- (4,6);
\draw [-, ultra thick] (6,8) -- (6,6);
\draw [-, ultra thick] (6,8) -- (8,6);

\draw [-, ultra thick] (0,8) -- (-2,10);
\draw [-, ultra thick] (0,8) -- (0,10);
\draw [-, ultra thick] (0,8) -- (2,10);

\draw [-, ultra thick] (6,8) -- (4,10);
\draw [-, ultra thick] (6,8) -- (6,10);
\draw [-, ultra thick] (6,8) -- (8,10);

\draw [-, ultra thick] (0,12) -- (-2,10);
\draw [-, ultra thick] (0,12) -- (0,10);
\draw [-, ultra thick] (0,12) -- (2,10);

\draw [-, ultra thick] (6,12) -- (4,10);
\draw [-, ultra thick] (6,12) -- (6,10);
\draw [-, ultra thick] (6,12) -- (8,10);

\draw[fill=white, radius = .2] (0,0) circle [radius = 0.4];

\draw[fill=white, radius = .2] (6,0) circle [radius = 0.4];

\draw[fill=red, radius = .2] (-2,2) circle [radius = 0.4];
\draw[fill=white, radius = .2] (0,2) circle [radius = 0.4];
\draw[fill=white, radius = .2] (2,2) circle [radius = 0.4];

\draw[fill=white, radius = .2] (4,2) circle [radius = 0.4];
\draw[fill=red, radius = .2] (6,2) circle [radius = 0.4];
\draw[fill=red, radius = .2] (8,2) circle [radius = 0.4];

\draw[fill=red, radius = .2] (0,4) circle [radius = 0.4];

\draw[fill=red, radius = .2] (6,4) circle [radius = 0.4];

\draw[fill=red, radius = .2] (-2,6) circle [radius = 0.4];
\draw[fill=red, radius = .2] (0,6) circle [radius = 0.4];
\draw[fill=red, radius = .2] (2,6) circle [radius = 0.4];

\draw[fill=red, radius = .2] (4,6) circle [radius = 0.4];
\draw[fill=red, radius = .2] (6,6) circle [radius = 0.4];
\draw[fill=red, radius = .2] (8,6) circle [radius = 0.4];

\draw[fill=red, radius = .2] (0,8) circle [radius = 0.4];

\draw[fill=red, radius = .2] (6,8) circle [radius = 0.4];

\draw[fill=white, radius = .2] (-2,10) circle [radius = 0.4];
\draw[fill=red, radius = .2] (0,10) circle [radius = 0.4];
\draw[fill=white, radius = .2] (2,10) circle [radius = 0.4];

\draw[fill=red, radius = .2] (4,10) circle [radius = 0.4];
\draw[fill=white, radius = .2] (6,10) circle [radius = 0.4];
\draw[fill=red, radius = .2] (8,10) circle [radius = 0.4];

\draw[fill=white, radius = .2] (0,12) circle [radius = 0.4];

\draw[fill=white, radius = .2] (6,12) circle [radius = 0.4];
\end{tikzpicture}
\end{center}
\end{figure}

\item The orbits of size $n+2$, which consist of interval-closed sets that are made up of consecutive full antichains, excluding the full poset.

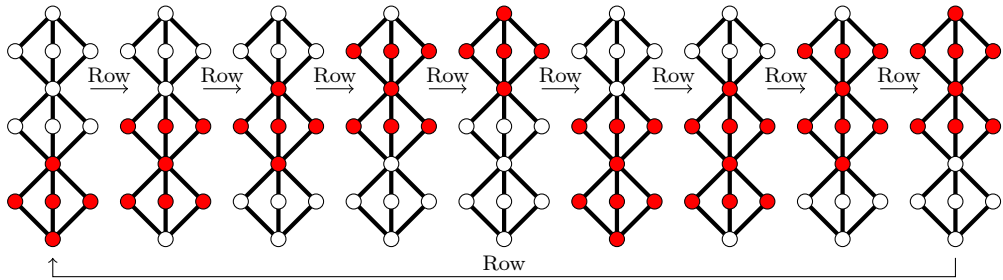
\begin{figure}[H]
\begin{center}
\begin{tikzpicture}[scale = .25]
\draw [->] (2,8) -- (4,8);
 \node[anchor=center] at (3,8.75) {\footnotesize Row};
 \draw [->] (8,8) -- (10,8);
  \node[anchor=center] at (9,8.75) {\footnotesize Row};
 \draw [->] (14,8) -- (16,8);
  \node[anchor=center] at (15,8.75) {\footnotesize Row};
 \draw [->] (20,8) -- (22,8);
  \node[anchor=center] at (21,8.75) {\footnotesize Row};
 \draw [->] (26,8) -- (28,8);
  \node[anchor=center] at (27,8.75) {\footnotesize Row};
 \draw [->] (32,8) -- (34,8);
  \node[anchor=center] at (33,8.75) {\footnotesize Row};
 \draw [->] (38,8) -- (40,8);
  \node[anchor=center] at (39,8.75) {\footnotesize Row};
 \draw [->] (44,8) -- (46,8);
  \node[anchor=center] at (45,8.75) {\footnotesize Row};
 \draw [->] (48,-1) -- (48,-2) -- (0,-2) -- (0,-1);
   \node[anchor=center] at (24,-1.25) {\footnotesize Row};

\draw [-, ultra thick] (0,0) -- (-2,2);
\draw [-, ultra thick] (0,0) -- (0,2);
\draw [-, ultra thick] (0,0) -- (2,2);

\draw [-, ultra thick] (6,0) -- (4,2);
\draw [-, ultra thick] (6,0) -- (6,2);
\draw [-, ultra thick] (6,0) -- (8,2);

\draw [-, ultra thick] (12,0) -- (10,2);
\draw [-, ultra thick] (12,0) -- (12,2);
\draw [-, ultra thick] (12,0) -- (14,2);

\draw [-, ultra thick] (18,0) -- (16,2);
\draw [-, ultra thick] (18,0) -- (18,2);
\draw [-, ultra thick] (18,0) -- (20,2);

\draw [-, ultra thick] (24,0) -- (22,2);
\draw [-, ultra thick] (24,0) -- (24,2);
\draw [-, ultra thick] (24,0) -- (26,2);

\draw [-, ultra thick] (30,0) -- (28,2);
\draw [-, ultra thick] (30,0) -- (30,2);
\draw [-, ultra thick] (30,0) -- (32,2);

\draw [-, ultra thick] (36,0) -- (34,2);
\draw [-, ultra thick] (36,0) -- (36,2);
\draw [-, ultra thick] (36,0) -- (38,2);

\draw [-, ultra thick] (42,0) -- (40,2);
\draw [-, ultra thick] (42,0) -- (42,2);
\draw [-, ultra thick] (42,0) -- (44,2);

\draw [-, ultra thick] (48,0) -- (46,2);
\draw [-, ultra thick] (48,0) -- (48,2);
\draw [-, ultra thick] (48,0) -- (50,2);

\draw [-, ultra thick] (0,4) -- (-2,2);
\draw [-, ultra thick] (0,4) -- (0,2);
\draw [-, ultra thick] (0,4) -- (2,2);

\draw [-, ultra thick] (6,4) -- (4,2);
\draw [-, ultra thick] (6,4) -- (6,2);
\draw [-, ultra thick] (6,4) -- (8,2);

\draw [-, ultra thick] (12,4) -- (10,2);
\draw [-, ultra thick] (12,4) -- (12,2);
\draw [-, ultra thick] (12,4) -- (14,2);

\draw [-, ultra thick] (18,4) -- (16,2);
\draw [-, ultra thick] (18,4) -- (18,2);
\draw [-, ultra thick] (18,4) -- (20,2);

\draw [-, ultra thick] (24,4) -- (22,2);
\draw [-, ultra thick] (24,4) -- (24,2);
\draw [-, ultra thick] (24,4) -- (26,2);

\draw [-, ultra thick] (30,4) -- (28,2);
\draw [-, ultra thick] (30,4) -- (30,2);
\draw [-, ultra thick] (30,4) -- (32,2);

\draw [-, ultra thick] (36,4) -- (34,2);
\draw [-, ultra thick] (36,4) -- (36,2);
\draw [-, ultra thick] (36,4) -- (38,2);

\draw [-, ultra thick] (42,4) -- (40,2);
\draw [-, ultra thick] (42,4) -- (42,2);
\draw [-, ultra thick] (42,4) -- (44,2);

\draw [-, ultra thick] (48,4) -- (46,2);
\draw [-, ultra thick] (48,4) -- (48,2);
\draw [-, ultra thick] (48,4) -- (50,2);

\draw [-, ultra thick] (0,4) -- (-2,6);
\draw [-, ultra thick] (0,4) -- (0,6);
\draw [-, ultra thick] (0,4) -- (2,6);

\draw [-, ultra thick] (6,4) -- (4,6);
\draw [-, ultra thick] (6,4) -- (6,6);
\draw [-, ultra thick] (6,4) -- (8,6);

\draw [-, ultra thick] (12,4) -- (10,6);
\draw [-, ultra thick] (12,4) -- (12,6);
\draw [-, ultra thick] (12,4) -- (14,6);

\draw [-, ultra thick] (18,4) -- (16,6);
\draw [-, ultra thick] (18,4) -- (18,6);
\draw [-, ultra thick] (18,4) -- (20,6);

\draw [-, ultra thick] (24,4) -- (22,6);
\draw [-, ultra thick] (24,4) -- (24,6);
\draw [-, ultra thick] (24,4) -- (26,6);

\draw [-, ultra thick] (30,4) -- (28,6);
\draw [-, ultra thick] (30,4) -- (30,6);
\draw [-, ultra thick] (30,4) -- (32,6);

\draw [-, ultra thick] (36,4) -- (34,6);
\draw [-, ultra thick] (36,4) -- (36,6);
\draw [-, ultra thick] (36,4) -- (38,6);

\draw [-, ultra thick] (42,4) -- (40,6);
\draw [-, ultra thick] (42,4) -- (42,6);
\draw [-, ultra thick] (42,4) -- (44,6);

\draw [-, ultra thick] (48,4) -- (46,6);
\draw [-, ultra thick] (48,4) -- (48,6);
\draw [-, ultra thick] (48,4) -- (50,6);

\draw [-, ultra thick] (0,8) -- (-2,6);
\draw [-, ultra thick] (0,8) -- (0,6);
\draw [-, ultra thick] (0,8) -- (2,6);

\draw [-, ultra thick] (6,8) -- (4,6);
\draw [-, ultra thick] (6,8) -- (6,6);
\draw [-, ultra thick] (6,8) -- (8,6);

\draw [-, ultra thick] (12,8) -- (10,6);
\draw [-, ultra thick] (12,8) -- (12,6);
\draw [-, ultra thick] (12,8) -- (14,6);

\draw [-, ultra thick] (18,8) -- (16,6);
\draw [-, ultra thick] (18,8) -- (18,6);
\draw [-, ultra thick] (18,8) -- (20,6);

\draw [-, ultra thick] (24,8) -- (22,6);
\draw [-, ultra thick] (24,8) -- (24,6);
\draw [-, ultra thick] (24,8) -- (26,6);

\draw [-, ultra thick] (30,8) -- (28,6);
\draw [-, ultra thick] (30,8) -- (30,6);
\draw [-, ultra thick] (30,8) -- (32,6);

\draw [-, ultra thick] (36,8) -- (34,6);
\draw [-, ultra thick] (36,8) -- (36,6);
\draw [-, ultra thick] (36,8) -- (38,6);

\draw [-, ultra thick] (42,8) -- (40,6);
\draw [-, ultra thick] (42,8) -- (42,6);
\draw [-, ultra thick] (42,8) -- (44,6);

\draw [-, ultra thick] (48,8) -- (46,6);
\draw [-, ultra thick] (48,8) -- (48,6);
\draw [-, ultra thick] (48,8) -- (50,6);

\draw [-, ultra thick] (0,8) -- (-2,10);
\draw [-, ultra thick] (0,8) -- (0,10);
\draw [-, ultra thick] (0,8) -- (2,10);

\draw [-, ultra thick] (6,8) -- (4,10);
\draw [-, ultra thick] (6,8) -- (6,10);
\draw [-, ultra thick] (6,8) -- (8,10);

\draw [-, ultra thick] (12,8) -- (10,10);
\draw [-, ultra thick] (12,8) -- (12,10);
\draw [-, ultra thick] (12,8) -- (14,10);

\draw [-, ultra thick] (18,8) -- (16,10);
\draw [-, ultra thick] (18,8) -- (18,10);
\draw [-, ultra thick] (18,8) -- (20,10);

\draw [-, ultra thick] (24,8) -- (22,10);
\draw [-, ultra thick] (24,8) -- (24,10);
\draw [-, ultra thick] (24,8) -- (26,10);

\draw [-, ultra thick] (30,8) -- (28,10);
\draw [-, ultra thick] (30,8) -- (30,10);
\draw [-, ultra thick] (30,8) -- (32,10);

\draw [-, ultra thick] (36,8) -- (34,10);
\draw [-, ultra thick] (36,8) -- (36,10);
\draw [-, ultra thick] (36,8) -- (38,10);

\draw [-, ultra thick] (42,8) -- (40,10);
\draw [-, ultra thick] (42,8) -- (42,10);
\draw [-, ultra thick] (42,8) -- (44,10);

\draw [-, ultra thick] (48,8) -- (46,10);
\draw [-, ultra thick] (48,8) -- (48,10);
\draw [-, ultra thick] (48,8) -- (50,10);

\draw [-, ultra thick] (0,12) -- (-2,10);
\draw [-, ultra thick] (0,12) -- (0,10);
\draw [-, ultra thick] (0,12) -- (2,10);

\draw [-, ultra thick] (6,12) -- (4,10);
\draw [-, ultra thick] (6,12) -- (6,10);
\draw [-, ultra thick] (6,12) -- (8,10);

\draw [-, ultra thick] (12,12) -- (10,10);
\draw [-, ultra thick] (12,12) -- (12,10);
\draw [-, ultra thick] (12,12) -- (14,10);

\draw [-, ultra thick] (18,12) -- (16,10);
\draw [-, ultra thick] (18,12) -- (18,10);
\draw [-, ultra thick] (18,12) -- (20,10);

\draw [-, ultra thick] (24,12) -- (22,10);
\draw [-, ultra thick] (24,12) -- (24,10);
\draw [-, ultra thick] (24,12) -- (26,10);

\draw [-, ultra thick] (30,12) -- (28,10);
\draw [-, ultra thick] (30,12) -- (30,10);
\draw [-, ultra thick] (30,12) -- (32,10);

\draw [-, ultra thick] (36,12) -- (34,10);
\draw [-, ultra thick] (36,12) -- (36,10);
\draw [-, ultra thick] (36,12) -- (38,10);

\draw [-, ultra thick] (42,12) -- (40,10);
\draw [-, ultra thick] (42,12) -- (42,10);
\draw [-, ultra thick] (42,12) -- (44,10);

\draw [-, ultra thick] (48,12) -- (46,10);
\draw [-, ultra thick] (48,12) -- (48,10);
\draw [-, ultra thick] (48,12) -- (50,10);

\draw[fill=red, radius = .2] (0,0) circle [radius = 0.4];

\draw[fill=white, radius = .2] (6,0) circle [radius = 0.4];

\draw[fill=white, radius = .2] (12,0) circle [radius = 0.4];

\draw[fill=white, radius = .2] (18,0) circle [radius = 0.4];

\draw[fill=white, radius = .2] (24,0) circle [radius = 0.4];

\draw[fill=red, radius = .2] (30,0) circle [radius = 0.4];

\draw[fill=white, radius = .2] (36,0) circle [radius = 0.4];

\draw[fill=white, radius = .2] (42,0) circle [radius = 0.4];

\draw[fill=white, radius = .2] (48,0) circle [radius = 0.4];

\draw[fill=red, radius = .2] (-2,2) circle [radius = 0.4];
\draw[fill=red, radius = .2] (0,2) circle [radius = 0.4];
\draw[fill=red, radius = .2] (2,2) circle [radius = 0.4];

\draw[fill=red, radius = .2] (4,2) circle [radius = 0.4];
\draw[fill=red, radius = .2] (6,2) circle [radius = 0.4];
\draw[fill=red, radius = .2] (8,2) circle [radius = 0.4];

\draw[fill=white, radius = .2] (10,2) circle [radius = 0.4];
\draw[fill=white, radius = .2] (12,2) circle [radius = 0.4];
\draw[fill=white, radius = .2] (14,2) circle [radius = 0.4];

\draw[fill=white, radius = .2] (16,2) circle [radius = 0.4];
\draw[fill=white, radius = .2] (18,2) circle [radius = 0.4];
\draw[fill=white, radius = .2] (20,2) circle [radius = 0.4];

\draw[fill=white, radius = .2] (22,2) circle [radius = 0.4];
\draw[fill=white, radius = .2] (24,2) circle [radius = 0.4];
\draw[fill=white, radius = .2] (26,2) circle [radius = 0.4];

\draw[fill=red, radius = .2] (28,2) circle [radius = 0.4];
\draw[fill=red, radius = .2] (30,2) circle [radius = 0.4];
\draw[fill=red, radius = .2] (32,2) circle [radius = 0.4];

\draw[fill=red, radius = .2] (34,2) circle [radius = 0.4];
\draw[fill=red, radius = .2] (36,2) circle [radius = 0.4];
\draw[fill=red, radius = .2] (38,2) circle [radius = 0.4];

\draw[fill=white, radius = .2] (40,2) circle [radius = 0.4];
\draw[fill=white, radius = .2] (42,2) circle [radius = 0.4];
\draw[fill=white, radius = .2] (44,2) circle [radius = 0.4];

\draw[fill=white, radius = .2] (46,2) circle [radius = 0.4];
\draw[fill=white, radius = .2] (48,2) circle [radius = 0.4];
\draw[fill=white, radius = .2] (50,2) circle [radius = 0.4];

\draw[fill=red, radius = .2] (0,4) circle [radius = 0.4];

\draw[fill=red, radius = .2] (6,4) circle [radius = 0.4];

\draw[fill=red, radius = .2] (12,4) circle [radius = 0.4];

\draw[fill=white, radius = .2] (18,4) circle [radius = 0.4];

\draw[fill=white, radius = .2] (24,4) circle [radius = 0.4];

\draw[fill=red, radius = .2] (30,4) circle [radius = 0.4];

\draw[fill=red, radius = .2] (36,4) circle [radius = 0.4];

\draw[fill=red, radius = .2] (42,4) circle [radius = 0.4];

\draw[fill=white, radius = .2] (48,4) circle [radius = 0.4];

\draw[fill=white, radius = .2] (-2,6) circle [radius = 0.4];
\draw[fill=white, radius = .2] (0,6) circle [radius = 0.4];
\draw[fill=white, radius = .2] (2,6) circle [radius = 0.4];

\draw[fill=red, radius = .2] (4,6) circle [radius = 0.4];
\draw[fill=red, radius = .2] (6,6) circle [radius = 0.4];
\draw[fill=red, radius = .2] (8,6) circle [radius = 0.4];

\draw[fill=red, radius = .2] (10,6) circle [radius = 0.4];
\draw[fill=red, radius = .2] (12,6) circle [radius = 0.4];
\draw[fill=red, radius = .2] (14,6) circle [radius = 0.4];

\draw[fill=red, radius = .2] (16,6) circle [radius = 0.4];
\draw[fill=red, radius = .2] (18,6) circle [radius = 0.4];
\draw[fill=red, radius = .2] (20,6) circle [radius = 0.4];

\draw[fill=white, radius = .2] (22,6) circle [radius = 0.4];
\draw[fill=white, radius = .2] (24,6) circle [radius = 0.4];
\draw[fill=white, radius = .2] (26,6) circle [radius = 0.4];

\draw[fill=red, radius = .2] (28,6) circle [radius = 0.4];
\draw[fill=red, radius = .2] (30,6) circle [radius = 0.4];
\draw[fill=red, radius = .2] (32,6) circle [radius = 0.4];

\draw[fill=red, radius = .2] (34,6) circle [radius = 0.4];
\draw[fill=red, radius = .2] (36,6) circle [radius = 0.4];
\draw[fill=red, radius = .2] (38,6) circle [radius = 0.4];

\draw[fill=red, radius = .2] (40,6) circle [radius = 0.4];
\draw[fill=red, radius = .2] (42,6) circle [radius = 0.4];
\draw[fill=red, radius = .2] (44,6) circle [radius = 0.4];

\draw[fill=red, radius = .2] (46,6) circle [radius = 0.4];
\draw[fill=red, radius = .2] (48,6) circle [radius = 0.4];
\draw[fill=red, radius = .2] (50,6) circle [radius = 0.4];

\draw[fill=white, radius = .2] (0,8) circle [radius = 0.4];

\draw[fill=white, radius = .2] (6,8) circle [radius = 0.4];

\draw[fill=red, radius = .2] (12,8) circle [radius = 0.4];

\draw[fill=red, radius = .2] (18,8) circle [radius = 0.4];

\draw[fill=red, radius = .2] (24,8) circle [radius = 0.4];

\draw[fill=white, radius = .2] (30,8) circle [radius = 0.4];

\draw[fill=red, radius = .2] (36,8) circle [radius = 0.4];

\draw[fill=red, radius = .2] (42,8) circle [radius = 0.4];

\draw[fill=red, radius = .2] (48,8) circle [radius = 0.4];

\draw[fill=white, radius = .2] (-2,10) circle [radius = 0.4];
\draw[fill=white, radius = .2] (0,10) circle [radius = 0.4];
\draw[fill=white, radius = .2] (2,10) circle [radius = 0.4];

\draw[fill=white, radius = .2] (4,10) circle [radius = 0.4];
\draw[fill=white, radius = .2] (6,10) circle [radius = 0.4];
\draw[fill=white, radius = .2] (8,10) circle [radius = 0.4];

\draw[fill=white, radius = .2] (10,10) circle [radius = 0.4];
\draw[fill=white, radius = .2] (12,10) circle [radius = 0.4];
\draw[fill=white, radius = .2] (14,10) circle [radius = 0.4];

\draw[fill=red, radius = .2] (16,10) circle [radius = 0.4];
\draw[fill=red, radius = .2] (18,10) circle [radius = 0.4];
\draw[fill=red, radius = .2] (20,10) circle [radius = 0.4];

\draw[fill=red, radius = .2] (22,10) circle [radius = 0.4];
\draw[fill=red, radius = .2] (24,10) circle [radius = 0.4];
\draw[fill=red, radius = .2] (26,10) circle [radius = 0.4];

\draw[fill=white, radius = .2] (28,10) circle [radius = 0.4];
\draw[fill=white, radius = .2] (30,10) circle [radius = 0.4];
\draw[fill=white, radius = .2] (32,10) circle [radius = 0.4];

\draw[fill=white, radius = .2] (34,10) circle [radius = 0.4];
\draw[fill=white, radius = .2] (36,10) circle [radius = 0.4];
\draw[fill=white, radius = .2] (38,10) circle [radius = 0.4];

\draw[fill=red, radius = .2] (40,10) circle [radius = 0.4];
\draw[fill=red, radius = .2] (42,10) circle [radius = 0.4];
\draw[fill=red, radius = .2] (44,10) circle [radius = 0.4];

\draw[fill=red, radius = .2] (46,10) circle [radius = 0.4];
\draw[fill=red, radius = .2] (48,10) circle [radius = 0.4];
\draw[fill=red, radius = .2] (50,10) circle [radius = 0.4];

\draw[fill=white, radius = .2] (0,12) circle [radius = 0.4];

\draw[fill=white, radius = .2] (6,12) circle [radius = 0.4];

\draw[fill=white, radius = .2] (12,12) circle [radius = 0.4];

\draw[fill=white, radius = .2] (18,12) circle [radius = 0.4];

\draw[fill=red, radius = .2] (24,12) circle [radius = 0.4];

\draw[fill=white, radius = .2] (30,12) circle [radius = 0.4];

\draw[fill=white, radius = .2] (36,12) circle [radius = 0.4];

\draw[fill=white, radius = .2] (42,12) circle [radius = 0.4];

\draw[fill=red, radius = .2] (48,12) circle [radius = 0.4];
\end{tikzpicture}
\end{center}
\end{figure}

\item Lastly, the orbits of size $2n$, which consist of interval-closed sets with a single antichain that is partially full and possibly other complete antichains.

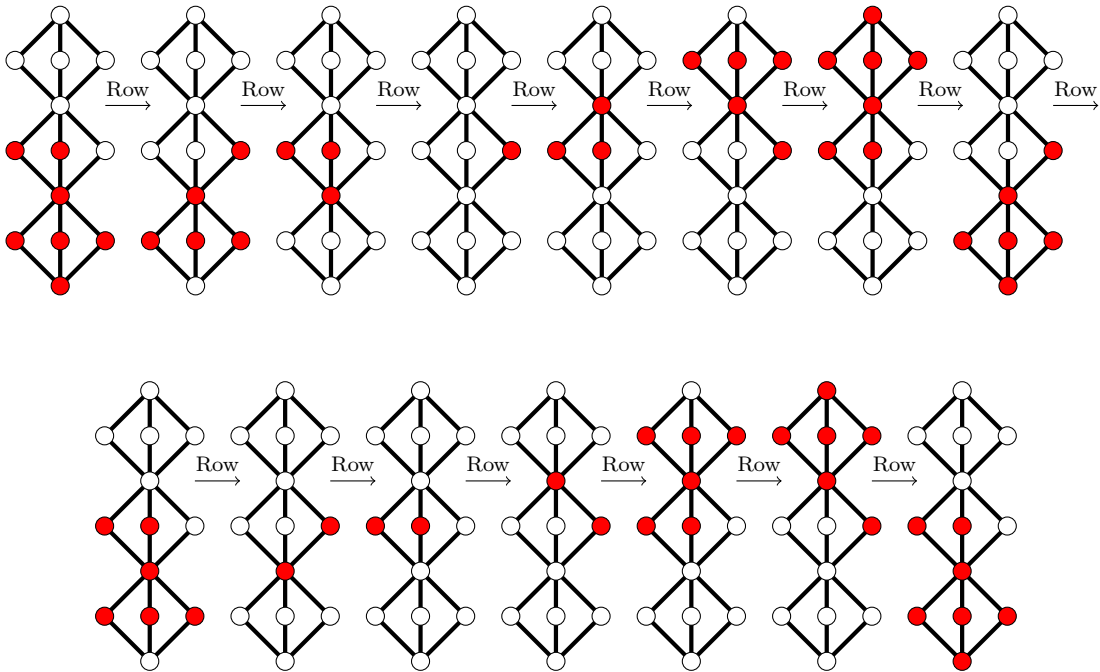
\begin{figure}[H]
\begin{center}
\begin{tikzpicture}[scale = .30]
\draw [->] (2,8) -- (4,8);
 \node[anchor=center] at (3,8.75) {\footnotesize Row};
 \draw [->] (8,8) -- (10,8);
  \node[anchor=center] at (9,8.75) {\footnotesize Row};
 \draw [->] (14,8) -- (16,8);
  \node[anchor=center] at (15,8.75) {\footnotesize Row};
 \draw [->] (20,8) -- (22,8);
  \node[anchor=center] at (21,8.75) {\footnotesize Row};
 \draw [->] (26,8) -- (28,8);
  \node[anchor=center] at (27,8.75) {\footnotesize Row};
 \draw [->] (32,8) -- (34,8);
  \node[anchor=center] at (33,8.75) {\footnotesize Row};
 \draw [->] (38,8) -- (40,8);
  \node[anchor=center] at (39,8.75) {\footnotesize Row};
  \draw [->] (44,8) -- (46,8);
   \node[anchor=center] at (45,8.75) {\footnotesize Row};
 
\draw [-, ultra thick] (0,0) -- (-2,2);
\draw [-, ultra thick] (0,0) -- (0,2);
\draw [-, ultra thick] (0,0) -- (2,2);

\draw [-, ultra thick] (6,0) -- (4,2);
\draw [-, ultra thick] (6,0) -- (6,2);
\draw [-, ultra thick] (6,0) -- (8,2);

\draw [-, ultra thick] (12,0) -- (10,2);
\draw [-, ultra thick] (12,0) -- (12,2);
\draw [-, ultra thick] (12,0) -- (14,2);

\draw [-, ultra thick] (18,0) -- (16,2);
\draw [-, ultra thick] (18,0) -- (18,2);
\draw [-, ultra thick] (18,0) -- (20,2);

\draw [-, ultra thick] (24,0) -- (22,2);
\draw [-, ultra thick] (24,0) -- (24,2);
\draw [-, ultra thick] (24,0) -- (26,2);

\draw [-, ultra thick] (30,0) -- (28,2);
\draw [-, ultra thick] (30,0) -- (30,2);
\draw [-, ultra thick] (30,0) -- (32,2);

\draw [-, ultra thick] (36,0) -- (34,2);
\draw [-, ultra thick] (36,0) -- (36,2);
\draw [-, ultra thick] (36,0) -- (38,2);

\draw [-, ultra thick] (42,0) -- (40,2);
\draw [-, ultra thick] (42,0) -- (42,2);
\draw [-, ultra thick] (42,0) -- (44,2);

\draw [-, ultra thick] (0,4) -- (-2,2);
\draw [-, ultra thick] (0,4) -- (0,2);
\draw [-, ultra thick] (0,4) -- (2,2);

\draw [-, ultra thick] (6,4) -- (4,2);
\draw [-, ultra thick] (6,4) -- (6,2);
\draw [-, ultra thick] (6,4) -- (8,2);

\draw [-, ultra thick] (12,4) -- (10,2);
\draw [-, ultra thick] (12,4) -- (12,2);
\draw [-, ultra thick] (12,4) -- (14,2);

\draw [-, ultra thick] (18,4) -- (16,2);
\draw [-, ultra thick] (18,4) -- (18,2);
\draw [-, ultra thick] (18,4) -- (20,2);

\draw [-, ultra thick] (24,4) -- (22,2);
\draw [-, ultra thick] (24,4) -- (24,2);
\draw [-, ultra thick] (24,4) -- (26,2);

\draw [-, ultra thick] (30,4) -- (28,2);
\draw [-, ultra thick] (30,4) -- (30,2);
\draw [-, ultra thick] (30,4) -- (32,2);

\draw [-, ultra thick] (36,4) -- (34,2);
\draw [-, ultra thick] (36,4) -- (36,2);
\draw [-, ultra thick] (36,4) -- (38,2);

\draw [-, ultra thick] (42,4) -- (40,2);
\draw [-, ultra thick] (42,4) -- (42,2);
\draw [-, ultra thick] (42,4) -- (44,2);

\draw [-, ultra thick] (0,4) -- (-2,6);
\draw [-, ultra thick] (0,4) -- (0,6);
\draw [-, ultra thick] (0,4) -- (2,6);

\draw [-, ultra thick] (6,4) -- (4,6);
\draw [-, ultra thick] (6,4) -- (6,6);
\draw [-, ultra thick] (6,4) -- (8,6);

\draw [-, ultra thick] (12,4) -- (10,6);
\draw [-, ultra thick] (12,4) -- (12,6);
\draw [-, ultra thick] (12,4) -- (14,6);

\draw [-, ultra thick] (18,4) -- (16,6);
\draw [-, ultra thick] (18,4) -- (18,6);
\draw [-, ultra thick] (18,4) -- (20,6);

\draw [-, ultra thick] (24,4) -- (22,6);
\draw [-, ultra thick] (24,4) -- (24,6);
\draw [-, ultra thick] (24,4) -- (26,6);

\draw [-, ultra thick] (30,4) -- (28,6);
\draw [-, ultra thick] (30,4) -- (30,6);
\draw [-, ultra thick] (30,4) -- (32,6);

\draw [-, ultra thick] (36,4) -- (34,6);
\draw [-, ultra thick] (36,4) -- (36,6);
\draw [-, ultra thick] (36,4) -- (38,6);

\draw [-, ultra thick] (42,4) -- (40,6);
\draw [-, ultra thick] (42,4) -- (42,6);
\draw [-, ultra thick] (42,4) -- (44,6);

\draw [-, ultra thick] (0,8) -- (-2,6);
\draw [-, ultra thick] (0,8) -- (0,6);
\draw [-, ultra thick] (0,8) -- (2,6);

\draw [-, ultra thick] (6,8) -- (4,6);
\draw [-, ultra thick] (6,8) -- (6,6);
\draw [-, ultra thick] (6,8) -- (8,6);

\draw [-, ultra thick] (12,8) -- (10,6);
\draw [-, ultra thick] (12,8) -- (12,6);
\draw [-, ultra thick] (12,8) -- (14,6);

\draw [-, ultra thick] (18,8) -- (16,6);
\draw [-, ultra thick] (18,8) -- (18,6);
\draw [-, ultra thick] (18,8) -- (20,6);

\draw [-, ultra thick] (24,8) -- (22,6);
\draw [-, ultra thick] (24,8) -- (24,6);
\draw [-, ultra thick] (24,8) -- (26,6);

\draw [-, ultra thick] (30,8) -- (28,6);
\draw [-, ultra thick] (30,8) -- (30,6);
\draw [-, ultra thick] (30,8) -- (32,6);

\draw [-, ultra thick] (36,8) -- (34,6);
\draw [-, ultra thick] (36,8) -- (36,6);
\draw [-, ultra thick] (36,8) -- (38,6);

\draw [-, ultra thick] (42,8) -- (40,6);
\draw [-, ultra thick] (42,8) -- (42,6);
\draw [-, ultra thick] (42,8) -- (44,6);

\draw [-, ultra thick] (0,8) -- (-2,10);
\draw [-, ultra thick] (0,8) -- (0,10);
\draw [-, ultra thick] (0,8) -- (2,10);

\draw [-, ultra thick] (6,8) -- (4,10);
\draw [-, ultra thick] (6,8) -- (6,10);
\draw [-, ultra thick] (6,8) -- (8,10);

\draw [-, ultra thick] (12,8) -- (10,10);
\draw [-, ultra thick] (12,8) -- (12,10);
\draw [-, ultra thick] (12,8) -- (14,10);

\draw [-, ultra thick] (18,8) -- (16,10);
\draw [-, ultra thick] (18,8) -- (18,10);
\draw [-, ultra thick] (18,8) -- (20,10);

\draw [-, ultra thick] (24,8) -- (22,10);
\draw [-, ultra thick] (24,8) -- (24,10);
\draw [-, ultra thick] (24,8) -- (26,10);

\draw [-, ultra thick] (30,8) -- (28,10);
\draw [-, ultra thick] (30,8) -- (30,10);
\draw [-, ultra thick] (30,8) -- (32,10);

\draw [-, ultra thick] (36,8) -- (34,10);
\draw [-, ultra thick] (36,8) -- (36,10);
\draw [-, ultra thick] (36,8) -- (38,10);

\draw [-, ultra thick] (42,8) -- (40,10);
\draw [-, ultra thick] (42,8) -- (42,10);
\draw [-, ultra thick] (42,8) -- (44,10);

\draw [-, ultra thick] (0,12) -- (-2,10);
\draw [-, ultra thick] (0,12) -- (0,10);
\draw [-, ultra thick] (0,12) -- (2,10);

\draw [-, ultra thick] (6,12) -- (4,10);
\draw [-, ultra thick] (6,12) -- (6,10);
\draw [-, ultra thick] (6,12) -- (8,10);

\draw [-, ultra thick] (12,12) -- (10,10);
\draw [-, ultra thick] (12,12) -- (12,10);
\draw [-, ultra thick] (12,12) -- (14,10);

\draw [-, ultra thick] (18,12) -- (16,10);
\draw [-, ultra thick] (18,12) -- (18,10);
\draw [-, ultra thick] (18,12) -- (20,10);

\draw [-, ultra thick] (24,12) -- (22,10);
\draw [-, ultra thick] (24,12) -- (24,10);
\draw [-, ultra thick] (24,12) -- (26,10);

\draw [-, ultra thick] (30,12) -- (28,10);
\draw [-, ultra thick] (30,12) -- (30,10);
\draw [-, ultra thick] (30,12) -- (32,10);

\draw [-, ultra thick] (36,12) -- (34,10);
\draw [-, ultra thick] (36,12) -- (36,10);
\draw [-, ultra thick] (36,12) -- (38,10);

\draw [-, ultra thick] (42,12) -- (40,10);
\draw [-, ultra thick] (42,12) -- (42,10);
\draw [-, ultra thick] (42,12) -- (44,10);

\draw[fill=red, radius = .2] (0,0) circle [radius = 0.4];

\draw[fill=white, radius = .2] (6,0) circle [radius = 0.4];

\draw[fill=white, radius = .2] (12,0) circle [radius = 0.4];

\draw[fill=white, radius = .2] (18,0) circle [radius = 0.4];

\draw[fill=white, radius = .2] (24,0) circle [radius = 0.4];

\draw[fill=white, radius = .2] (30,0) circle [radius = 0.4];

\draw[fill=white, radius = .2] (36,0) circle [radius = 0.4];

\draw[fill=red, radius = .2] (42,0) circle [radius = 0.4];

\draw[fill=red, radius = .2] (-2,2) circle [radius = 0.4];
\draw[fill=red, radius = .2] (0,2) circle [radius = 0.4];
\draw[fill=red, radius = .2] (2,2) circle [radius = 0.4];

\draw[fill=red, radius = .2] (4,2) circle [radius = 0.4];
\draw[fill=red, radius = .2] (6,2) circle [radius = 0.4];
\draw[fill=red, radius = .2] (8,2) circle [radius = 0.4];

\draw[fill=white, radius = .2] (10,2) circle [radius = 0.4];
\draw[fill=white, radius = .2] (12,2) circle [radius = 0.4];
\draw[fill=white, radius = .2] (14,2) circle [radius = 0.4];

\draw[fill=white, radius = .2] (16,2) circle [radius = 0.4];
\draw[fill=white, radius = .2] (18,2) circle [radius = 0.4];
\draw[fill=white, radius = .2] (20,2) circle [radius = 0.4];

\draw[fill=white, radius = .2] (22,2) circle [radius = 0.4];
\draw[fill=white, radius = .2] (24,2) circle [radius = 0.4];
\draw[fill=white, radius = .2] (26,2) circle [radius = 0.4];

\draw[fill=white, radius = .2] (28,2) circle [radius = 0.4];
\draw[fill=white, radius = .2] (30,2) circle [radius = 0.4];
\draw[fill=white, radius = .2] (32,2) circle [radius = 0.4];

\draw[fill=white, radius = .2] (34,2) circle [radius = 0.4];
\draw[fill=white, radius = .2] (36,2) circle [radius = 0.4];
\draw[fill=white, radius = .2] (38,2) circle [radius = 0.4];

\draw[fill=red, radius = .2] (40,2) circle [radius = 0.4];
\draw[fill=red, radius = .2] (42,2) circle [radius = 0.4];
\draw[fill=red, radius = .2] (44,2) circle [radius = 0.4];

\draw[fill=red, radius = .2] (0,4) circle [radius = 0.4];

\draw[fill=red, radius = .2] (6,4) circle [radius = 0.4];

\draw[fill=red, radius = .2] (12,4) circle [radius = 0.4];

\draw[fill=white, radius = .2] (18,4) circle [radius = 0.4];

\draw[fill=white, radius = .2] (24,4) circle [radius = 0.4];

\draw[fill=white, radius = .2] (30,4) circle [radius = 0.4];

\draw[fill=white, radius = .2] (36,4) circle [radius = 0.4];

\draw[fill=red, radius = .2] (42,4) circle [radius = 0.4];

\draw[fill=red, radius = .2] (-2,6) circle [radius = 0.4];
\draw[fill=red, radius = .2] (0,6) circle [radius = 0.4];
\draw[fill=white, radius = .2] (2,6) circle [radius = 0.4];

\draw[fill=white, radius = .2] (4,6) circle [radius = 0.4];
\draw[fill=white, radius = .2] (6,6) circle [radius = 0.4];
\draw[fill=red, radius = .2] (8,6) circle [radius = 0.4];

\draw[fill=red, radius = .2] (10,6) circle [radius = 0.4];
\draw[fill=red, radius = .2] (12,6) circle [radius = 0.4];
\draw[fill=white, radius = .2] (14,6) circle [radius = 0.4];

\draw[fill=white, radius = .2] (16,6) circle [radius = 0.4];
\draw[fill=white, radius = .2] (18,6) circle [radius = 0.4];
\draw[fill=red, radius = .2] (20,6) circle [radius = 0.4];

\draw[fill=red, radius = .2] (22,6) circle [radius = 0.4];
\draw[fill=red, radius = .2] (24,6) circle [radius = 0.4];
\draw[fill=white, radius = .2] (26,6) circle [radius = 0.4];

\draw[fill=white, radius = .2] (28,6) circle [radius = 0.4];
\draw[fill=white, radius = .2] (30,6) circle [radius = 0.4];
\draw[fill=red, radius = .2] (32,6) circle [radius = 0.4];

\draw[fill=red, radius = .2] (34,6) circle [radius = 0.4];
\draw[fill=red, radius = .2] (36,6) circle [radius = 0.4];
\draw[fill=white, radius = .2] (38,6) circle [radius = 0.4];

\draw[fill=white, radius = .2] (40,6) circle [radius = 0.4];
\draw[fill=white, radius = .2] (42,6) circle [radius = 0.4];
\draw[fill=red, radius = .2] (44,6) circle [radius = 0.4];

\draw[fill=white, radius = .2] (0,8) circle [radius = 0.4];

\draw[fill=white, radius = .2] (6,8) circle [radius = 0.4];

\draw[fill=white, radius = .2] (12,8) circle [radius = 0.4];

\draw[fill=white, radius = .2] (18,8) circle [radius = 0.4];

\draw[fill=red, radius = .2] (24,8) circle [radius = 0.4];

\draw[fill=red, radius = .2] (30,8) circle [radius = 0.4];

\draw[fill=red, radius = .2] (36,8) circle [radius = 0.4];

\draw[fill=white, radius = .2] (42,8) circle [radius = 0.4];

\draw[fill=white, radius = .2] (-2,10) circle [radius = 0.4];
\draw[fill=white, radius = .2] (0,10) circle [radius = 0.4];
\draw[fill=white, radius = .2] (2,10) circle [radius = 0.4];

\draw[fill=white, radius = .2] (4,10) circle [radius = 0.4];
\draw[fill=white, radius = .2] (6,10) circle [radius = 0.4];
\draw[fill=white, radius = .2] (8,10) circle [radius = 0.4];

\draw[fill=white, radius = .2] (10,10) circle [radius = 0.4];
\draw[fill=white, radius = .2] (12,10) circle [radius = 0.4];
\draw[fill=white, radius = .2] (14,10) circle [radius = 0.4];

\draw[fill=white, radius = .2] (16,10) circle [radius = 0.4];
\draw[fill=white, radius = .2] (18,10) circle [radius = 0.4];
\draw[fill=white, radius = .2] (20,10) circle [radius = 0.4];

\draw[fill=white, radius = .2] (22,10) circle [radius = 0.4];
\draw[fill=white, radius = .2] (24,10) circle [radius = 0.4];
\draw[fill=white, radius = .2] (26,10) circle [radius = 0.4];

\draw[fill=red, radius = .2] (28,10) circle [radius = 0.4];
\draw[fill=red, radius = .2] (30,10) circle [radius = 0.4];
\draw[fill=red, radius = .2] (32,10) circle [radius = 0.4];

\draw[fill=red, radius = .2] (34,10) circle [radius = 0.4];
\draw[fill=red, radius = .2] (36,10) circle [radius = 0.4];
\draw[fill=red, radius = .2] (38,10) circle [radius = 0.4];

\draw[fill=white, radius = .2] (40,10) circle [radius = 0.4];
\draw[fill=white, radius = .2] (42,10) circle [radius = 0.4];
\draw[fill=white, radius = .2] (44,10) circle [radius = 0.4];

\draw[fill=white, radius = .2] (0,12) circle [radius = 0.4];

\draw[fill=white, radius = .2] (6,12) circle [radius = 0.4];

\draw[fill=white, radius = .2] (12,12) circle [radius = 0.4];

\draw[fill=white, radius = .2] (18,12) circle [radius = 0.4];

\draw[fill=white, radius = .2] (24,12) circle [radius = 0.4];

\draw[fill=white, radius = .2] (30,12) circle [radius = 0.4];

\draw[fill=red, radius = .2] (36,12) circle [radius = 0.4];

\draw[fill=white, radius = .2] (42,12) circle [radius = 0.4];

\end{tikzpicture}
\end{center}
\end{figure}

\begin{figure}[H]
\begin{center}
\begin{tikzpicture}[scale = .30]

 \draw [->] (8,8) -- (10,8);
  \node[anchor=center] at (9,8.75) {\footnotesize Row};
 \draw [->] (14,8) -- (16,8);
  \node[anchor=center] at (15,8.75) {\footnotesize Row};
 \draw [->] (20,8) -- (22,8);
  \node[anchor=center] at (21,8.75) {\footnotesize Row};
 \draw [->] (26,8) -- (28,8);
  \node[anchor=center] at (27,8.75) {\footnotesize Row};
 \draw [->] (32,8) -- (34,8);
  \node[anchor=center] at (33,8.75) {\footnotesize Row};
 \draw [->] (38,8) -- (40,8);
  \node[anchor=center] at (39,8.75) {\footnotesize Row};

\draw [-, ultra thick] (6,0) -- (4,2);
\draw [-, ultra thick] (6,0) -- (6,2);
\draw [-, ultra thick] (6,0) -- (8,2);

\draw [-, ultra thick] (12,0) -- (10,2);
\draw [-, ultra thick] (12,0) -- (12,2);
\draw [-, ultra thick] (12,0) -- (14,2);

\draw [-, ultra thick] (18,0) -- (16,2);
\draw [-, ultra thick] (18,0) -- (18,2);
\draw [-, ultra thick] (18,0) -- (20,2);

\draw [-, ultra thick] (24,0) -- (22,2);
\draw [-, ultra thick] (24,0) -- (24,2);
\draw [-, ultra thick] (24,0) -- (26,2);

\draw [-, ultra thick] (30,0) -- (28,2);
\draw [-, ultra thick] (30,0) -- (30,2);
\draw [-, ultra thick] (30,0) -- (32,2);

\draw [-, ultra thick] (36,0) -- (34,2);
\draw [-, ultra thick] (36,0) -- (36,2);
\draw [-, ultra thick] (36,0) -- (38,2);

\draw [-, ultra thick] (42,0) -- (40,2);
\draw [-, ultra thick] (42,0) -- (42,2);
\draw [-, ultra thick] (42,0) -- (44,2);


\draw [-, ultra thick] (6,4) -- (4,2);
\draw [-, ultra thick] (6,4) -- (6,2);
\draw [-, ultra thick] (6,4) -- (8,2);

\draw [-, ultra thick] (12,4) -- (10,2);
\draw [-, ultra thick] (12,4) -- (12,2);
\draw [-, ultra thick] (12,4) -- (14,2);

\draw [-, ultra thick] (18,4) -- (16,2);
\draw [-, ultra thick] (18,4) -- (18,2);
\draw [-, ultra thick] (18,4) -- (20,2);

\draw [-, ultra thick] (24,4) -- (22,2);
\draw [-, ultra thick] (24,4) -- (24,2);
\draw [-, ultra thick] (24,4) -- (26,2);

\draw [-, ultra thick] (30,4) -- (28,2);
\draw [-, ultra thick] (30,4) -- (30,2);
\draw [-, ultra thick] (30,4) -- (32,2);

\draw [-, ultra thick] (36,4) -- (34,2);
\draw [-, ultra thick] (36,4) -- (36,2);
\draw [-, ultra thick] (36,4) -- (38,2);

\draw [-, ultra thick] (42,4) -- (40,2);
\draw [-, ultra thick] (42,4) -- (42,2);
\draw [-, ultra thick] (42,4) -- (44,2);


\draw [-, ultra thick] (6,4) -- (4,6);
\draw [-, ultra thick] (6,4) -- (6,6);
\draw [-, ultra thick] (6,4) -- (8,6);

\draw [-, ultra thick] (12,4) -- (10,6);
\draw [-, ultra thick] (12,4) -- (12,6);
\draw [-, ultra thick] (12,4) -- (14,6);

\draw [-, ultra thick] (18,4) -- (16,6);
\draw [-, ultra thick] (18,4) -- (18,6);
\draw [-, ultra thick] (18,4) -- (20,6);

\draw [-, ultra thick] (24,4) -- (22,6);
\draw [-, ultra thick] (24,4) -- (24,6);
\draw [-, ultra thick] (24,4) -- (26,6);

\draw [-, ultra thick] (30,4) -- (28,6);
\draw [-, ultra thick] (30,4) -- (30,6);
\draw [-, ultra thick] (30,4) -- (32,6);

\draw [-, ultra thick] (36,4) -- (34,6);
\draw [-, ultra thick] (36,4) -- (36,6);
\draw [-, ultra thick] (36,4) -- (38,6);

\draw [-, ultra thick] (42,4) -- (40,6);
\draw [-, ultra thick] (42,4) -- (42,6);
\draw [-, ultra thick] (42,4) -- (44,6);

\draw [-, ultra thick] (6,8) -- (4,6);
\draw [-, ultra thick] (6,8) -- (6,6);
\draw [-, ultra thick] (6,8) -- (8,6);

\draw [-, ultra thick] (12,8) -- (10,6);
\draw [-, ultra thick] (12,8) -- (12,6);
\draw [-, ultra thick] (12,8) -- (14,6);

\draw [-, ultra thick] (18,8) -- (16,6);
\draw [-, ultra thick] (18,8) -- (18,6);
\draw [-, ultra thick] (18,8) -- (20,6);

\draw [-, ultra thick] (24,8) -- (22,6);
\draw [-, ultra thick] (24,8) -- (24,6);
\draw [-, ultra thick] (24,8) -- (26,6);

\draw [-, ultra thick] (30,8) -- (28,6);
\draw [-, ultra thick] (30,8) -- (30,6);
\draw [-, ultra thick] (30,8) -- (32,6);

\draw [-, ultra thick] (36,8) -- (34,6);
\draw [-, ultra thick] (36,8) -- (36,6);
\draw [-, ultra thick] (36,8) -- (38,6);

\draw [-, ultra thick] (42,8) -- (40,6);
\draw [-, ultra thick] (42,8) -- (42,6);
\draw [-, ultra thick] (42,8) -- (44,6);

\draw [-, ultra thick] (6,8) -- (4,10);
\draw [-, ultra thick] (6,8) -- (6,10);
\draw [-, ultra thick] (6,8) -- (8,10);

\draw [-, ultra thick] (12,8) -- (10,10);
\draw [-, ultra thick] (12,8) -- (12,10);
\draw [-, ultra thick] (12,8) -- (14,10);

\draw [-, ultra thick] (18,8) -- (16,10);
\draw [-, ultra thick] (18,8) -- (18,10);
\draw [-, ultra thick] (18,8) -- (20,10);

\draw [-, ultra thick] (24,8) -- (22,10);
\draw [-, ultra thick] (24,8) -- (24,10);
\draw [-, ultra thick] (24,8) -- (26,10);

\draw [-, ultra thick] (30,8) -- (28,10);
\draw [-, ultra thick] (30,8) -- (30,10);
\draw [-, ultra thick] (30,8) -- (32,10);

\draw [-, ultra thick] (36,8) -- (34,10);
\draw [-, ultra thick] (36,8) -- (36,10);
\draw [-, ultra thick] (36,8) -- (38,10);

\draw [-, ultra thick] (42,8) -- (40,10);
\draw [-, ultra thick] (42,8) -- (42,10);
\draw [-, ultra thick] (42,8) -- (44,10);


\draw [-, ultra thick] (6,12) -- (4,10);
\draw [-, ultra thick] (6,12) -- (6,10);
\draw [-, ultra thick] (6,12) -- (8,10);

\draw [-, ultra thick] (12,12) -- (10,10);
\draw [-, ultra thick] (12,12) -- (12,10);
\draw [-, ultra thick] (12,12) -- (14,10);

\draw [-, ultra thick] (18,12) -- (16,10);
\draw [-, ultra thick] (18,12) -- (18,10);
\draw [-, ultra thick] (18,12) -- (20,10);

\draw [-, ultra thick] (24,12) -- (22,10);
\draw [-, ultra thick] (24,12) -- (24,10);
\draw [-, ultra thick] (24,12) -- (26,10);

\draw [-, ultra thick] (30,12) -- (28,10);
\draw [-, ultra thick] (30,12) -- (30,10);
\draw [-, ultra thick] (30,12) -- (32,10);

\draw [-, ultra thick] (36,12) -- (34,10);
\draw [-, ultra thick] (36,12) -- (36,10);
\draw [-, ultra thick] (36,12) -- (38,10);

\draw [-, ultra thick] (42,12) -- (40,10);
\draw [-, ultra thick] (42,12) -- (42,10);
\draw [-, ultra thick] (42,12) -- (44,10);

\draw[fill=white, radius = .2] (6,0) circle [radius = 0.4];

\draw[fill=white, radius = .2] (12,0) circle [radius = 0.4];

\draw[fill=white, radius = .2] (18,0) circle [radius = 0.4];

\draw[fill=white, radius = .2] (24,0) circle [radius = 0.4];

\draw[fill=white, radius = .2] (30,0) circle [radius = 0.4];

\draw[fill=white, radius = .2] (36,0) circle [radius = 0.4];

\draw[fill=red, radius = .2] (42,0) circle [radius = 0.4];


\draw[fill=red, radius = .2] (4,2) circle [radius = 0.4];
\draw[fill=red, radius = .2] (6,2) circle [radius = 0.4];
\draw[fill=red, radius = .2] (8,2) circle [radius = 0.4];

\draw[fill=white, radius = .2] (10,2) circle [radius = 0.4];
\draw[fill=white, radius = .2] (12,2) circle [radius = 0.4];
\draw[fill=white, radius = .2] (14,2) circle [radius = 0.4];

\draw[fill=white, radius = .2] (16,2) circle [radius = 0.4];
\draw[fill=white, radius = .2] (18,2) circle [radius = 0.4];
\draw[fill=white, radius = .2] (20,2) circle [radius = 0.4];

\draw[fill=white, radius = .2] (22,2) circle [radius = 0.4];
\draw[fill=white, radius = .2] (24,2) circle [radius = 0.4];
\draw[fill=white, radius = .2] (26,2) circle [radius = 0.4];

\draw[fill=white, radius = .2] (28,2) circle [radius = 0.4];
\draw[fill=white, radius = .2] (30,2) circle [radius = 0.4];
\draw[fill=white, radius = .2] (32,2) circle [radius = 0.4];

\draw[fill=white, radius = .2] (34,2) circle [radius = 0.4];
\draw[fill=white, radius = .2] (36,2) circle [radius = 0.4];
\draw[fill=white, radius = .2] (38,2) circle [radius = 0.4];

\draw[fill=red, radius = .2] (40,2) circle [radius = 0.4];
\draw[fill=red, radius = .2] (42,2) circle [radius = 0.4];
\draw[fill=red, radius = .2] (44,2) circle [radius = 0.4];


\draw[fill=red, radius = .2] (6,4) circle [radius = 0.4];

\draw[fill=red, radius = .2] (12,4) circle [radius = 0.4];

\draw[fill=white, radius = .2] (18,4) circle [radius = 0.4];

\draw[fill=white, radius = .2] (24,4) circle [radius = 0.4];

\draw[fill=white, radius = .2] (30,4) circle [radius = 0.4];

\draw[fill=white, radius = .2] (36,4) circle [radius = 0.4];

\draw[fill=red, radius = .2] (42,4) circle [radius = 0.4];


\draw[fill=red, radius = .2] (4,6) circle [radius = 0.4];
\draw[fill=red, radius = .2] (6,6) circle [radius = 0.4];
\draw[fill=white, radius = .2] (8,6) circle [radius = 0.4];

\draw[fill=white, radius = .2] (10,6) circle [radius = 0.4];
\draw[fill=white, radius = .2] (12,6) circle [radius = 0.4];
\draw[fill=red, radius = .2] (14,6) circle [radius = 0.4];

\draw[fill=red, radius = .2] (16,6) circle [radius = 0.4];
\draw[fill=red, radius = .2] (18,6) circle [radius = 0.4];
\draw[fill=white, radius = .2] (20,6) circle [radius = 0.4];

\draw[fill=white, radius = .2] (22,6) circle [radius = 0.4];
\draw[fill=white, radius = .2] (24,6) circle [radius = 0.4];
\draw[fill=red, radius = .2] (26,6) circle [radius = 0.4];

\draw[fill=red, radius = .2] (28,6) circle [radius = 0.4];
\draw[fill=red, radius = .2] (30,6) circle [radius = 0.4];
\draw[fill=white, radius = .2] (32,6) circle [radius = 0.4];

\draw[fill=white, radius = .2] (34,6) circle [radius = 0.4];
\draw[fill=white, radius = .2] (36,6) circle [radius = 0.4];
\draw[fill=red, radius = .2] (38,6) circle [radius = 0.4];

\draw[fill=red, radius = .2] (40,6) circle [radius = 0.4];
\draw[fill=red, radius = .2] (42,6) circle [radius = 0.4];
\draw[fill=white, radius = .2] (44,6) circle [radius = 0.4];


\draw[fill=white, radius = .2] (6,8) circle [radius = 0.4];

\draw[fill=white, radius = .2] (12,8) circle [radius = 0.4];

\draw[fill=white, radius = .2] (18,8) circle [radius = 0.4];

\draw[fill=red, radius = .2] (24,8) circle [radius = 0.4];

\draw[fill=red, radius = .2] (30,8) circle [radius = 0.4];

\draw[fill=red, radius = .2] (36,8) circle [radius = 0.4];

\draw[fill=white, radius = .2] (42,8) circle [radius = 0.4];

\draw[fill=white, radius = .2] (4,10) circle [radius = 0.4];
\draw[fill=white, radius = .2] (6,10) circle [radius = 0.4];
\draw[fill=white, radius = .2] (8,10) circle [radius = 0.4];

\draw[fill=white, radius = .2] (10,10) circle [radius = 0.4];
\draw[fill=white, radius = .2] (12,10) circle [radius = 0.4];
\draw[fill=white, radius = .2] (14,10) circle [radius = 0.4];

\draw[fill=white, radius = .2] (16,10) circle [radius = 0.4];
\draw[fill=white, radius = .2] (18,10) circle [radius = 0.4];
\draw[fill=white, radius = .2] (20,10) circle [radius = 0.4];

\draw[fill=white, radius = .2] (22,10) circle [radius = 0.4];
\draw[fill=white, radius = .2] (24,10) circle [radius = 0.4];
\draw[fill=white, radius = .2] (26,10) circle [radius = 0.4];

\draw[fill=red, radius = .2] (28,10) circle [radius = 0.4];
\draw[fill=red, radius = .2] (30,10) circle [radius = 0.4];
\draw[fill=red, radius = .2] (32,10) circle [radius = 0.4];

\draw[fill=red, radius = .2] (34,10) circle [radius = 0.4];
\draw[fill=red, radius = .2] (36,10) circle [radius = 0.4];
\draw[fill=red, radius = .2] (38,10) circle [radius = 0.4];

\draw[fill=white, radius = .2] (40,10) circle [radius = 0.4];
\draw[fill=white, radius = .2] (42,10) circle [radius = 0.4];
\draw[fill=white, radius = .2] (44,10) circle [radius = 0.4];


\draw[fill=white, radius = .2] (6,12) circle [radius = 0.4];

\draw[fill=white, radius = .2] (12,12) circle [radius = 0.4];

\draw[fill=white, radius = .2] (18,12) circle [radius = 0.4];

\draw[fill=white, radius = .2] (24,12) circle [radius = 0.4];

\draw[fill=white, radius = .2] (30,12) circle [radius = 0.4];

\draw[fill=red, radius = .2] (36,12) circle [radius = 0.4];

\draw[fill=white, radius = .2] (42,12) circle [radius = 0.4];
\end{tikzpicture}
\end{center}
\end{figure}

\end{enumerate}
\end{ex}

\subsection{Ordinal sums of repeated antichains}
In this subsection, we consider ordinal sums of the form $\bigoplus_{i=1}^n \mathbf{m}$. That is, ordinal sums where each rank has antichains of the same size.

\begin{cor}
Let $n\geq 3$, and $m\geq 2$. The cardinality of $\IC(\bigoplus_{i=1}^n \mathbf{m})$ is 
\[
1 + n(2^m - 1) + \binom{n}{2}(2^m - 1)^2.
\]
\end{cor}

\begin{proof}
Consider $\bigoplus_{i=1}^n \mathbf{m}$. From Theorem \ref{thm:gen_ord_sum_ics_card}, we have a formula for the cardinality of $\IC(\bigoplus_{i=1}^n \mathbf{m})$ split into cases based on the number of antichains contributing to the interval-closed set: no antichains, a single antichain, or more than one antichain.

For $\bigoplus_{i=1}^n \mathbf{m}$, there is only one size for the antichains. Thus $a_i = m$ for all $i$. Using this information, we have the following:

\begin{itemize}
\item No antichains: there is one element of this type, the empty set.
\item A single antichain: there are $n(2^m-1)$ elements of this type.
\item More than one antichain: there are $\binom{n}{2}(2^m - 1)^2$ elements of this type.
\end{itemize}
Summing these, we get the cardinality of $\IC(\bigoplus_{i=1}^n \mathbf{m})$.
\end{proof}

\begin{cor}
Let $n\geq 3$, and denote $\mathbf{m}_i$ as the antichain of rank $i + 1$. The set $\IC(\bigoplus_{i=1}^n \mathbf{m})$ has rowmotion order $2n(n+2)$ when $n$ is odd and $n(n+2)/2$ when $n$ is even. 
Moreover, a complete description of its rowmotion orbit structure is below:
\begin{itemize}
\item $1+\frac{1}{2}\binom{n}{2}(2^m-2)^2$ orbits of size $2$, corresponding to $\{\emptyset, P\}$ and the orbits with representatives of the form $\left [ \binom{\mathbf{m}_i}{k}, \binom{\mathbf{m}_j}{l} \right ]$,
\item $\lfloor \frac{n-1}{2}\rfloor$ orbits of size $n+2$, with representatives of the form $[\mathbf{m}_1, \mathbf{m}_j]$ with $j < \frac{n}{2}$,
\item An orbit of size $\frac{n+2}{2}$ when $n$ is even, having a representative $[\mathbf{m}_1, \mathbf{m}_{n/2}]$,
\item $n(2^{m-1}-1)$ orbits of size $2n$ when $n$ is odd, and $n(2^m-2)$ orbits of size $n$ if $n$ is even, with representatives of the form $\binom{\mathbf{m}_i}{k}$ in both cases. 
\end{itemize}

\end{cor}

\begin{proof} Consider $\bigoplus_{i=1}^n \mathbf{m}$. We can use the formulas provided in Theorem \ref{thm:gen_ord_sum_row} to find the complete descriptions of the rowmotion orbit structures. As all antichains are of size $m$, we can simplify the calculations significantly.

We have $$1+\displaystyle\frac{1}{2}\sum_{1\leq i<j\leq n}(2^{a_i}-2)(2^{a_j}-2) = 1+\displaystyle\frac{1}{2}\sum_{1\leq i<j\leq n}(2^{m}-2)^2.$$ Thus, there are $1 + \frac{1}{2}\binom{n}{2}(2^{m}-2)^2$ orbits of size $2$.

Similarly, $$\displaystyle \sum_{1\leq i\leq n}(2^{a_i}-2) = \displaystyle \sum_{1\leq i\leq n}(2^{m}-2),$$ so there are $\frac{n}{2}(2^{m}-2)$ orbits of size $2n$ when $n$ is odd and $n(2^m - 2)$ orbits of size $n$ when $n$ is even. 

As the $\lfloor \frac{n-1}{2}\rfloor$ orbits of size $n+2$, and the single orbit of size $\frac{n + 2}{2}$ when $n$ is even, depend only on the number of ranks, and not the size of the antichains, these remain unchanged. 

Finally, the order of rowmotion is inherited directly from Theorem \ref{thm:gen_ord_sum_row}.
\end{proof}

\subsection{Rowmotion as a global action on ordinal sums of antichains}

Next, we examine how the alternative definition of rowmotion given in Theorem \ref{thm:AltRow} simplifies when working with ordinal sums of antichains. Example \ref{ex:altdef ordsum} illustrates the categories of interval-closed sets considered in the proof of Theorem \ref{alt_def_ordinal}.  The reader is encouraged to refer to this example while processing the proof.

\begin{thm}\label{alt_def_ordinal}
Let $P = \mathbf{a}_1\oplus\mathbf{a}_2\oplus\cdots\oplus\mathbf{a}_n$ and $I \in \IC(P)$. Then rowmotion on $I$ is given as
\[\row(I) =
\begin{cases}
\overline{I} &\text{if } I=\emptyset, \ \mathbf{a}_n\subseteq I, \text{ or }  I\subseteq \mathbf{a}_n , \\
\oi\ceil{I}-\oi\Min(I) &\text{otherwise}.
\end{cases}\]
\end{thm}

\begin{proof}  Consider  $I \in \IC(P)$ where $P = \mathbf{a}_1\oplus\mathbf{a}_2\oplus\cdots\oplus\mathbf{a}_n.$  Corollary \ref{cor:empty ICS} and Lemma \ref{lem:max_elt_comp} give the result $\row(I)=\overline{I}$ when $I=\emptyset$ and $\mathbf{a}_n\subseteq I,$ respectively.  For all other cases, we start with the formula \eqref{Eqn:arow formula} from Theorem \ref{thm:AltRow}.

Consider the case $I\subseteq\mathbf{a}_n,$ i.e. $I=\binom{\mathbf{a}_n}{k}$ using the notation of Definition \ref{def: ordsumICS}.  In this case $\f(I)=I,$ implying $\ceil{I}$ and $\oi\ceil{I}-\oi\left(\Min(I)\cap\oi\ceil{I}\right)$ are both equal to the empty set.
Furthermore, $\inc(I)=\overline{\binom{\mathbf{a}_n}{k}}$ and $\inc_I(\ceil{I})=I.$  Thus Equation \eqref{Eqn:arow formula} becomes 
\[\row(I)=\overline{\binom{\mathbf{a}_n}{k}}\cup \left(\oi I-I\right)\cup \emptyset.\]
Since $I=\binom{\mathbf{a}_n}{k}$, $\oi I-I=[\mathbf{a}_1,\mathbf{a}_{n-1}]$ and $\overline{\binom{\mathbf{a}_n}{k}}\cup \left(\oi I-I\right)=[\mathbf{a}_1,\overline{\binom{\mathbf{a}_n}{k}}]=\overline{I}.$

By Theorem \ref{thm:gen_ord_sum_ics_card}, all other interval-closed sets of $\mathbf{a}_1\oplus\mathbf{a}_2\oplus\cdots\oplus\mathbf{a}_n$ fall into one of the following three categories.
\begin{enumerate}
\item Interval closed sets of the form $\mathbf{a}_j$, $[\mathbf{a}_i,\mathbf{a}_j]$, or $[\binom{\mathbf{a}_i}{k},\mathbf{a}_j]$ with $i<j<n$,
\item Interval closed sets of the form $\binom{\mathbf{a}_j}{l}$ with $j<n$,
\item Interval closed sets of the form $[\mathbf{a}_i,\binom{\mathbf{a}_j}{l}]$ or $[\binom{\mathbf{a}_i}{k},\binom{\mathbf{a}_j}{l
}]$ with $i<j\leq n$.
\end{enumerate}

Consider interval-closed sets of the form $\mathbf{a}_j$, $[\mathbf{a}_i,\mathbf{a}_j]$, or $[\binom{\mathbf{a}_i}{k},\mathbf{a}_j],$ with $i<j<n.$  It follows by the ordinal sum of antichains construction that $\inc(I)=\emptyset$, $\ceil{I}=\mathbf{a}_{j+1}$, $\inc_I(\ceil{I})=\emptyset$, and $\Min(I)\cap\oi\ceil{I}=\Min(I).$  Substituting these values into Equation \eqref{Eqn:arow formula} we find 
\[\row(I)=\emptyset\cup\emptyset\cup\left(\oi\ceil{I}-\oi\Min(I)\right)=\oi\ceil{I}-\oi\Min(I).\]

For interval-closed sets of the form $I=\binom{\mathbf{a}_j}{l}$ with $j<n$, we find $\inc(I)=\overline{\binom{\mathbf{a}_j}{l}}.$ Once again, however, $\ceil{I}=\mathbf{a}_{j+1},$ $\inc_I(\ceil{I})=\emptyset,$ and $\Min(I)\cap\oi\ceil{I}=\Min(I).$ Thus, Equation \eqref{Eqn:arow formula} simplifies to 
\[\row(I)=\overline{\binom{\mathbf{a}_j}{l}}\cup\emptyset\cup \left(\oi\ceil{I}-\oi\Min(I)\right).\]
As $\Min(I)=I$ in this case, $\overline{\binom{\mathbf{a}_j}{l}}$ is disjoint from $\oi\Min(I)$ but a subset of $\oi\ceil{I}.$ Therefore $\overline{\binom{\mathbf{a}_j}{l}}\cup\emptyset\cup \left(\oi\ceil{I}-\oi\Min(I)\right)=\oi\ceil{I}-\oi\Min(I).$

A slightly more complex computation finds the same result for interval-closed sets of the form $[\mathbf{a}_i,\binom{\mathbf{a}_j}{l}]$ or $[\binom{\mathbf{a}_i}{k},\binom{\mathbf{a}_j}{l
}],$ with $i<j\leq n$.  For $I$ in this category, $\inc(I)=\emptyset$, $\ceil{I}=\overline{\binom{\mathbf{a}_j}{l}}$, and $\inc_I(\ceil{I})=\binom{\mathbf{a}_j}{l}=\Max(I).$ Thus $\oi\ceil{I}=[\mathbf{a}_1,\overline{\binom{\mathbf{a}_j}{l}}]$ and $\Min(I)\cap\oi\ceil{I}=\Min(I).$  Moreover $\oi\inc_I(\ceil{I})=\oi I,$ which is a subset of $I\cup \oi\ceil{I},$ and $\oi\inc_I(\ceil{I})-\left(I\cup\oi\ceil{I}\right)=\emptyset.$   By Equation \eqref{Eqn:arow formula},
\[\row(I)=\oi\ceil{I}-\oi\Min(I).\]
\end{proof}

\begin{ex}\label{ex:altdef ordsum}
Let $P = \textbf{2} \oplus \textbf{4}   \oplus \textbf{2}   \oplus \textbf{4} $. Figures \ref{Fig: altdef ordsum ex1} - \ref{Fig: altdef ordsum ex4} illustrate the categories of interval-closed sets considered in the proof of Theorem \ref{alt_def_ordinal}. In each figure, the interval closed set $I$ is shown with red nodes followed by $\row(I)$ shown with blue nodes.  The minimum elements of $I$ are shown with pentagonal nodes.  Where existent, the ceiling of $I$ is indicated with triangular shaped nodes and the elements incomparable to $I$ are indicated with diamond shaped nodes.   

\begin{figure}[hbpt]
\begin{minipage}{.35\linewidth}\centering
		\begin{tikzpicture}[scale = .58]
\node[draw,circle](A) at (3,0) {\small{$a_{1,1}$}};
		\node[draw,circle](B) at (5,0) {\small{$a_{1,2}$} };
		\node[draw,circle](C) at (1,2) {\small{$a_{2,1}$}};
		\node[draw,circle](D) at (3,2) {\small{$a_{2,2}$} };
		\node[draw,circle](E) at (5,2) {\small{$a_{2,3}$} };
		\node[draw,circle](F) at (7, 2) {\small{$a_{2,4}$} };
		\node[draw,circle](G) at (3,4) {\small{$a_{3,1}$} };
    \node[draw,circle](H) at (5,4) {\small{$a_{3,2}$}};
    \node[draw,diamond, scale=.88](I) at (1,6) {\small{$a_{4,1}$}};
    \node[draw,diamond,  scale=.88](J) at (3,6) {\small{$a_{4,2}$}};
    \node[regular polygon, draw, regular polygon sides=5, inner sep=.4mm, fill=red!60](K) at (5,6) {\small{$a_{4,3}$}};
    \node[regular polygon, draw, regular polygon sides=5, inner sep=.4mm, fill=red!60](L) at (7,6) {\small{$a_{4,4}$}};
		\draw[very thick] (A) --(C) --(G) --(I);
		\draw[very thick] (A) --(D) --(G) --(J);
		\draw[very thick] (A) --(E) --(G) --(K);
		\draw[very thick] (A) --(F) --(G) --(L);
		\draw[very thick] (B) --(C) --(H) --(I);
		\draw[very thick] (B) --(D) --(H) --(J);
		\draw[very thick] (B) --(E) --(H) --(K);
		\draw[very thick] (B) --(F) --(H) --(L);
\end{tikzpicture} 
\end{minipage}\quad $\xlongrightarrow{\text{Row}}$\quad
\begin{minipage}{.35\linewidth}\centering
		\begin{tikzpicture}[scale = .58]
\node[draw,circle, fill=cyan](A) at (3,0) {\small{$a_{1,1}$}};
		\node[draw,circle, fill=cyan](B) at (5,0) {\small{$a_{1,2}$} };
		\node[draw,circle, fill=cyan](C) at (1,2) {\small{$a_{2,1}$}};
		\node[draw,circle, fill=cyan](D) at (3,2) {\small{$a_{2,2}$} };
		\node[draw,circle, fill=cyan](E) at (5,2) {\small{$a_{2,3}$} };
		\node[draw,circle, fill=cyan](F) at (7, 2) {\small{$a_{2,4}$} };
		\node[draw,circle, fill=cyan](G) at (3,4) {\small{$a_{3,1}$} };
    \node[draw,circle, fill=cyan](H) at (5,4) {\small{$a_{3,2}$}};
    \node[draw,diamond, fill=cyan, scale=.88](I) at (1,6) {\small{$a_{4,1}$}};
    \node[draw,diamond,  fill=cyan, scale=.88](J) at (3,6) {\small{$a_{4,2}$}};
    \node[regular polygon, draw, regular polygon sides=5, inner sep=.4mm](K) at (5,6) {\small{$a_{4,3}$}};
    \node[regular polygon, draw, regular polygon sides=5, inner sep=.4mm](L) at (7,6) {\small{$a_{4,4}$}};
		\draw[very thick] (A) --(C) --(G) --(I);
		\draw[very thick] (A) --(D) --(G) --(J);
		\draw[very thick] (A) --(E) --(G) --(K);
		\draw[very thick] (A) --(F) --(G) --(L);
		\draw[very thick] (B) --(C) --(H) --(I);
		\draw[very thick] (B) --(D) --(H) --(J);
		\draw[very thick] (B) --(E) --(H) --(K);
		\draw[very thick] (B) --(F) --(H) --(L);
\end{tikzpicture} 
\end{minipage}
\caption{ $I\subset \mathbf{a}_n$, i.e. $I=\binom{\mathbf{a}_n}{k},$ and $\row(I)=\overline{I}.$}\label{Fig: altdef ordsum ex1}
\bigskip
\begin{minipage}{.35\linewidth}\centering
		\begin{tikzpicture}[scale = .58]
\node[draw,circle](A) at (3,0) {\small{$a_{1,1}$}};
		\node[draw,circle](B) at (5,0) {\small{$a_{1,2}$} };
		\node[draw,circle](C) at (1,2) {\small{$a_{2,1}$}};
		\node[regular polygon, draw, regular polygon sides=5, inner sep=.4mm, fill=red!60](D) at (3,2) {\small{$a_{2,2}$} };
		\node[draw,circle](E) at (5,2) {\small{$a_{2,3}$} };
		\node[regular polygon, draw, regular polygon sides=5, inner sep=.4mm, fill=red!60](F) at (7, 2) {\small{$a_{2,4}$} };
		\node[draw,circle, fill=red!60](G) at (3,4) {\small{$a_{3,1}$} };
    \node[draw,circle, fill=red!60](H) at (5,4) {\small{$a_{3,2}$}};
    \node[regular polygon, draw, regular polygon sides=3, scale=.52](I) at (1,6) {$a_{4,1}$};
    \node[regular polygon, draw, regular polygon sides=3, scale=.52](J) at (3,6) {$a_{4,2}$};
    \node[regular polygon, draw, regular polygon sides=3, scale=.52](K) at (5,6) {$a_{4,3}$};
    \node[regular polygon, draw, regular polygon sides=3, scale=.52](L) at (7,6) {$a_{4,4}$};
		\draw[very thick] (A) --(C) --(G) --(I);
		\draw[very thick] (A) --(D) --(G) --(J);
		\draw[very thick] (A) --(E) --(G) --(K);
		\draw[very thick] (A) --(F) --(G) --(L);
		\draw[very thick] (B) --(C) --(H) --(I);
		\draw[very thick] (B) --(D) --(H) --(J);
		\draw[very thick] (B) --(E) --(H) --(K);
		\draw[very thick] (B) --(F) --(H) --(L);
\end{tikzpicture} 
\end{minipage}\quad $\xlongrightarrow{\text{Row}}$\quad
\begin{minipage}{.35\linewidth}\centering
		\begin{tikzpicture}[scale = .58]
\node[draw,circle](A) at (3,0) {\small{$a_{1,1}$}};
		\node[draw,circle](B) at (5,0) {\small{$a_{1,2}$} };
		\node[draw,circle, fill=cyan](C) at (1,2) {\small{$a_{2,1}$}};
		\node[regular polygon, draw, regular polygon sides=5, inner sep=.4mm](D) at (3,2) {\small{$a_{2,2}$} };
		\node[draw,circle, fill=cyan](E) at (5,2) {\small{$a_{2,3}$} };
		\node[regular polygon, draw, regular polygon sides=5, inner sep=.4mm](F) at (7, 2) {\small{$a_{2,4}$} };
		\node[draw,circle, fill=cyan](G) at (3,4) {\small{$a_{3,1}$} };
    \node[draw,circle, fill=cyan](H) at (5,4) {\small{$a_{3,2}$}};
    \node[regular polygon, draw, regular polygon sides=3, scale=.52, fill=cyan](I) at (1,6) {$a_{4,1}$};
    \node[regular polygon, draw, regular polygon sides=3, scale=.52, fill=cyan](J) at (3,6) {$a_{4,2}$};
    \node[regular polygon, draw, regular polygon sides=3, scale=.52, fill=cyan](K) at (5,6) {$a_{4,3}$};
    \node[regular polygon, draw, regular polygon sides=3, scale=.52, fill=cyan](L) at (7,6) {$a_{4,4}$};
		\draw[very thick] (A) --(C) --(G) --(I);
		\draw[very thick] (A) --(D) --(G) --(J);
		\draw[very thick] (A) --(E) --(G) --(K);
		\draw[very thick] (A) --(F) --(G) --(L);
		\draw[very thick] (B) --(C) --(H) --(I);
		\draw[very thick] (B) --(D) --(H) --(J);
		\draw[very thick] (B) --(E) --(H) --(K);
		\draw[very thick] (B) --(F) --(H) --(L);
\end{tikzpicture} 
\end{minipage}
\caption{ $I=[\binom{\mathbf{a}_i}{k},\mathbf{a}_j]$ with $i<j<n$ (Category (1)), and $\row(I)=\oi\ceil{I}-\oi\Min(I).$}

\bigskip

\begin{minipage}{.35\linewidth}\centering
		\begin{tikzpicture}[scale = .58]
\node[draw,circle](A) at (3,0) {\small{$a_{1,1}$}};
		\node[draw,circle](B) at (5,0) {\small{$a_{1,2}$} };
		\node[draw,diamond, scale=.88](C) at (1,2) {\small{$a_{2,1}$}};
		\node[regular polygon, draw, regular polygon sides=5, inner sep=.4mm, fill=red!60](D) at (3,2) {\small{$a_{2,2}$} };
		\node[draw,diamond,scale=.88](E) at (5,2) {\small{$a_{2,3}$} };
		\node[regular polygon, draw, regular polygon sides=5, inner sep=.4mm, fill=red!60](F) at (7, 2) {\small{$a_{2,4}$} };
		\node[regular polygon, draw, regular polygon sides=3, scale=.52](G) at (3,4) {$a_{3,1}$ };
    \node[regular polygon, draw, regular polygon sides=3, scale=.52](H) at (5,4) {$a_{3,2}$};
    \node[draw,circle](I) at (1,6) {\small{$a_{4,1}$}};
    \node[draw,circle](J) at (3,6) {\small{$a_{4,2}$}};
    \node[draw,circle](K) at (5,6) {\small{$a_{4,3}$}};
    \node[draw,circle](L) at (7,6) {\small{$a_{4,4}$}};
		\draw[very thick] (A) --(C) --(G) --(I);
		\draw[very thick] (A) --(D) --(G) --(J);
		\draw[very thick] (A) --(E) --(G) --(K);
		\draw[very thick] (A) --(F) --(G) --(L);
		\draw[very thick] (B) --(C) --(H) --(I);
		\draw[very thick] (B) --(D) --(H) --(J);
		\draw[very thick] (B) --(E) --(H) --(K);
		\draw[very thick] (B) --(F) --(H) --(L);
\end{tikzpicture} 
\end{minipage}\quad $\xlongrightarrow{\text{Row}}$\quad
\begin{minipage}{.35\linewidth}\centering
		\begin{tikzpicture}[scale = .58]
\node[draw,circle](A) at (3,0) {\small{$a_{1,1}$}};
		\node[draw,circle](B) at (5,0) {\small{$a_{1,2}$} };
		\node[draw,diamond, scale=.88,fill=cyan](C) at (1,2) {\small{$a_{2,1}$}};
		\node[regular polygon, draw, regular polygon sides=5, inner sep=.4mm](D) at (3,2) {\small{$a_{2,2}$} };
		\node[draw,diamond,scale=.88,fill=cyan](E) at (5,2) {\small{$a_{2,3}$} };
		\node[regular polygon, draw, regular polygon sides=5, inner sep=.4mm](F) at (7, 2) {\small{$a_{2,4}$} };
		\node[regular polygon, draw, regular polygon sides=3, scale=.52,fill=cyan](G) at (3,4) {$a_{3,1}$ };
    \node[regular polygon, draw, regular polygon sides=3, scale=.52,fill=cyan](H) at (5,4) {$a_{3,2}$};
    \node[draw,circle](I) at (1,6) {\small{$a_{4,1}$}};
    \node[draw,circle](J) at (3,6) {\small{$a_{4,2}$}};
    \node[draw,circle](K) at (5,6) {\small{$a_{4,3}$}};
    \node[draw,circle](L) at (7,6) {\small{$a_{4,4}$}};
		\draw[very thick] (A) --(C) --(G) --(I);
		\draw[very thick] (A) --(D) --(G) --(J);
		\draw[very thick] (A) --(E) --(G) --(K);
		\draw[very thick] (A) --(F) --(G) --(L);
		\draw[very thick] (B) --(C) --(H) --(I);
		\draw[very thick] (B) --(D) --(H) --(J);
		\draw[very thick] (B) --(E) --(H) --(K);
		\draw[very thick] (B) --(F) --(H) --(L);
\end{tikzpicture} 
\end{minipage}
	\caption{ $I=\binom{\mathbf{a}_j}{l}$ with $j<n$ (Category (2)), and $\row(I)=\oi\ceil{I}-\oi\Min(I).$}

\bigskip

\begin{minipage}{.35\linewidth}\centering
		\begin{tikzpicture}[scale = .58]
\node[draw,circle](A) at (3,0) {\small{$a_{1,1}$}};
		\node[draw,circle](B) at (5,0) {\small{$a_{1,2}$} };
		\node[draw,circle](C) at (1,2) {\small{$a_{2,1}$}};
		\node[draw,circle](D) at (3,2) {\small{$a_{2,2}$} };
		\node[draw,circle](E) at (5,2) {\small{$a_{2,3}$} };
		\node[draw,circle](F) at (7, 2) {\small{$a_{2,4}$} };
		\node[draw,circle](G) at (3,4) {\small{$a_{3,1}$} };
    \node[regular polygon, draw, regular polygon sides=5, inner sep=.4mm, fill=red!60](H) at (5,4) {\small{$a_{3,2}$}};
     \node[regular polygon, draw, regular polygon sides=3, scale=.52](I) at (1,6) {$a_{4,1}$};
    \node[regular polygon, draw, regular polygon sides=3, scale=.52](J) at (3,6) {$a_{4,2}$};
    \node[draw,circle, fill=red!60](K) at (5,6) {\small{$a_{4,3}$}};
    \node[draw,circle, fill=red!60](L) at (7,6) {\small{$a_{4,4}$}};
		\draw[very thick] (A) --(C) --(G) --(I);
		\draw[very thick] (A) --(D) --(G) --(J);
		\draw[very thick] (A) --(E) --(G) --(K);
		\draw[very thick] (A) --(F) --(G) --(L);
		\draw[very thick] (B) --(C) --(H) --(I);
		\draw[very thick] (B) --(D) --(H) --(J);
		\draw[very thick] (B) --(E) --(H) --(K);
		\draw[very thick] (B) --(F) --(H) --(L);
\end{tikzpicture} 
\end{minipage}\quad $\xlongrightarrow{\text{Row}}$\quad
\begin{minipage}{.35\linewidth}\centering
		\begin{tikzpicture}[scale = .58]
\node[draw,circle](A) at (3,0) {\small{$a_{1,1}$}};
		\node[draw,circle](B) at (5,0) {\small{$a_{1,2}$} };
		\node[draw,circle](C) at (1,2) {\small{$a_{2,1}$}};
		\node[draw,circle](D) at (3,2) {\small{$a_{2,2}$} };
		\node[draw,circle](E) at (5,2) {\small{$a_{2,3}$} };
		\node[draw,circle](F) at (7, 2) {\small{$a_{2,4}$} };
		\node[draw,circle, fill=cyan](G) at (3,4) {\small{$a_{3,1}$} };
    \node[regular polygon, draw, regular polygon sides=5, inner sep=.4mm](H) at (5,4) {\small{$a_{3,2}$}};
 \node[regular polygon, draw, regular polygon sides=3, scale=.52, fill=cyan](I) at (1,6) {$a_{4,1}$};
    \node[regular polygon, draw, regular polygon sides=3, scale=.52,fill=cyan](J) at (3,6) {$a_{4,2}$};
    \node[draw,circle](K) at (5,6) {\small{$a_{4,3}$}};
    \node[draw,circle](L) at (7,6) {\small{$a_{4,4}$}};
		\draw[very thick] (A) --(C) --(G) --(I);
		\draw[very thick] (A) --(D) --(G) --(J);
		\draw[very thick] (A) --(E) --(G) --(K);
		\draw[very thick] (A) --(F) --(G) --(L);
		\draw[very thick] (B) --(C) --(H) --(I);
		\draw[very thick] (B) --(D) --(H) --(J);
		\draw[very thick] (B) --(E) --(H) --(K);
		\draw[very thick] (B) --(F) --(H) --(L);
\end{tikzpicture}
\end{minipage}
	\caption{ $I=[\binom{\mathbf{a}_i}{k},\binom{\mathbf{a}_j}{l}]$ with $i<j\leq n$ (Category (3)), and $\row(I)=\oi\ceil{I}-\oi\Min(I).$}\label{Fig: altdef ordsum ex4}
\end{figure}

\end{ex}

\subsection{Signed cardinality homomesy}
\label{sec:ord_sum_homomesy}

Unlike in the case of order ideals, chains (and as we'll see in Section~\ref{sec:prod_chains_homomesy}, products of chains) do not in general exhibit homomesy with respect to the cardinality statistic. 

\begin{prop}\label{coro:chains_orbits} 
The set $\IC([n])$ with $n \geq 3$ under rowmotion does not exhibit homomesy with respect to the cardinality statistic.
Moreover, the average cardinality for each rowmotion orbit  is as follows:
\begin{itemize}
\item the orbit $\{\emptyset,[n]\}$ has average cardinality $\frac{n}{2}$, 
\item the $\lfloor \frac{n-1}{2}\rfloor$ orbits of size $n + 2$ have average cardinality $\displaystyle\frac{2k(n-k)+n}{n+2}$ for $1\leq k<\frac{n}{2}$,
\item and when $n$ is even, the single orbit of size $\frac{n + 2}{2}$ has average cardinality $\frac{n}{2}$.
\end{itemize}
\end{prop}
\begin{proof}
We begin by proving the average cardinality counts. The average cardinality of $\{\emptyset,[n]\}$ is $\frac{n}{2}$. 
By Theorem~\ref{thm:chains_orbits}, when $n$ is even there is a single orbit of size $\frac{n+2}{2}$ of the form $\O([1,\frac{n}{2}])=\{[1,\frac{n}{2}],[2,\frac{n}{2}+1],\ldots,[\frac{n}{2}+1,n]\}.$ Each interval-closed set in this orbit has cardinality $\frac{n}{2}$, thus the average is again $\frac{n}{2}$.

By Theorem~\ref{thm:chains_orbits}, when $n \geq 3$ there are $\lfloor\frac{n-1}{2}\rfloor$ orbits of size $n+2$ having the form $\O([1,k])=\{[1,k],[2,k+1],\ldots,[n-k+1,n],[1,n-k],\ldots,[k+1,n]\}$ for any $k<n/2$.  These orbits have $n-k+1$ interval-closed sets of cardinality $k$ and $k+1$ interval-closed sets of cardinality $n-k$. The sum of these cardinalities is $2nk-2k^2+n,$ and the average over the orbit is $\frac{2nk-2k^2+n}{n+2}.$

For $n=1,2$, there are no orbits of size $n+2$, and homomesy holds. But for $n=3$, we have an orbit of size $5$ with average cardinality $\frac{7}{5},$ as well as an orbit of size 2 with average cardinality $\frac{3}{2}.$ Thus the cardinality statistic is not homomesic. 
\end{proof}

Note that since the values $\frac{2k(n-k)+n}{n+2}$ differ for various values of $k$ in $1\leq k<\frac{n}{2}$, interval-closed sets of a chain are not generally \textit{homometric} (meaning orbit averages over orbits of the same size are not necessarily equal \cite{homometry}).

Unlike cardinality, signed cardinality, does exhibit homomesy for ordinal sums of antichains under certain conditions; however, there are still many cases where homomesy does not hold.

\begin{definition}
Fix a finite poset $P$. For each $x\in P$, define the \textbf{signed cardinality statistic} $SC(x): P \rightarrow \{-1, 1\}$  as follows:
\[SC(x) =
\begin{cases}
1 \textrm{ if } \rank(x) \textrm{ is even, } \\
-1 \textrm{ if } \rank(x) \textrm{ is odd. }
\end{cases}\]
For an interval-closed set $I$, $SC(I) = \sum_{x\in I}SC(x)$.
\end{definition}

\begin{prop}\label{prop:signed_card_chains}
The signed cardinality statistic exhibits homomesy with respect to rowmotion on the chain poset $P=[n]$ when $n$ is even.
\end{prop}

\begin{proof}
Suppose $n$ is even.  By Theorem \ref{thm:chains_orbits}, the orbits of $\IC([n])$ are: $\{\emptyset,[n]\}$, the orbits $\O([1,k])$ of size $n+2$ for $k<n/2$, and the orbit $\O([1,n/2])$ of size $\frac{n}{2}+1.$

Since $n$ is even, the signed cardinality of $[n]$ is zero, thus the average of the signed cardinality over the orbit $\{\emptyset,[n]\}$ is $0$.  For $k<n/2$, consider the orbit $\O([1,k])=\{[1,k],[2,k+1],\ldots,[n-k+1,n],[1,n-k],\ldots,[k+1,n]\}.$  If $k$ is even, the signed cardinality of each element of this orbit is $0$.  If $k$ is odd, the signed cardinality alternates between $\pm 1$.  As the orbit has even size $n+2$, the sum of the signed cardinality is $\frac{n+2}{2}(1-1)=0.$  Thus, the average of the signed cardinality is zero over all orbits of the form $\O([1,k])$ for $k<n/2$ when $n$ is even.

If $n/2$ is even, then signed cardinality of each element in the orbit $\O([1,n/2])=\{[1,n/2],[2,\frac{n}{2}+1],\ldots,[\frac{n}{2}+1,n]\}$ is zero.  If $n/2$ is odd, then the signed cardinalities once again alternate between $\pm 1$ over the orbit of even size $\frac{n}{2}+1,$ and the average of the signed cardinality is $\frac{1}{2}(\frac{n}{2}+1)(1-1)=0.$

Thus, in all cases the average of the signed cardinality is $0$.
\end{proof}

\begin{remark}
Note, when $n$ is odd the sum of the signed cardinalities over the orbits of $\IC([n])$ under rowmotion is always $1$.  Thus, as the orbits have different sizes, the signed cardinality is not homomesic when $n$ is odd.
\end{remark}

\begin{ex}
For ordinal sums of antichains with the following sequence of antichain lengths, we tested the homomesy of the signed cardinality statistic under rowmotion.  
\[\begin{array}{|c|c|}
\hline
\textrm{Homomesic} & \textrm{Not homomesic}\\
\hline
1,1,1,1 & 2,1,1,2\\
2,2,2,2 & 2,3,3,2\\
3,3,3,3 & 1,2,1,2,1\\
1,2,2,1 & 2,2,1,1,2,2\\
1,1,2,2,1,1 & 1,2,1,2\\
1,1,3,3,1,1 & 1,1,2,1,1\\
1,3,3,1 & 2,2,2\\
~& 1,2,1\\
~& 1,1,1,1,1\\
\hline
\end{array}\]
\end{ex}

\begin{thm}\label{thm:signed_card_ord_sum}
The signed cardinality statistic exhibits homomesy with respect to rowmotion on the poset $\bigoplus_{i=1}^n \mathbf{m}$ where $n$ is even.
\end{thm}

\begin{proof}
Let $P=\bigoplus_{i=1}^n \mathbf{m}$ where $n$ is even. By Theorem \ref{thm:gen_ord_sum_row} there are three different types of orbits to consider: those consisting of interval-closed sets containing only full antichains, those consisting of interval-closed sets containing one partial antichain and possibly other full antichains, and the orbits of size 2.

\medskip
\noindent Case 1: Interval-closed sets containing only full antichains. 

\smallskip
As in the proof of Theorem \ref{thm:gen_ord_sum_row}, we map our interval-closed set $I=[\mathbf{m}_i,\mathbf{m}_j]$ to the interval-closed set $[i,j]$ in $[n]$.  By Lemma \ref{lem:ints_commute_map}, this map commutes with rowmotion and the average of the signed cardinality statistic over the orbit of $I=[\mathbf{m}_i,\mathbf{m}_j]$ is just $m$ times the average over the orbit of $[i,j]$ in $[n].$  By Proposition \ref{prop:signed_card_chains}, we know that the signed cardinality statistic is 0-mesic for all orbits in $[n]$, thus it is alos $0$-mesic for all orbits with representatives of the form $I=[\mathbf{m}_i,\mathbf{m}_j]$ in $P.$  

\medskip
\noindent Case 2: Interval-closed sets containing one partial antichain and possibly other full antichains.

\smallskip
By Theorem~\ref{thm:chains_orbits}, the orbits of these interval-closed sets have the from \[\O\left(\binom{\mathbf{m}_i}{k}\right)=\left\{\binom{\mathbf{m}_i}{k},\left[\overline{\binom{\mathbf{m}_i}{k}},\mathbf{m}_{i+1}\right],\left[\binom{\mathbf{m}_i}{k},\mathbf{m}_{i+2}\right],\ldots, \left[x,\mathbf{m}_n],[\mathbf{m}_1,\overline{x}\right],\ldots,\left[\mathbf{m}_{i-1},\overline{\binom{\mathbf{m}_i}{k}}\right]\right\}.\] 
Elements of $\binom{\mathbf{m}_i}{k}$ are all at rank $i + 1$.  If $i$ is odd, this rank is signed $+1$ and thus the signed cardinality of $\binom{\mathbf{m}_i}{k}$ is $k$.  Similarly, the $m$ elements of $\mathbf{m}_{i+1}$ all lie on rank $i + 2$, which when $i$ is odd is signed $-1$, thus the signed cardinality of $[\overline{\binom{\mathbf{m}_i}{k}},\mathbf{m}_{i+1}]$ is $(m-k)-m=-k$.  Continuing, we find the signed cardinality of $[\binom{\mathbf{m}_i}{k},\mathbf{m}_{i+2}]$ is $k-m+m=k$, and so forth, with the signed cardinalities alternating between $\pm k$. 
If $i$ is even, the signed cardinalities alternate between $\mp k$.  In either case, as we are summing over an orbit of even size $n$, the average of the signed cardinality is zero. 

\medskip
\noindent Case 3: Orbits of size 2.

\smallskip
First consider the orbit $\{\emptyset, P\}$. As $n$ is even, we can pair the ranks of $P$ so the signed cardinality is $\sum_{i = 1}^{n/2} (m - m) = 0$. 

Next, consider the orbits of the form $\{[\binom{\mathbf{m}_i}{k}, \binom{\mathbf{m}_j}{l}], [\overline{\binom{\mathbf{m}_i}{k}}, \overline{\binom{\mathbf{m}_j}{l}}]\}$. Recall that the interval-closed sets in this orbit consist of a distinct minimal and maximal rank that are partially full (and all ranks between, if any, completely full). We break this into two cases depending on whether or not $i$ and $j$ have the same parity.

If rank $i$ and rank $j$ have different parity, then the number of ranks between them is even. Thus any of these ranks can be paired so their contribution to the signed cardinality sums to 0. As the maximal and minimal ranks toggle by taking the complement, the contribution of each over the full orbit is $(-1)^i m$  and $(-1)^j m$ respectively. As $i$ and $j$ differ in parity, these cancel, resulting in an overall signed cardinality of 0.

If rank $i$ and rank $j$ have the same parity, then the number of ranks between them is odd. We pair these except for the $(j-1)^{th}$ rank, resulting in an overall contribution of $(-1)^{j-1}2m$ to the signed cardinality across the orbit. As the maximal and minimal ranks toggle by taking the complement, the contribution of each is $(-1)^i m$  and $(-1)^j m$ respectively. As $i$ and $j$ have the same parity, this totals to $(-1)^j 2m$, which cancels with $(-1)^{j-1}2m$, resulting in an overall signed cardinality of 0.
\end{proof}

\section{Interval-closed sets of products of chains}\label{sec:products_of_chains}

Consider the Cartesian product $[m]\times[n]$ of chain posets $[m]$ and $[n]$.
Elements in $[m]\times[n]$ are the tuples $\{(a,b) \mid 1\leq a\leq m, 1\leq b\leq n\}$. For the poset $[2]\times [n]$, we say that $\{(1, i) \mid 1\leq i \leq n\}$ is in the \textit{bottom chain} and that $\{(2, j) \mid  1\leq j \leq n\}$ is in the \textit{top chain}. 
Similarly, one can take the products of multiple chains $[a_1]\times [a_2] \times \cdots \times[a_k]$, whose elements are given by tuples of length $k$.

\begin{remark}
We may interpret posets constructed as products of chains as \emph{divisor posets}. Let $d = p_1^{e_1}p_2^{e_2}\dotsm p_k^{e_k}$ be the prime factorization of $d$. Then the divisor poset of $d$ is the product of chains poset $[e_1 + 1] \times [e_2 + 1] \times \dotsm [e_k + 1]$. 

For example, the divisor poset of $48 = 3*2^4$ is the poset  $[2] \times [5]$. The two primes correspond to the two chains, and the number of elements in each chain is one more than the exponent of the prime in the decomposition. 

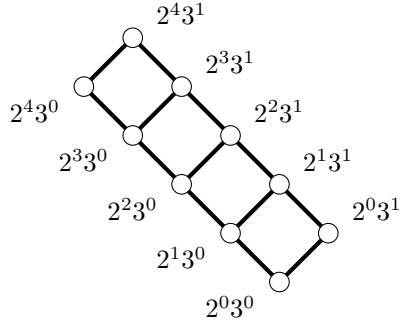
\begin{center}
\begin{minipage}[c]{5cm}
\begin{tikzpicture}[scale = .65]

\draw [-, ultra thick] (-1,0) -- (-2,1);
\draw [-, ultra thick] (-2,1) -- (-3,2);
\draw [-, ultra thick] (-3,2) -- (-4,3);
\draw [-, ultra thick] (-4,3) -- (-5,4);

\draw [-, ultra thick] (-3,0) -- (-4,1);
\draw [-, ultra thick] (-2,-1) -- (-3,0);
\draw [-, ultra thick] (-4,1) -- (-5,2);
\draw [-, ultra thick] (-5,2) -- (-6,3);

\draw [-, ultra thick] (-1,0) -- (-2,-1);
\draw [-, ultra thick] (-2,1) -- (-3,0);
\draw [-, ultra thick] (-3,2) -- (-4,1);
\draw [-, ultra thick] (-4,3) -- (-5,2);
\draw [-, ultra thick] (-5,4) -- (-6,3);

\draw[fill=white, radius = .2] (-1,0) circle [radius = 0.2];
\draw (-0,0.5) node{$2^0 3^1$};
\draw[fill=white, radius = .2] (-2,1) circle [radius = 0.2];
\draw (-1,1.5) node{$2^1 3^1$};
\draw[fill=white, radius = .2] (-3,2) circle [radius = 0.2];
\draw (-2,2.5) node{$2^2 3^1$};
\draw[fill=white, radius = .2] (-4,3) circle [radius = 0.2];
\draw (-3,3.5) node{$2^3 3^1$};
\draw[fill=white, radius = .2] (-5,4) circle [radius = 0.2];
\draw (-4,4.5) node{$2^4 3^1$};
\draw[fill=white, radius = .2] (-2,-1) circle [radius = 0.2];
\draw (-3,-1.5) node{$2^0 3^0$};
\draw[fill=white, radius = .2] (-3,0) circle [radius = 0.2];
\draw (-4,-0.5) node{$2^1 3^0$};
\draw[fill=white, radius = .2] (-4,1) circle [radius = 0.2];
\draw (-5,0.5) node{$2^2 3^0$};
\draw[fill=white, radius = .2] (-5,2) circle [radius = 0.2];
\draw (-6,1.5) node{$2^3 3^0$};
\draw[fill=white, radius = .2] (-6,3) circle [radius = 0.2];
\draw (-7,2.5) node{$2^4 3^0$};
\end{tikzpicture}
  \end{minipage}
\end{center}
\end{remark}

\subsection{Enumeration of interval-closed sets in products of two chains}
In this section, we give a formula for the cardinality of $\IC([2]\times[n])$ (Theorem \ref{prodofchainICS}), as well as for the number of interval-closed sets in $\IC([m]\times[n])$ that contain elements in each chain (Theorem \ref{thm:enum_ICS_m_by_n_with_items_in_each_chain}). Finding an enumeration formula for the cardinality of $\IC([a_1]\times [a_2] \times \cdots \times[a_k])$, even for the product of two chains, remains an open problem.  

\begin{thm}\label{prodofchainICS}
The cardinality of $\IC([2] \times [n])$ for $n\geq 2$ is $1+2\left(\binom{n}{2}+n\right)+ \frac{n+1}{2} \binom{n+2}{3}$. This sum counts the following:
\begin{itemize}
\item the empty set $\emptyset$,
\item non-empty intervals that are completely contained in either the top chain or the bottom chain, and
\item pairs made of one non-empty interval in each chain (these may or may not form a single interval in the poset).
\end{itemize}
\end{thm}

\begin{proof}
By construction, the largest antichain of $[2] \times [n]$ has size $2$. Therefore, any interval-closed set is made of at most two disjoint (incomparable) intervals. For an interval-closed set made of two intervals, one of them must belong to the top chain and one to the bottom chain. Given $I \in \IC([2] \times [n])$, there are four cases:
\begin{enumerate}
\item $I$ is empty.
\item $I$ is a single interval, belonging to either the top or bottom chain. 
\item $I$ is a single interval with items in both the top and bottom chains. These can be described as $[(1,i_1),(1, j_1)] \cup [(2,i_2),(2,j_2)]$, with $i_1 \leq j_1$, $i_2 \leq j_2$, $i_2 \leq i_1$, and $j_2 \leq j_1$, where the last two conditions ensure that the set is an interval-closed set. Furthermore, to ensure that this is a single interval, we need that $i_1 \leq j_2$.
\item $I$ is made of two disjoint intervals $[(1,i_1), (1,j_1)]$ and $[(2,i_2),(2,j_2)]$, with $i_2 \leq j_2 < i_1 \leq j_1$.
\end{enumerate}

In Case (2), the number of interval-closed sets is twice as much as the number of non-empty interval-closed sets in the chain poset $[n]$. By Theorem \ref{thm:chains_ICS}, this is $2\left(\binom{n}{2}+n\right)$.
For counting purposes, we can group Cases (3) and (4) and count the number of quadruples $(i_1, j_1, i_2, j_2)$ satisfying the constraint $(i_2 \leq i_1, j_2 \leq j_1)$. Therefore, the number of interval-closed sets with items in both the top and bottom chains is
\[  \sum_{1 \leq i_2 ,j_1 \leq n} (j_1-i_2+1)^2 = \sum_{i_2=1}^n  \sum_{\ell=1}^{n-i_2+1} \ell^2 = \frac{n+1}{2}\binom{n+2}{3}, \]
where the first equality is obtained by setting $l = j_1-i_2+1$, and last equality is obtained by a straightforward computation.
\end{proof}

\begin{ex}Let $P = [2] \times [7]$. Then $[(2, 2), (2, 6)]$ (below, left) is an example of an interval-closed set that is made up of a non-empty interval contained in the top chain, and $[(1, 4),(1,7)]\cup[(2, 2), (2,6)]$ (below, right) is an example of aninterval-closed set made up of a single interval containing elements from both the top and bottom chains. 

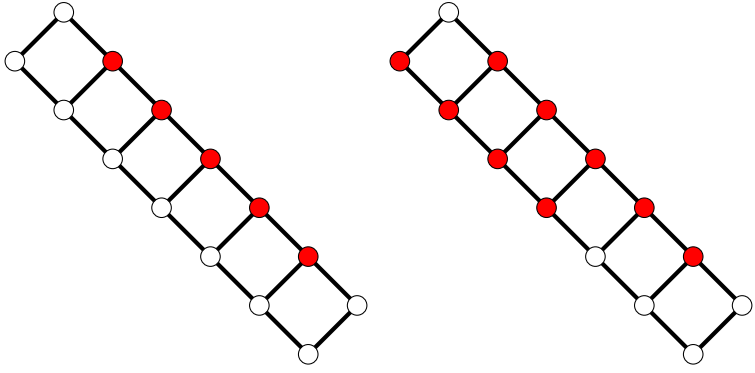
\begin{center}
\begin{minipage}[c]{5cm}
\begin{tikzpicture}[scale = .65]

\draw [-, ultra thick] (-1,0) -- (-2,1);
\draw [-, ultra thick] (-2,1) -- (-3,2);
\draw [-, ultra thick] (-3,2) -- (-4,3);
\draw [-, ultra thick] (-4,3) -- (-5,4);
\draw [-, ultra thick] (-5,4) -- (-6,5);
\draw [-, ultra thick] (-3,0) -- (-4,1);
\draw [-, ultra thick] (-2,-1) -- (-3,0);
\draw [-, ultra thick] (-4,1) -- (-5,2);
\draw [-, ultra thick] (-5,2) -- (-6,3);
\draw [-, ultra thick] (-6,3) -- (-7,4);
\draw [-, ultra thick] (-7,4) -- (-8,5);
\draw [-, ultra thick] (-6,5) -- (-7,6);
\draw [-, ultra thick] (-1,0) -- (-2,-1);
\draw [-, ultra thick] (-2,1) -- (-3,0);
\draw [-, ultra thick] (-3,2) -- (-4,1);
\draw [-, ultra thick] (-4,3) -- (-5,2);
\draw [-, ultra thick] (-5,4) -- (-6,3);
\draw [-, ultra thick] (-6,5) -- (-7,4);
\draw [-, ultra thick] (-7,6) -- (-8,5);
\draw[fill=white, radius = .2] (-1,0) circle [radius = 0.2];
\draw[fill=red, radius = .2] (-2,1) circle [radius = 0.2];
\draw[fill=red, radius = .2] (-3,2) circle [radius = 0.2];
\draw[fill=red, radius = .2] (-4,3) circle [radius = 0.2];
\draw[fill=red, radius = .2] (-5,4) circle [radius = 0.2];
\draw[fill=red, radius = .2] (-6,5) circle [radius = 0.2];
\draw[fill=white, radius = .2] (-7,6) circle [radius = 0.2];
\draw[fill=white, radius = .2] (-2,-1) circle [radius = 0.2];
\draw[fill=white, radius = .2] (-3,0) circle [radius = 0.2];
\draw[fill=white, radius = .2] (-4,1) circle [radius = 0.2];
\draw[fill=white, radius = .2] (-5,2) circle [radius = 0.2];
\draw[fill=white, radius = .2] (-6,3) circle [radius = 0.2];
\draw[fill=white, radius = .2] (-7,4) circle [radius = 0.2];
\draw[fill=white, radius = .2] (-8,5) circle [radius = 0.2];
\end{tikzpicture}
  \end{minipage}
\begin{minipage}[c]{5cm}
\begin{tikzpicture}[scale = .65]
\draw [-, ultra thick] (-1,0) -- (-2,1);
\draw [-, ultra thick] (-2,1) -- (-3,2);
\draw [-, ultra thick] (-3,2) -- (-4,3);
\draw [-, ultra thick] (-4,3) -- (-5,4);
\draw [-, ultra thick] (-5,4) -- (-6,5);
\draw [-, ultra thick] (-3,0) -- (-4,1);
\draw [-, ultra thick] (-2,-1) -- (-3,0);
\draw [-, ultra thick] (-4,1) -- (-5,2);
\draw [-, ultra thick] (-5,2) -- (-6,3);
\draw [-, ultra thick] (-6,3) -- (-7,4);
\draw [-, ultra thick] (-7,4) -- (-8,5);
\draw [-, ultra thick] (-6,5) -- (-7,6);
\draw [-, ultra thick] (-1,0) -- (-2,-1);
\draw [-, ultra thick] (-2,1) -- (-3,0);
\draw [-, ultra thick] (-3,2) -- (-4,1);
\draw [-, ultra thick] (-4,3) -- (-5,2);
\draw [-, ultra thick] (-5,4) -- (-6,3);
\draw [-, ultra thick] (-6,5) -- (-7,4);
\draw [-, ultra thick] (-7,6) -- (-8,5);
\draw[fill=white, radius = .2] (-1,0) circle [radius = 0.2];
\draw[fill=red, radius = .2] (-2,1) circle [radius = 0.2];
\draw[fill=red, radius = .2] (-3,2) circle [radius = 0.2];
\draw[fill=red, radius = .2] (-4,3) circle [radius = 0.2];
\draw[fill=red, radius = .2] (-5,4) circle [radius = 0.2];
\draw[fill=red, radius = .2] (-6,5) circle [radius = 0.2];
\draw[fill=white, radius = .2] (-7,6) circle [radius = 0.2];
\draw[fill=white, radius = .2] (-2,-1) circle [radius = 0.2];
\draw[fill=white, radius = .2] (-3,0) circle [radius = 0.2];
\draw[fill=white, radius = .2] (-4,1) circle [radius = 0.2];
\draw[fill=red, radius = .2] (-5,2) circle [radius = 0.2];
\draw[fill=red, radius = .2] (-6,3) circle [radius = 0.2];
\draw[fill=red, radius = .2] (-7,4) circle [radius = 0.2];
\draw[fill=red, radius = .2] (-8,5) circle [radius = 0.2];
\end{tikzpicture}
  \end{minipage}
\end{center}
\end{ex}

The last proof leads us to ask if there is a way to enumerate interval-closed sets of $[m]\times[n]$ that have elements in all $m$ parallel chains. We show that those are counted by the Narayana numbers by constructing a bijection between order ideals of $[m]\times [n-1]\times [2]$ and interval-closed sets of $[m] \times [n]$ that have elements in all $m$ chains, assuming $n\geq 2$. 

\begin{thm}\label{thm:enum_ICS_m_by_n_with_items_in_each_chain}
Suppose $n\geq 2$. The number of interval-closed sets of $[m] \times [n]$ containing at least one element of the form $(a, b)$ for each $a \in [m]$ is the Narayana number
\[ N(n+m,n) = \frac{\binom{n+m}{n-1}\binom{n+m-1}{n-1} }{n}.  \]
\end{thm}

We use the following lemma in the proof of the theorem. Though the result is known (see e.g. \cite[A001263]{OEIS}, \cite[Thm.\ 7.8]{SW2012}), we include a short proof.
\begin{lem}\label{lem:order_ideal_ab2_Narayana} 
The number of order ideals of $[m]\times[n-1]\times[2]$ is the Narayana number $N(m+n,n)$.
\end{lem}
\begin{proof}
Order ideals of $[m]\times[r]\times[c]$ are enumerated as a ratio of binomial coefficients in \cite[Thm.\ 18.1]{Stanley_PP_II}. By substituting $r=n-1$ and $c=2$ we obtain that the number of order ideals of $[m]\times[n-1]\times[2]$ is
\[ \frac{\binom{m+n}{n-1}\binom{m+n-1}{n-1}}{\binom{n-1}{n-1}\binom{n}{n-1}} = \frac{\binom{m+n}{n-1}\binom{m+n-1}{n-1}}{n} = N(m+n, n), \]
where the Narayana number $N(j,k)$ is given by 
\[ N(j,k) = \frac{1}{k}\binom{j}{k-1}\binom{j-1}{k-1}. \]
\end{proof}

\begin{proof}[Proof of Theorem~\ref{thm:enum_ICS_m_by_n_with_items_in_each_chain}]
In light of Lemma \ref{lem:order_ideal_ab2_Narayana}, it is sufficient to establish a bijection between the order ideals of $[m]\times[n-1]\times[2]$ and the interval-closed sets of $[m]\times[n]$ that contain at least one element of the form $(a,b)$ for each $a \in [m]$.

To do so, we use Proposition \ref{prop:ICS_alt_def}, that describes  an interval-closed set $I$ as a pair of disjoint antichains $(\Max(I),\fl(I))$, in which $\fl(I)$ is strictly below $\Max(I)$. Then the interval-closed set is the difference of the order ideals generated by $\Max(I)$ and $\fl(I)$.

Let $J$ be an order ideal of $[m]\times[n-1]\times[2]$, and define $A_1$ to be the antichain of maxima of $J$ in the restriction of the poset to the elements with $1$ as their last coordinate. Similarly, define $A_2$ to be the antichain of maxima of $J$ in the restriction of the poset to the elements with $2$ as their last coordinate. The pair $(A_1, A_2)$ completely defines $J$. Because $J$ is an order ideal, if $(a,b,2) \in A_2$, then an element $(a,c,1)$ can belong to $A_1$ only if $c\geq b$. Therefore, the antichain defining $J$ has exactly one element of each of the forms $(a,b,2)$ and $(a,c,1)$ for each $a \in [m]$ and for some $c\in \{0,\ldots, n-1\}$, $b \in \{0,\ldots,c\}$. (Here we write $(x,0,y)$ to say that there is no element of $J$ with first coordinate $x$ and last coordinate $y$.)

Consider the values $(a,b,2)$ and $(a,c,1)$ defining the contours of $J$. We claim that we can build an interval-closed set of $[m]\times[n]$ using them. Fix $a \in [m]$ and consider $(a,b,2)$ and $(a,c,1)$ that are at the top of the antichains in $J$. Then, the set of elements $\{(a,x)\mid b < x\leq c+1 \}$ forms a non-empty interval-closed set of $[m]\times[n]$. Furthermore, from Proposition \ref{prop:ICS_alt_def} we may construct an interval-closed set $I$ such that $\Max(I) = \{(a,x+1) \mid (a,x,1) \in A_1\}$ and $\fl(I) = \{(a,y) \mid (a,y,2) \in A_2\}$. This interval-closed set has elements of the form $(a,b)$ for each $a \in [m]$. Let the map just described be denoted as $\psi$, that is, $\psi(J)=I$.

Similarly, one may construct the inverse map as follows. 
Given an interval-closed set $I$ of $[m] \times [n]$ that has elements of the form $(a,b)$ for each $a \in [m]$, 
we use Proposition \ref{prop:ICS_alt_def} to translate it into two antichains $\Max(I)$ and $\fl(I)$. Then, an order ideal of $[m] \times [n-1] \times [2]$ is defined by all the elements below $\{(a,x-1,1) \mid (a,x) \in \max(I)\} \cup \{(a,y,2) \mid (a,y) \in \fl(I)\}$. This map is clearly the inverse of $\psi$.
An example of this bijection is given in Example \ref{ex:ICS-OI_ex}.

Finally, using Lemma \ref{lem:order_ideal_ab2_Narayana}, the number of interval-closed sets of $[m]\times[n]$ containing at least one element of the form $(a, b)$ for each $a \in [m]$ is the Narayana number $N(n+m,n)$.
\end{proof}

\begin{ex}\label{ex:ICS-OI_ex}
We illustrate the bijection with a small example, namely, an instance of the bijection map $\psi$ from an order ideal of $[2]\times[2]\times[2]$ to an interval-closed set of $[2]\times [3]$ with elements in both the top and bottom chain. We start with order ideal $J=\{(1,1,1),(1,1,2),(1,2,1),(1,2,2),(2,1,1)\}$ of $[2]\times[2]\times[2]$. The bijection requires us to identify the maximal elements of the restriction of $J$ to items with last coordinate being respectively $1$ and $2$: here, these are $A_1 = \{(1,2,1), (2,1,1)\}$ and $A_2=\{(1,2,2)\}$. Therefore, for each pair $(a,d) \in \{1,2\}\times\{1,2\}$, we identify the elements of the form $(a, x,d)$ in $A_1$ or $A_2$. Here, these are $\{(1,2,1), (1,2,2), (2,1,1), (2,0,2)\}$. The description of the bijection then yields that $\psi(J)  = \Delta(\{(1,3),(2,2)\}) - \underbrace{\Delta(\{(1,2), (2,0)\}}_{\Delta(\{(1,2)\})} = \{(1,3), (2,1), (2,2)\}$.

Hence, $\psi(J) =  \{(1,3),(2,1), (2,2)\}$ as in Figure \ref{fig:ICS-OI_ex}.\\

For the inverse direction, we start with $I = \{(1,3),(2,1), (2,2)\}$. Then, $\Max(I) = \{(1,3), (2,2)\}$, and $\fl(I) = \{(1,2)\}$. Hence, $\psi^{-1}(I) = \Delta(\{(1,2,1),(2,1,1),(1,2,2) \}) = J$. 
\end{ex}

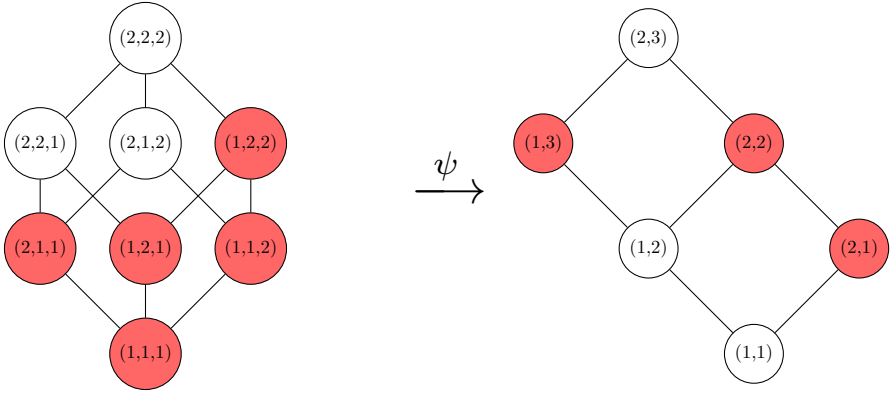
\begin{figure}
\centering
\begin{minipage}{.3\linewidth}
\begin{tikzpicture}[scale = .7]
\node[draw, circle, fill=red!60, scale=.7] (A) at (2,0) {(1,1,1)};
\node[draw, circle, fill=red!60, scale=.7] (B) at (0,2) {(2,1,1)};
\node[draw, circle, fill=red!60, scale=.7] (C) at (2,2) {(1,2,1)};
\node[draw, circle, fill=red!60, scale=.7] (D) at (4,2) {(1,1,2)};
\node[draw, circle, scale=.7] (E) at (0,4) {(2,2,1)};
\node[draw, circle, scale=.7] (F) at (2,4) {(2,1,2)};
\node[draw, circle, fill=red!60, scale=.7] (G) at (4,4) {(1,2,2)};
\node[draw, circle, scale=.7] (H) at (2,6) {(2,2,2)};
\draw (A) --(B) --(E) -- (H);
\draw (A) --(C) --(E);
\draw (A) --(D) --(F) -- (H);
\draw (B) --(F);
\draw (C) --(G);
\draw (D) --(G) -- (H);
\end{tikzpicture}
\end{minipage}
\quad \resizebox{1cm}{!}{$\xlongrightarrow{\large \psi}$}\quad
\begin{minipage}{.4\linewidth}
\begin{tikzpicture}[scale = .7]
\node[draw, circle, scale=.7] (A) at (4,0) {(1,1)};
\node[draw, circle, fill=red!60, scale=.7] (B) at (6,2) {(2,1)};
\node[draw, circle, scale=.7] (C) at (2,2) {(1,2)};
\node[draw, circle, fill=red!60, scale=.7] (E) at (4,4) {(2,2)};
\node[draw, circle, fill=red!60, scale=.7] (F) at (0,4) {(1,3)};
\node[draw, circle, scale=.7] (G) at (2,6) {(2,3)};
\draw (A) --(B) --(E) -- (G);
\draw (A) --(C) --(F) --(G);
\draw (C) -- (E);
\end{tikzpicture}
\end{minipage}
\caption{The bijection $\psi$ between order ideals of $[2]\times[2]\times[2]$ and interval-closed sets of $[2]\times [3]$ with elements in both the top and bottom chain. Here, the order ideal of $[2]\times[2]\times[2]$ illustrated is $J=\{(1,1,1),(1,1,2),(1,2,1),(1,2,2),(2,1,1)\}$ (shown in red on the left), so that $\psi(J) =  \{(1,3),(2,1), (2,2)\}$ is an interval-closed set of $[2]\times [3]$ with elements in both the top and bottom chain (shown in red on the right).}
\label{fig:ICS-OI_ex}
\end{figure}

Unfortunately, Theorem \ref{thm:enum_ICS_m_by_n_with_items_in_each_chain} does not give enough information to enumerate all interval-closed sets in the poset $[m]\times[n]$. 
We leave as an open problem to count the number of interval-closed sets in products of two chains and provide some data for small values of $m+n$ in Table \ref{table:number_ICS_a_by_b}.

\begin{table}[htbp]
\begin{center}
\begin{tabular}{c|c|c|c|c|c|c|c|c}
\diagbox{$m$}{$n$} & 1 & 2 & 3 & 4 & 5 & 6 & 7 & 8\\
\hline
1 & 2 & 4 & 7 & 11 & 16 & 22 & 29 & 37 \\
2 & 4 & 13 & 33 & 71 & 136 & 239 & 393 & 613 \\
3 & 7 & 33 & 114 & 321 & 781 & 1,702 & 3,403 & 6,349\\
4 & 11 & 71 & 321 & 1146 & 3,449 & 9,115 & 21,743 & 47,737\\
5 & 16 & 136 & 781 & 3,449 &12,578 & 39,614 &111,063& 283,243\\
\end{tabular}
\end{center}
\caption{Number of interval-closed sets in $[m]\times[n]$}\label{table:number_ICS_a_by_b}
\end{table}

\subsection{Max minus min homomesy}\label{sec:prod_chains_homomesy}
Rowmotion on product of chains posets provided the setting in which one of the first instances of the homomesy phenomenon (see Definition~\ref{def:homomesy}) was discovered. Propp and Roby showed that the cardinality statistic  on order ideals of $[m]\times[n]$ exhibits homomesy under rowmotion \cite{PR2015}; this result was extended to $[m]\times[n]\times[2]$ by Vorland \cite{Vorland2019}. It is natural to ask whether rowmotion on interval-closed sets of these posets also exhibits homomesy for some statistic. In this section, we give such a homomesy result for the $[2]\times[n]$ poset in Theorem~\ref{thm:homomesy_2_by_n_max-min} and conjecture an analogous statement for $[m]\times[n]$. The statistic exhibiting homomesy in this case is not the cardinality statistic, but rather the number of maximal elements minus the number of minimal elements of the interval-closed set.

\begin{thm}\label{thm:homomesy_2_by_n_max-min}
The number of maximal elements minus the number of minimal elements is $0$-mesic under rowmotion on $\IC(P)$, for $P=[2] \times [n]$.
\end{thm}

\begin{proof} Consider $P = [2] \times [n]$. Given $I \in \IC(P)$, we saw in the proof of Theorem \ref{prodofchainICS} that there are four types of interval-closed sets $I$: 
\begin{enumerate}
\item $I$ is empty.
\item $I$ is a single interval, belonging to either the top or bottom chain. 
\item $I$ is a single interval with items in both the top and bottom chains of the form $[(1,i_1),(1, j_1)] \cup [(2,i_2),(2,j_2)]$, with $i_2\leq i_1 \leq j_2\leq j_1$.
\item $I$ is made of two disjoint intervals on each chain of the form $[(1,i_1), (1,j_1)]$ and $[(2,i_2),(2,j_2)]$ with $i_2 \leq j_2 < i_1 \leq j_1$.
\end{enumerate}

Let $\mm(I)$ denote the number of maximal elements minus the number of minimal elements of $I$. 
If $I$ is in Case (1), $I$ is the empty set, so clearly $\mm(I)=0$. If $I$ is in Case (2), $I$ is a single interval, so there is one maximal and one minimal element. Thus, $\mm(I) = 1 - 1 =0$. If $I$ is in Case (4), $I$ consists of two disjoint intervals. So each interval contributes a maximal and minimal element, resulting in $\mm(I) = 2 - 2 = 0$.

Therefore, the only case we need study is Case (3), where $I$ consists of two intervals that overlap. We first examine how the number of maximal and minimal elements changes depending on several subcases:

\begin{enumerate}
\item[(3.1)] If $j_1 = j_2$, then there is only one maximal element, $j_2$.
\item[(3.2)] If $j_2 < j_1$, then there are two maximal elements, $j_1$ and $j_2$.
\item[(3.3)] If $i_1 = i_2$, then there is only one minimal element, $i_1$. 
\item[(3.4)] If $i_2 < i_1$, then there are two minimal elements, $i_1$ and $i_2$. 
\end{enumerate}

We consider all the ways to pair (3.1)/(3.2) with (3.3)/(3.4). We have $\mm(I) = 0$ when either $j_1 = j_2$ and $i_1 = i_2$ ((3.1) $+$ (3.3)) or $j_1 > j_2$ and $i_1 > i_2$ ((3.2) $+$ (3.4)),  $\mm(I) = 1$ when $j_2 < j_1$ and $i_1 = i_2$ ((3.2) $+$ (3.3)), and $\mm(I) = -1$ when $j_1 = j_2$ and $i_2 < i_1$ ((3.1) + (3.4)).

We finish the proof by showing that whenever you encounter an $I$ of the form ((3.2) $+$ (3.3)), then repeated applications of rowmotion result in a (possibly empty) sequence of interval-closed sets contributing $0$ to the statistic and finally an interval-closed set contributing $-1$ to the statistic. In this way, we can pair any element in the orbit that contributes $+1$ to an element that contributes $-1$. 

Consider $I \in \IC(P)$ of the form $[(1,i_1),(1, j_1)] \cup [(2,i_2),(2,j_2)]$ with $j_2 < j_1$ and $i_1 = i_2$. If $i_1 \neq n$, then performing rowmotion results in the interval-closed set $[(1,i_1 + 1),(1, j_1 + 1)] \cup [(2,1),(2,j_2 + 1)]$. Performing rowmotion again results in sliding each interval up its chain by one element. This continues until the interval on the bottom chain reaches the top. 

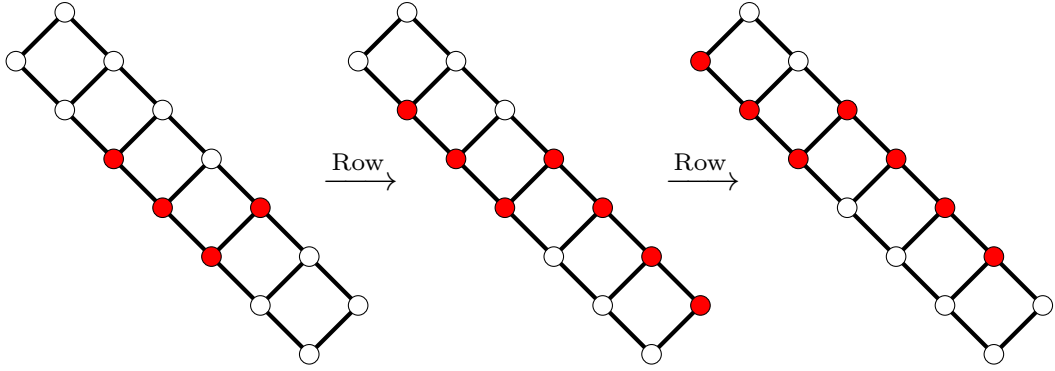
\begin{center}
\begin{minipage}[c]{5cm}
\begin{tikzpicture}[scale = .65]
\draw [-, ultra thick] (-1,0) -- (-2,1);
\draw [-, ultra thick] (-2,1) -- (-3,2);
\draw [-, ultra thick] (-3,2) -- (-4,3);
\draw [-, ultra thick] (-4,3) -- (-5,4);
\draw [-, ultra thick] (-5,4) -- (-6,5);
\draw [-, ultra thick] (-3,0) -- (-4,1);
\draw [-, ultra thick] (-2,-1) -- (-3,0);
\draw [-, ultra thick] (-4,1) -- (-5,2);
\draw [-, ultra thick] (-5,2) -- (-6,3);
\draw [-, ultra thick] (-6,3) -- (-7,4);
\draw [-, ultra thick] (-7,4) -- (-8,5);
\draw [-, ultra thick] (-6,5) -- (-7,6);
\draw [-, ultra thick] (-1,0) -- (-2,-1);
\draw [-, ultra thick] (-2,1) -- (-3,0);
\draw [-, ultra thick] (-3,2) -- (-4,1);
\draw [-, ultra thick] (-4,3) -- (-5,2);
\draw [-, ultra thick] (-5,4) -- (-6,3);
\draw [-, ultra thick] (-6,5) -- (-7,4);
\draw [-, ultra thick] (-7,6) -- (-8,5);
\draw[fill=white, radius = .2] (-1,0) circle [radius = 0.2];
\draw[fill=white, radius = .2] (-2,1) circle [radius = 0.2];
\draw[fill=red, radius = .2] (-3,2) circle [radius = 0.2];
\draw[fill=white, radius = .2] (-4,3) circle [radius = 0.2];
\draw[fill=white, radius = .2] (-5,4) circle [radius = 0.2];
\draw[fill=white, radius = .2] (-6,5) circle [radius = 0.2];
\draw[fill=white, radius = .2] (-7,6) circle [radius = 0.2];
\draw[fill=white, radius = .2] (-2,-1) circle [radius = 0.2];
\draw[fill=white, radius = .2] (-3,0) circle [radius = 0.2];
\draw[fill=red, radius = .2] (-4,1) circle [radius = 0.2];
\draw[fill=red, radius = .2] (-5,2) circle [radius = 0.2];
\draw[fill=red, radius = .2] (-6,3) circle [radius = 0.2];
\draw[fill=white, radius = .2] (-7,4) circle [radius = 0.2];
\draw[fill=white, radius = .2] (-8,5) circle [radius = 0.2];
\end{tikzpicture}
  \end{minipage}
   \hspace{-0.9cm}\resizebox{1cm}{!}{$\xlongrightarrow{\text{Row}}$} \hspace{-0.9cm}
  \begin{minipage}[c]{5cm}
\begin{tikzpicture}[scale = .65]
\draw [-, ultra thick] (-1,0) -- (-2,1);
\draw [-, ultra thick] (-2,1) -- (-3,2);
\draw [-, ultra thick] (-3,2) -- (-4,3);
\draw [-, ultra thick] (-4,3) -- (-5,4);
\draw [-, ultra thick] (-5,4) -- (-6,5);
\draw [-, ultra thick] (-3,0) -- (-4,1);
\draw [-, ultra thick] (-2,-1) -- (-3,0);
\draw [-, ultra thick] (-4,1) -- (-5,2);
\draw [-, ultra thick] (-5,2) -- (-6,3);
\draw [-, ultra thick] (-6,3) -- (-7,4);
\draw [-, ultra thick] (-7,4) -- (-8,5);
\draw [-, ultra thick] (-6,5) -- (-7,6);
\draw [-, ultra thick] (-1,0) -- (-2,-1);
\draw [-, ultra thick] (-2,1) -- (-3,0);
\draw [-, ultra thick] (-3,2) -- (-4,1);
\draw [-, ultra thick] (-4,3) -- (-5,2);
\draw [-, ultra thick] (-5,4) -- (-6,3);
\draw [-, ultra thick] (-6,5) -- (-7,4);
\draw [-, ultra thick] (-7,6) -- (-8,5);
\draw[fill=red, radius = .2] (-1,0) circle [radius = 0.2];
\draw[fill=red, radius = .2] (-2,1) circle [radius = 0.2];
\draw[fill=red, radius = .2] (-3,2) circle [radius = 0.2];
\draw[fill=red, radius = .2] (-4,3) circle [radius = 0.2];
\draw[fill=white, radius = .2] (-5,4) circle [radius = 0.2];
\draw[fill=white, radius = .2] (-6,5) circle [radius = 0.2];
\draw[fill=white, radius = .2] (-7,6) circle [radius = 0.2];
\draw[fill=white, radius = .2] (-2,-1) circle [radius = 0.2];
\draw[fill=white, radius = .2] (-3,0) circle [radius = 0.2];
\draw[fill=white, radius = .2] (-4,1) circle [radius = 0.2];
\draw[fill=red, radius = .2] (-5,2) circle [radius = 0.2];
\draw[fill=red, radius = .2] (-6,3) circle [radius = 0.2];
\draw[fill=red, radius = .2] (-7,4) circle [radius = 0.2];
\draw[fill=white, radius = .2] (-8,5) circle [radius = 0.2];
\end{tikzpicture}
  \end{minipage}
   \hspace{-0.9cm}\resizebox{1cm}{!}{$\xlongrightarrow{\text{Row}}$} \hspace{-0.9cm}
\begin{minipage}[c]{5cm}
\begin{tikzpicture}[scale = .65]
\draw [-, ultra thick] (-1,0) -- (-2,1);
\draw [-, ultra thick] (-2,1) -- (-3,2);
\draw [-, ultra thick] (-3,2) -- (-4,3);
\draw [-, ultra thick] (-4,3) -- (-5,4);
\draw [-, ultra thick] (-5,4) -- (-6,5);
\draw [-, ultra thick] (-3,0) -- (-4,1);
\draw [-, ultra thick] (-2,-1) -- (-3,0);
\draw [-, ultra thick] (-4,1) -- (-5,2);
\draw [-, ultra thick] (-5,2) -- (-6,3);
\draw [-, ultra thick] (-6,3) -- (-7,4);
\draw [-, ultra thick] (-7,4) -- (-8,5);
\draw [-, ultra thick] (-6,5) -- (-7,6);
\draw [-, ultra thick] (-1,0) -- (-2,-1);
\draw [-, ultra thick] (-2,1) -- (-3,0);
\draw [-, ultra thick] (-3,2) -- (-4,1);
\draw [-, ultra thick] (-4,3) -- (-5,2);
\draw [-, ultra thick] (-5,4) -- (-6,3);
\draw [-, ultra thick] (-6,5) -- (-7,4);
\draw [-, ultra thick] (-7,6) -- (-8,5);
\draw[fill=white, radius = .2] (-1,0) circle [radius = 0.2];
\draw[fill=red, radius = .2] (-2,1) circle [radius = 0.2];
\draw[fill=red, radius = .2] (-3,2) circle [radius = 0.2];
\draw[fill=red, radius = .2] (-4,3) circle [radius = 0.2];
\draw[fill=red, radius = .2] (-5,4) circle [radius = 0.2];
\draw[fill=white, radius = .2] (-6,5) circle [radius = 0.2];
\draw[fill=white, radius = .2] (-7,6) circle [radius = 0.2];
\draw[fill=white, radius = .2] (-2,-1) circle [radius = 0.2];
\draw[fill=white, radius = .2] (-3,0) circle [radius = 0.2];
\draw[fill=white, radius = .2] (-4,1) circle [radius = 0.2];
\draw[fill=white, radius = .2] (-5,2) circle [radius = 0.2];
\draw[fill=red, radius = .2] (-6,3) circle [radius = 0.2];
\draw[fill=red, radius = .2] (-7,4) circle [radius = 0.2];
\draw[fill=red, radius = .2] (-8,5) circle [radius = 0.2];
\end{tikzpicture}
  \end{minipage}
\end{center}

At this point, we must consider two cases. If the interval on the top chain is more than one away from the top of the chain, then performing rowmotion slides that top interval up by one.  It removes elements from the top of the bottom chain until the maximal element in the bottom chain is at the same height as that on the top chain. It also removes the bottom most element of the bottom chain.  More formally, if $I'=[(1,i_1'),(1,n)]\cup[(2,i_2'),(2,j_2')]$ with $i_1'=j_2'$ and $j_2'<n-1$, then $\row(I')=[(1,i_2'+1),(1,j_2'+1)]\cup[(2,i_2'+1),(2,j_2'+1)].$ Thus the resulting interval closed set is of the form ((3.1) $+$ (3.4)) and the number of maximal elements minus the number of minimal elements is $1 - 2 = -1$

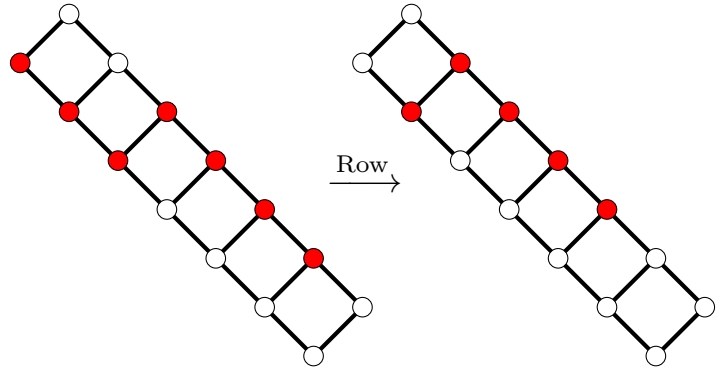
\begin{center}
\begin{minipage}[c]{5cm}
\begin{tikzpicture}[scale = .65]
\draw [-, ultra thick] (-1,0) -- (-2,1);
\draw [-, ultra thick] (-2,1) -- (-3,2);
\draw [-, ultra thick] (-3,2) -- (-4,3);
\draw [-, ultra thick] (-4,3) -- (-5,4);
\draw [-, ultra thick] (-5,4) -- (-6,5);
\draw [-, ultra thick] (-3,0) -- (-4,1);
\draw [-, ultra thick] (-2,-1) -- (-3,0);
\draw [-, ultra thick] (-4,1) -- (-5,2);
\draw [-, ultra thick] (-5,2) -- (-6,3);
\draw [-, ultra thick] (-6,3) -- (-7,4);
\draw [-, ultra thick] (-7,4) -- (-8,5);
\draw [-, ultra thick] (-6,5) -- (-7,6);
\draw [-, ultra thick] (-1,0) -- (-2,-1);
\draw [-, ultra thick] (-2,1) -- (-3,0);
\draw [-, ultra thick] (-3,2) -- (-4,1);
\draw [-, ultra thick] (-4,3) -- (-5,2);
\draw [-, ultra thick] (-5,4) -- (-6,3);
\draw [-, ultra thick] (-6,5) -- (-7,4);
\draw [-, ultra thick] (-7,6) -- (-8,5);
\draw[fill=white, radius = .2] (-1,0) circle [radius = 0.2];
\draw[fill=red, radius = .2] (-2,1) circle [radius = 0.2];
\draw[fill=red, radius = .2] (-3,2) circle [radius = 0.2];
\draw[fill=red, radius = .2] (-4,3) circle [radius = 0.2];
\draw[fill=red, radius = .2] (-5,4) circle [radius = 0.2];
\draw[fill=white, radius = .2] (-6,5) circle [radius = 0.2];
\draw[fill=white, radius = .2] (-7,6) circle [radius = 0.2];
\draw[fill=white, radius = .2] (-2,-1) circle [radius = 0.2];
\draw[fill=white, radius = .2] (-3,0) circle [radius = 0.2];
\draw[fill=white, radius = .2] (-4,1) circle [radius = 0.2];
\draw[fill=white, radius = .2] (-5,2) circle [radius = 0.2];
\draw[fill=red, radius = .2] (-6,3) circle [radius = 0.2];
\draw[fill=red, radius = .2] (-7,4) circle [radius = 0.2];
\draw[fill=red, radius = .2] (-8,5) circle [radius = 0.2];
\end{tikzpicture}
  \end{minipage}
   \hspace{-0.9cm}\resizebox{1cm}{!}{$\xlongrightarrow{\text{Row}}$} \hspace{-0.9cm}
\begin{minipage}[c]{5cm}
\begin{tikzpicture}[scale = .65]
\draw [-, ultra thick] (-1,0) -- (-2,1);
\draw [-, ultra thick] (-2,1) -- (-3,2);
\draw [-, ultra thick] (-3,2) -- (-4,3);
\draw [-, ultra thick] (-4,3) -- (-5,4);
\draw [-, ultra thick] (-5,4) -- (-6,5);
\draw [-, ultra thick] (-3,0) -- (-4,1);
\draw [-, ultra thick] (-2,-1) -- (-3,0);
\draw [-, ultra thick] (-4,1) -- (-5,2);
\draw [-, ultra thick] (-5,2) -- (-6,3);
\draw [-, ultra thick] (-6,3) -- (-7,4);
\draw [-, ultra thick] (-7,4) -- (-8,5);
\draw [-, ultra thick] (-6,5) -- (-7,6);
\draw [-, ultra thick] (-1,0) -- (-2,-1);
\draw [-, ultra thick] (-2,1) -- (-3,0);
\draw [-, ultra thick] (-3,2) -- (-4,1);
\draw [-, ultra thick] (-4,3) -- (-5,2);
\draw [-, ultra thick] (-5,4) -- (-6,3);
\draw [-, ultra thick] (-6,5) -- (-7,4);
\draw [-, ultra thick] (-7,6) -- (-8,5);
\draw[fill=white, radius = .2] (-1,0) circle [radius = 0.2];
\draw[fill=white, radius = .2] (-2,1) circle [radius = 0.2];
\draw[fill=red, radius = .2] (-3,2) circle [radius = 0.2];
\draw[fill=red, radius = .2] (-4,3) circle [radius = 0.2];
\draw[fill=red, radius = .2] (-5,4) circle [radius = 0.2];
\draw[fill=red, radius = .2] (-6,5) circle [radius = 0.2];
\draw[fill=white, radius = .2] (-7,6) circle [radius = 0.2];
\draw[fill=white, radius = .2] (-2,-1) circle [radius = 0.2];
\draw[fill=white, radius = .2] (-3,0) circle [radius = 0.2];
\draw[fill=white, radius = .2] (-4,1) circle [radius = 0.2];
\draw[fill=white, radius = .2] (-5,2) circle [radius = 0.2];
\draw[fill=white, radius = .2] (-6,3) circle [radius = 0.2];
\draw[fill=red, radius = .2] (-7,4) circle [radius = 0.2];
\draw[fill=white, radius = .2] (-8,5) circle [radius = 0.2];
\end{tikzpicture}
  \end{minipage}
\end{center}

If the interval on the top chain is one away from the top of the chain, then performing rowmotion results in sliding the interval on the top chain up by one element and removing an element from the bottom of the interval on the bottom chain. Thus, the resulting interval-closed set is of the form $[(1,i_1'+1),(1, n)] \cup [(2,i_2'+1),(2,n)]$ with $i_1' > i_2',$ and the number of maximal elements minus the number of minimal elements is $1 - 2 = -1$.

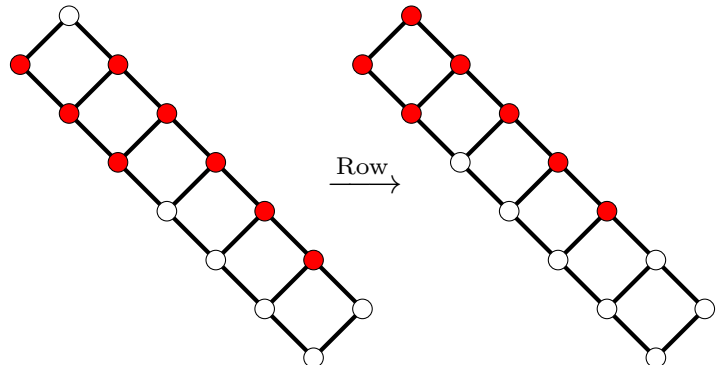
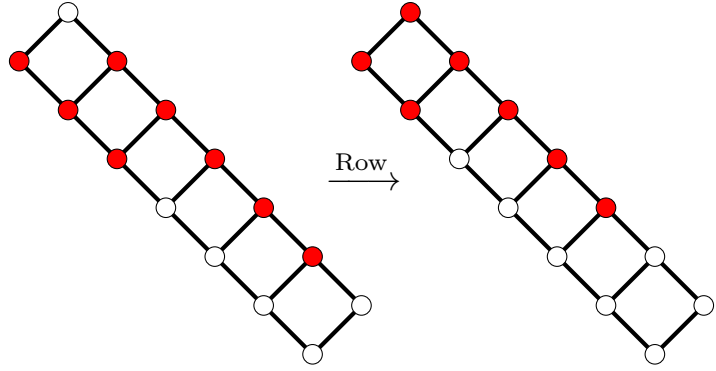
\begin{center}
\begin{minipage}[c]{5cm}
\begin{tikzpicture}[scale = .65]
\draw [-, ultra thick] (-1,0) -- (-2,1);
\draw [-, ultra thick] (-2,1) -- (-3,2);
\draw [-, ultra thick] (-3,2) -- (-4,3);
\draw [-, ultra thick] (-4,3) -- (-5,4);
\draw [-, ultra thick] (-5,4) -- (-6,5);
\draw [-, ultra thick] (-3,0) -- (-4,1);
\draw [-, ultra thick] (-2,-1) -- (-3,0);
\draw [-, ultra thick] (-4,1) -- (-5,2);
\draw [-, ultra thick] (-5,2) -- (-6,3);
\draw [-, ultra thick] (-6,3) -- (-7,4);
\draw [-, ultra thick] (-7,4) -- (-8,5);
\draw [-, ultra thick] (-6,5) -- (-7,6);
\draw [-, ultra thick] (-1,0) -- (-2,-1);
\draw [-, ultra thick] (-2,1) -- (-3,0);
\draw [-, ultra thick] (-3,2) -- (-4,1);
\draw [-, ultra thick] (-4,3) -- (-5,2);
\draw [-, ultra thick] (-5,4) -- (-6,3);
\draw [-, ultra thick] (-6,5) -- (-7,4);
\draw [-, ultra thick] (-7,6) -- (-8,5);
\draw[fill=white, radius = .2] (-1,0) circle [radius = 0.2];
\draw[fill=red, radius = .2] (-2,1) circle [radius = 0.2];
\draw[fill=red, radius = .2] (-3,2) circle [radius = 0.2];
\draw[fill=red, radius = .2] (-4,3) circle [radius = 0.2];
\draw[fill=red, radius = .2] (-5,4) circle [radius = 0.2];
\draw[fill=red, radius = .2] (-6,5) circle [radius = 0.2];
\draw[fill=white, radius = .2] (-7,6) circle [radius = 0.2];
\draw[fill=white, radius = .2] (-2,-1) circle [radius = 0.2];
\draw[fill=white, radius = .2] (-3,0) circle [radius = 0.2];
\draw[fill=white, radius = .2] (-4,1) circle [radius = 0.2];
\draw[fill=white, radius = .2] (-5,2) circle [radius = 0.2];
\draw[fill=red, radius = .2] (-6,3) circle [radius = 0.2];
\draw[fill=red, radius = .2] (-7,4) circle [radius = 0.2];
\draw[fill=red, radius = .2] (-8,5) circle [radius = 0.2];
\end{tikzpicture}
  \end{minipage}
       \hspace{-0.9cm}\resizebox{1cm}{!}{$\xlongrightarrow{\text{Row}}$}   \hspace{-0.9cm}
\begin{minipage}[c]{5cm}
\begin{tikzpicture}[scale = .65]
\draw [-, ultra thick] (-1,0) -- (-2,1);
\draw [-, ultra thick] (-2,1) -- (-3,2);
\draw [-, ultra thick] (-3,2) -- (-4,3);
\draw [-, ultra thick] (-4,3) -- (-5,4);
\draw [-, ultra thick] (-5,4) -- (-6,5);
\draw [-, ultra thick] (-3,0) -- (-4,1);
\draw [-, ultra thick] (-2,-1) -- (-3,0);
\draw [-, ultra thick] (-4,1) -- (-5,2);
\draw [-, ultra thick] (-5,2) -- (-6,3);
\draw [-, ultra thick] (-6,3) -- (-7,4);
\draw [-, ultra thick] (-7,4) -- (-8,5);
\draw [-, ultra thick] (-6,5) -- (-7,6);
\draw [-, ultra thick] (-1,0) -- (-2,-1);
\draw [-, ultra thick] (-2,1) -- (-3,0);
\draw [-, ultra thick] (-3,2) -- (-4,1);
\draw [-, ultra thick] (-4,3) -- (-5,2);
\draw [-, ultra thick] (-5,4) -- (-6,3);
\draw [-, ultra thick] (-6,5) -- (-7,4);
\draw [-, ultra thick] (-7,6) -- (-8,5);
\draw[fill=white, radius = .2] (-1,0) circle [radius = 0.2];
\draw[fill=white, radius = .2] (-2,1) circle [radius = 0.2];
\draw[fill=red, radius = .2] (-3,2) circle [radius = 0.2];
\draw[fill=red, radius = .2] (-4,3) circle [radius = 0.2];
\draw[fill=red, radius = .2] (-5,4) circle [radius = 0.2];
\draw[fill=red, radius = .2] (-6,5) circle [radius = 0.2];
\draw[fill=red, radius = .2] (-7,6) circle [radius = 0.2];
\draw[fill=white, radius = .2] (-2,-1) circle [radius = 0.2];
\draw[fill=white, radius = .2] (-3,0) circle [radius = 0.2];
\draw[fill=white, radius = .2] (-4,1) circle [radius = 0.2];
\draw[fill=white, radius = .2] (-5,2) circle [radius = 0.2];
\draw[fill=white, radius = .2] (-6,3) circle [radius = 0.2];
\draw[fill=red, radius = .2] (-7,4) circle [radius = 0.2];
\draw[fill=red, radius = .2] (-8,5) circle [radius = 0.2];
\end{tikzpicture}
  \end{minipage}
\end{center}

\begin{center}
\begin{minipage}[c]{5cm}
\begin{tikzpicture}[scale = .65]
\draw [-, ultra thick] (-1,0) -- (-2,1);
\draw [-, ultra thick] (-2,1) -- (-3,2);
\draw [-, ultra thick] (-3,2) -- (-4,3);
\draw [-, ultra thick] (-4,3) -- (-5,4);
\draw [-, ultra thick] (-5,4) -- (-6,5);
\draw [-, ultra thick] (-3,0) -- (-4,1);
\draw [-, ultra thick] (-2,-1) -- (-3,0);
\draw [-, ultra thick] (-4,1) -- (-5,2);
\draw [-, ultra thick] (-5,2) -- (-6,3);
\draw [-, ultra thick] (-6,3) -- (-7,4);
\draw [-, ultra thick] (-7,4) -- (-8,5);
\draw [-, ultra thick] (-6,5) -- (-7,6);
\draw [-, ultra thick] (-1,0) -- (-2,-1);
\draw [-, ultra thick] (-2,1) -- (-3,0);
\draw [-, ultra thick] (-3,2) -- (-4,1);
\draw [-, ultra thick] (-4,3) -- (-5,2);
\draw [-, ultra thick] (-5,4) -- (-6,3);
\draw [-, ultra thick] (-6,5) -- (-7,4);
\draw [-, ultra thick] (-7,6) -- (-8,5);
\draw[fill=white, radius = .2] (-1,0) circle [radius = 0.2];
\draw[fill=red, radius = .2] (-2,1) circle [radius = 0.2];
\draw[fill=red, radius = .2] (-3,2) circle [radius = 0.2];
\draw[fill=red, radius = .2] (-4,3) circle [radius = 0.2];
\draw[fill=red, radius = .2] (-5,4) circle [radius = 0.2];
\draw[fill=red, radius = .2] (-6,5) circle [radius = 0.2];
\draw[fill=white, radius = .2] (-7,6) circle [radius = 0.2];
\draw[fill=white, radius = .2] (-2,-1) circle [radius = 0.2];
\draw[fill=white, radius = .2] (-3,0) circle [radius = 0.2];
\draw[fill=white, radius = .2] (-4,1) circle [radius = 0.2];
\draw[fill=white, radius = .2] (-5,2) circle [radius = 0.2];
\draw[fill=red, radius = .2] (-6,3) circle [radius = 0.2];
\draw[fill=red, radius = .2] (-7,4) circle [radius = 0.2];
\draw[fill=red, radius = .2] (-8,5) circle [radius = 0.2];
\end{tikzpicture}
  \end{minipage}
       \hspace{-0.9cm}\resizebox{1cm}{!}{$\xlongrightarrow{\text{Row}}$}   \hspace{-0.9cm}
\begin{minipage}[c]{5cm}
\begin{tikzpicture}[scale = .65]
\draw [-, ultra thick] (-1,0) -- (-2,1);
\draw [-, ultra thick] (-2,1) -- (-3,2);
\draw [-, ultra thick] (-3,2) -- (-4,3);
\draw [-, ultra thick] (-4,3) -- (-5,4);
\draw [-, ultra thick] (-5,4) -- (-6,5);
\draw [-, ultra thick] (-3,0) -- (-4,1);
\draw [-, ultra thick] (-2,-1) -- (-3,0);
\draw [-, ultra thick] (-4,1) -- (-5,2);
\draw [-, ultra thick] (-5,2) -- (-6,3);
\draw [-, ultra thick] (-6,3) -- (-7,4);
\draw [-, ultra thick] (-7,4) -- (-8,5);
\draw [-, ultra thick] (-6,5) -- (-7,6);
\draw [-, ultra thick] (-1,0) -- (-2,-1);
\draw [-, ultra thick] (-2,1) -- (-3,0);
\draw [-, ultra thick] (-3,2) -- (-4,1);
\draw [-, ultra thick] (-4,3) -- (-5,2);
\draw [-, ultra thick] (-5,4) -- (-6,3);
\draw [-, ultra thick] (-6,5) -- (-7,4);
\draw [-, ultra thick] (-7,6) -- (-8,5);
\draw[fill=white, radius = .2] (-1,0) circle [radius = 0.2];
\draw[fill=white, radius = .2] (-2,1) circle [radius = 0.2];
\draw[fill=red, radius = .2] (-3,2) circle [radius = 0.2];
\draw[fill=red, radius = .2] (-4,3) circle [radius = 0.2];
\draw[fill=red, radius = .2] (-5,4) circle [radius = 0.2];
\draw[fill=red, radius = .2] (-6,5) circle [radius = 0.2];
\draw[fill=red, radius = .2] (-7,6) circle [radius = 0.2];
\draw[fill=white, radius = .2] (-2,-1) circle [radius = 0.2];
\draw[fill=white, radius = .2] (-3,0) circle [radius = 0.2];
\draw[fill=white, radius = .2] (-4,1) circle [radius = 0.2];
\draw[fill=white, radius = .2] (-5,2) circle [radius = 0.2];
\draw[fill=white, radius = .2] (-6,3) circle [radius = 0.2];
\draw[fill=red, radius = .2] (-7,4) circle [radius = 0.2];
\draw[fill=red, radius = .2] (-8,5) circle [radius = 0.2];
\end{tikzpicture}
  \end{minipage}
\end{center}

Note that if the original $I$ has $i_1 = n$, then one of those two cases applies and one instance of rowmotion returns an interval-closed set that contributes $-1$ to the statistic.

Therefore, for any $I \in \IC(P)$, the orbit under rowmotion consists of elements that either contribute $0$ to the statistic or elements that contribute $+1$ that can be paired with elements that contribute $-1$. If instead you start with the element that contributes $-1$, use the inverse of rowmotion and the reverse of the above argument to reach the required element that contributes $+1$.
\end{proof}

\begin{ex}
Consider $P = [2] \times [5]$ and $I = [(1,2),(1,3)] \cup [(2, 2),(2,2)]$. In Figure \ref{fig: min_max_fig}, we give the orbit of $I$ with the number of maximal elements minus the number of minimal elements listed below each poset.
\end{ex}

\begin{figure}[bhpt]
\begin{center}
\begin{minipage}[c]{1.7cm}
\begin{tikzpicture}[scale = .43]
\draw [-, ultra thick] (-1,0) -- (-2,1);
\draw [-, ultra thick] (-2,1) -- (-3,2);
\draw [-, ultra thick] (-3,2) -- (-4,3);
\draw [-, ultra thick] (-4,3) -- (-5,4);
\draw [-, ultra thick] (-3,0) -- (-4,1);
\draw [-, ultra thick] (-2,-1) -- (-3,0);
\draw [-, ultra thick] (-4,1) -- (-5,2);
\draw [-, ultra thick] (-5,2) -- (-6,3);
\draw [-, ultra thick] (-1,0) -- (-2,-1);
\draw [-, ultra thick] (-2,1) -- (-3,0);
\draw [-, ultra thick] (-3,2) -- (-4,1);
\draw [-, ultra thick] (-4,3) -- (-5,2);
\draw [-, ultra thick] (-5,4) -- (-6,3);
\draw[fill=white, radius = .2] (-1,0) circle [radius = 0.2];
\draw[fill=red, radius = .2] (-2,1) circle [radius = 0.2];
\draw[fill=white, radius = .2] (-3,2) circle [radius = 0.2];
\draw[fill=white, radius = .2] (-4,3) circle [radius = 0.2];
\draw[fill=white, radius = .2] (-5,4) circle [radius = 0.2];
\draw[fill=white, radius = .2] (-2,-1) circle [radius = 0.2];
\draw[fill=red, radius = .2] (-3,0) circle [radius = 0.2];
\draw[fill=red, radius = .2] (-4,1) circle [radius = 0.2];
\draw[fill=white, radius = .2] (-5,2) circle [radius = 0.2];
\draw[fill=white, radius = .2] (-6,3) circle [radius = 0.2];
\draw (-2,-2) node{$1$};
\end{tikzpicture}
  \end{minipage}
  \hspace{0.15cm} \resizebox{.5cm}{!}{$\xlongrightarrow{\text{Row}}$}   \hspace{-0.85cm}
\begin{minipage}[c]{1.7cm}
\begin{tikzpicture}[scale = .43]
\draw [-, ultra thick] (-1,0) -- (-2,1);
\draw [-, ultra thick] (-2,1) -- (-3,2);
\draw [-, ultra thick] (-3,2) -- (-4,3);
\draw [-, ultra thick] (-4,3) -- (-5,4);
\draw [-, ultra thick] (-3,0) -- (-4,1);
\draw [-, ultra thick] (-2,-1) -- (-3,0);
\draw [-, ultra thick] (-4,1) -- (-5,2);
\draw [-, ultra thick] (-5,2) -- (-6,3);
\draw [-, ultra thick] (-1,0) -- (-2,-1);
\draw [-, ultra thick] (-2,1) -- (-3,0);
\draw [-, ultra thick] (-3,2) -- (-4,1);
\draw [-, ultra thick] (-4,3) -- (-5,2);
\draw [-, ultra thick] (-5,4) -- (-6,3);
\draw[fill=red, radius = .2] (-1,0) circle [radius = 0.2];
\draw[fill=red, radius = .2] (-2,1) circle [radius = 0.2];
\draw[fill=red, radius = .2] (-3,2) circle [radius = 0.2];
\draw[fill=white, radius = .2] (-4,3) circle [radius = 0.2];
\draw[fill=white, radius = .2] (-5,4) circle [radius = 0.2];
\draw[fill=white, radius = .2] (-2,-1) circle [radius = 0.2];
\draw[fill=white, radius = .2] (-3,0) circle [radius = 0.2];
\draw[fill=red, radius = .2] (-4,1) circle [radius = 0.2];
\draw[fill=red, radius = .2] (-5,2) circle [radius = 0.2];
\draw[fill=white, radius = .2] (-6,3) circle [radius = 0.2];
\draw (-2,-2) node{$0$};
\end{tikzpicture}
  \end{minipage}
  \hspace{0.15cm} \resizebox{.5cm}{!}{$\xlongrightarrow{\text{Row}}$}   \hspace{-0.85cm}
  \begin{minipage}[c]{1.7cm}
\begin{tikzpicture}[scale = .43]
\draw [-, ultra thick] (-1,0) -- (-2,1);
\draw [-, ultra thick] (-2,1) -- (-3,2);
\draw [-, ultra thick] (-3,2) -- (-4,3);
\draw [-, ultra thick] (-4,3) -- (-5,4);
\draw [-, ultra thick] (-3,0) -- (-4,1);
\draw [-, ultra thick] (-2,-1) -- (-3,0);
\draw [-, ultra thick] (-4,1) -- (-5,2);
\draw [-, ultra thick] (-5,2) -- (-6,3);
\draw [-, ultra thick] (-1,0) -- (-2,-1);
\draw [-, ultra thick] (-2,1) -- (-3,0);
\draw [-, ultra thick] (-3,2) -- (-4,1);
\draw [-, ultra thick] (-4,3) -- (-5,2);
\draw [-, ultra thick] (-5,4) -- (-6,3);
\draw[fill=white, radius = .2] (-1,0) circle [radius = 0.2];
\draw[fill=red, radius = .2] (-2,1) circle [radius = 0.2];
\draw[fill=red, radius = .2] (-3,2) circle [radius = 0.2];
\draw[fill=red, radius = .2] (-4,3) circle [radius = 0.2];
\draw[fill=white, radius = .2] (-5,4) circle [radius = 0.2];
\draw[fill=white, radius = .2] (-2,-1) circle [radius = 0.2];
\draw[fill=white, radius = .2] (-3,0) circle [radius = 0.2];
\draw[fill=white, radius = .2] (-4,1) circle [radius = 0.2];
\draw[fill=red, radius = .2] (-5,2) circle [radius = 0.2];
\draw[fill=red, radius = .2] (-6,3) circle [radius = 0.2];
\draw (-2,-2) node{$0$};
\end{tikzpicture}
  \end{minipage}
  \hspace{0.15cm} \resizebox{.5cm}{!}{$\xlongrightarrow{\text{Row}}$}   \hspace{-0.85cm}
  \begin{minipage}[c]{1.7cm}
\begin{tikzpicture}[scale = .43]
\draw [-, ultra thick] (-1,0) -- (-2,1);
\draw [-, ultra thick] (-2,1) -- (-3,2);
\draw [-, ultra thick] (-3,2) -- (-4,3);
\draw [-, ultra thick] (-4,3) -- (-5,4);
\draw [-, ultra thick] (-3,0) -- (-4,1);
\draw [-, ultra thick] (-2,-1) -- (-3,0);
\draw [-, ultra thick] (-4,1) -- (-5,2);
\draw [-, ultra thick] (-5,2) -- (-6,3);
\draw [-, ultra thick] (-1,0) -- (-2,-1);
\draw [-, ultra thick] (-2,1) -- (-3,0);
\draw [-, ultra thick] (-3,2) -- (-4,1);
\draw [-, ultra thick] (-4,3) -- (-5,2);
\draw [-, ultra thick] (-5,4) -- (-6,3);
\draw[fill=white, radius = .2] (-1,0) circle [radius = 0.2];
\draw[fill=white, radius = .2] (-2,1) circle [radius = 0.2];
\draw[fill=red, radius = .2] (-3,2) circle [radius = 0.2];
\draw[fill=red, radius = .2] (-4,3) circle [radius = 0.2];
\draw[fill=red, radius = .2] (-5,4) circle [radius = 0.2];
\draw[fill=white, radius = .2] (-2,-1) circle [radius = 0.2];
\draw[fill=white, radius = .2] (-3,0) circle [radius = 0.2];
\draw[fill=white, radius = .2] (-4,1) circle [radius = 0.2];
\draw[fill=white, radius = .2] (-5,2) circle [radius = 0.2];
\draw[fill=red, radius = .2] (-6,3) circle [radius = 0.2];
\draw (-2,-2) node{$-1$};
 \end{tikzpicture}
  \end{minipage}
  \hspace{0.15cm} \resizebox{.5cm}{!}{$\xlongrightarrow{\text{Row}}$}   \hspace{-0.85cm}
  \begin{minipage}[c]{1.7cm}
\begin{tikzpicture}[scale = .43]
\draw [-, ultra thick] (-1,0) -- (-2,1);
\draw [-, ultra thick] (-2,1) -- (-3,2);
\draw [-, ultra thick] (-3,2) -- (-4,3);
\draw [-, ultra thick] (-4,3) -- (-5,4);
\draw [-, ultra thick] (-3,0) -- (-4,1);
\draw [-, ultra thick] (-2,-1) -- (-3,0);
\draw [-, ultra thick] (-4,1) -- (-5,2);
\draw [-, ultra thick] (-5,2) -- (-6,3);
\draw [-, ultra thick] (-1,0) -- (-2,-1);
\draw [-, ultra thick] (-2,1) -- (-3,0);
\draw [-, ultra thick] (-3,2) -- (-4,1);
\draw [-, ultra thick] (-4,3) -- (-5,2);
\draw [-, ultra thick] (-5,4) -- (-6,3);
\draw[fill=red, radius = .2] (-1,0) circle [radius = 0.2];
\draw[fill=red, radius = .2] (-2,1) circle [radius = 0.2];
\draw[fill=white, radius = .2] (-3,2) circle [radius = 0.2];
\draw[fill=white, radius = .2] (-4,3) circle [radius = 0.2];
\draw[fill=white, radius = .2] (-5,4) circle [radius = 0.2];
\draw[fill=red, radius = .2] (-2,-1) circle [radius = 0.2];
\draw[fill=red, radius = .2] (-3,0) circle [radius = 0.2];
\draw[fill=red, radius = .2] (-4,1) circle [radius = 0.2];
\draw[fill=red, radius = .2] (-5,2) circle [radius = 0.2];
\draw[fill=white, radius = .2] (-6,3) circle [radius = 0.2];
\draw (-2,-2) node{$1$};
\end{tikzpicture}
  \end{minipage}
  \hspace{0.15cm} \resizebox{.5cm}{!}{$\xlongrightarrow{\text{Row}}$}   \hspace{-0.85cm}
  \begin{minipage}[c]{1.7cm}
\begin{tikzpicture}[scale = .43]
\draw [-, ultra thick] (-1,0) -- (-2,1);
\draw [-, ultra thick] (-2,1) -- (-3,2);
\draw [-, ultra thick] (-3,2) -- (-4,3);
\draw [-, ultra thick] (-4,3) -- (-5,4);
\draw [-, ultra thick] (-3,0) -- (-4,1);
\draw [-, ultra thick] (-2,-1) -- (-3,0);
\draw [-, ultra thick] (-4,1) -- (-5,2);
\draw [-, ultra thick] (-5,2) -- (-6,3);
\draw [-, ultra thick] (-1,0) -- (-2,-1);
\draw [-, ultra thick] (-2,1) -- (-3,0);
\draw [-, ultra thick] (-3,2) -- (-4,1);
\draw [-, ultra thick] (-4,3) -- (-5,2);
\draw [-, ultra thick] (-5,4) -- (-6,3);
\draw[fill=red, radius = .2] (-1,0) circle [radius = 0.2];
\draw[fill=red, radius = .2] (-2,1) circle [radius = 0.2];
\draw[fill=red, radius = .2] (-3,2) circle [radius = 0.2];
\draw[fill=white, radius = .2] (-4,3) circle [radius = 0.2];
\draw[fill=white, radius = .2] (-5,4) circle [radius = 0.2];
\draw[fill=white, radius = .2] (-2,-1) circle [radius = 0.2];
\draw[fill=red, radius = .2] (-3,0) circle [radius = 0.2];
\draw[fill=red, radius = .2] (-4,1) circle [radius = 0.2];
\draw[fill=red, radius = .2] (-5,2) circle [radius = 0.2];
\draw[fill=red, radius = .2] (-6,3) circle [radius = 0.2];
\draw (-2,-2) node{$0$};
\end{tikzpicture}
  \end{minipage}
  \hspace{0.15cm} \resizebox{.5cm}{!}{$\xlongrightarrow{\text{Row}}$}   \hspace{-0.85cm}
  \begin{minipage}[c]{1.7cm}
\begin{tikzpicture}[scale = .43]
\draw [-, ultra thick] (-1,0) -- (-2,1);
\draw [-, ultra thick] (-2,1) -- (-3,2);
\draw [-, ultra thick] (-3,2) -- (-4,3);
\draw [-, ultra thick] (-4,3) -- (-5,4);
\draw [-, ultra thick] (-3,0) -- (-4,1);
\draw [-, ultra thick] (-2,-1) -- (-3,0);
\draw [-, ultra thick] (-4,1) -- (-5,2);
\draw [-, ultra thick] (-5,2) -- (-6,3);
\draw [-, ultra thick] (-1,0) -- (-2,-1);
\draw [-, ultra thick] (-2,1) -- (-3,0);
\draw [-, ultra thick] (-3,2) -- (-4,1);
\draw [-, ultra thick] (-4,3) -- (-5,2);
\draw [-, ultra thick] (-5,4) -- (-6,3);
\draw[fill=white, radius = .2] (-1,0) circle [radius = 0.2];
\draw[fill=red, radius = .2] (-2,1) circle [radius = 0.2];
\draw[fill=red, radius = .2] (-3,2) circle [radius = 0.2];
\draw[fill=red, radius = .2] (-4,3) circle [radius = 0.2];
\draw[fill=white, radius = .2] (-5,4) circle [radius = 0.2];
\draw[fill=white, radius = .2] (-2,-1) circle [radius = 0.2];
\draw[fill=white, radius = .2] (-3,0) circle [radius = 0.2];
\draw[fill=red, radius = .2] (-4,1) circle [radius = 0.2];
\draw[fill=red, radius = .2] (-5,2) circle [radius = 0.2];
\draw[fill=white, radius = .2] (-6,3) circle [radius = 0.2];
\draw (-2,-2) node{$-1$};
\end{tikzpicture}
  \end{minipage}
  \hspace{0.15cm} \resizebox{.5cm}{!}{$\xlongrightarrow{\text{Row}}$}   \hspace{-0.85cm}
  \begin{minipage}[c]{1.7cm}
\begin{tikzpicture}[scale = .43]
\draw [-, ultra thick] (-1,0) -- (-2,1);
\draw [-, ultra thick] (-2,1) -- (-3,2);
\draw [-, ultra thick] (-3,2) -- (-4,3);
\draw [-, ultra thick] (-4,3) -- (-5,4);
\draw [-, ultra thick] (-3,0) -- (-4,1);
\draw [-, ultra thick] (-2,-1) -- (-3,0);
\draw [-, ultra thick] (-4,1) -- (-5,2);
\draw [-, ultra thick] (-5,2) -- (-6,3);
\draw [-, ultra thick] (-1,0) -- (-2,-1);
\draw [-, ultra thick] (-2,1) -- (-3,0);
\draw [-, ultra thick] (-3,2) -- (-4,1);
\draw [-, ultra thick] (-4,3) -- (-5,2);
\draw [-, ultra thick] (-5,4) -- (-6,3);
\draw[fill=red, radius = .2] (-1,0) circle [radius = 0.2];
\draw[fill=white, radius = .2] (-2,1) circle [radius = 0.2];
\draw[fill=white, radius = .2] (-3,2) circle [radius = 0.2];
\draw[fill=white, radius = .2] (-4,3) circle [radius = 0.2];
\draw[fill=white, radius = .2] (-5,4) circle [radius = 0.2];
\draw[fill=white, radius = .2] (-2,-1) circle [radius = 0.2];
\draw[fill=white, radius = .2] (-3,0) circle [radius = 0.2];
\draw[fill=white, radius = .2] (-4,1) circle [radius = 0.2];
\draw[fill=red, radius = .2] (-5,2) circle [radius = 0.2];
\draw[fill=red, radius = .2] (-6,3) circle [radius = 0.2];
\draw (-2,-2) node{$0$};
\end{tikzpicture}
  \end{minipage}
    \vspace{1mm}
\begin{minipage}{1\linewidth}
\begin{tikzpicture}[scale=.95]\centering
\node[anchor=center, scale=0.6] at (12.5,.25) {Row};
\draw[->] (19.8,.5)--(19.8,0)--(5.3,0)--(5.3,.5);
\node at (3,0) {};
\end{tikzpicture}
\end{minipage}
\end{center}
\caption{An orbit showing the number of maximal elements minus the number of minimal elements}
\label{fig: min_max_fig}
\end{figure}
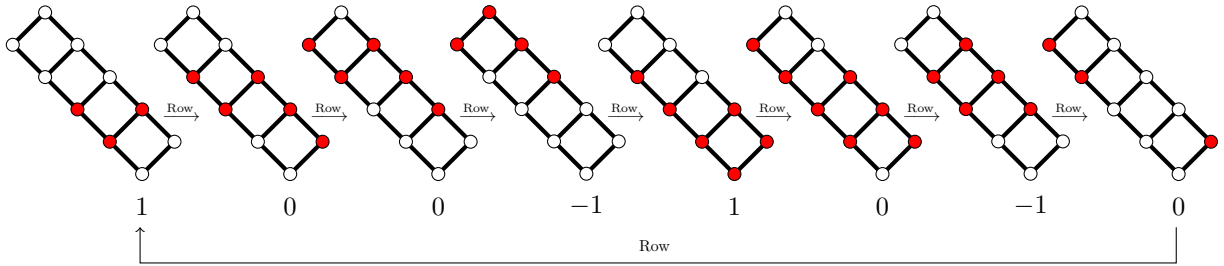

It is reasonable to ask if Theorem \ref{thm:homomesy_2_by_n_max-min} can be generalized to all products of chains. For products of two chains, we tested the following conjecture on all posets $[m] \times [n]$, with $m+n \leq 12$.
\begin{conj}\label{cj:homomesy_prod_2_chains}
The number of maximal elements minus the number of minimal elements is $0$-mesic under rowmotion on $\IC(P)$, for $P=[m] \times [n]$.
\end{conj}

\begin{remark}
Conjecture \ref{cj:homomesy_prod_2_chains} is the best possible such homomesy result, as this statistic does not exhibit homomesy under rowmotion for either of the posets $[2]\times [2]\times [5]$  or   $[2]\times [2]\times [2] \times [2]$.
\end{remark}

\begin{remark}
Theorem \ref{thm:homomesy_2_by_n_max-min} is the only positive homomesy result we found on products of chains. We give here a list of some statistics we determined are not homomesic under rowmotion on interval-closed sets. For each statistic we provide an example of a poset for which the statistic has different orbit averages.
\begin{itemize}
\item  Cardinality: Unlike the case of rowmotion on order ideals \cite{PR2015, Vorland2019}, cardinality (the number of elements in an interval-closed set) is not homomesic for rowmotion on interval-closed sets. Counter-example: $[2]\times[2]$.
\item Signed cardinality: This is the number of elements of an interval-closed set at even ranks minus the number of elements at odd ranks.  Counter-examples: $[2]\times [4]$, $[4]\times[5]$, and $[2]\times[2]\times[3]$. Note that this statistic is homomesic for rowmotion on order ideals of the Type A minuscule poset \cite{Haddadan21}.
\item Number of maximal (respectively, minimal) elements: Unlike what is known for order ideals \cite{PR2015}, this statistic is not homomesic for rowmotion on interval-closed sets. Counter-example: $[2]\times[2]$. 
\item Toggleability for any $x \notin\{1, n \}$ (as explained in Remark \ref{rmk:toggleability_in_general}). Counter-example: $[2]\times[2]$.
\end{itemize}
\end{remark}

The reader may notice that many statistics are not even homomesic for the smallest non-trivial product of chains poset ($[2]\times[2]$); for these statistics, we found several larger counter-examples as well. In the case of the signed cardinality statistic, whose smallest counter-examples are much larger, we offer the following conjecture. We tested this conjecture for all posets with at most 30 elements:

\begin{conj}\label{conj:signed_card} If $m = 2$ or $m = 3,$ then the signed cardinality
statistic is $0$-mesic under rowmotion on interval-closed sets of $[m]\times[n]$ whenever $m + n - 1$ is
even. 
\end{conj}

Many statistics that exhibit homomesy with respect to rowmotion on order ideals also exhibit the \emph{cyclic sieving phenomenon}~\cite{ReStWh2004}. For each statistic studied in this paper and for the diamond poset, we used SageMath to check for occurrences of the cyclic sieving phenomenon for the triplets made of the set of interval-closed sets, the statistic generating function, and the action of rowmotion, but did not find any. Since the diamond poset is both an ordinal sum of antichains and a product of chains, there is no occurrence of the phenomenon with a statistic from this paper and for either family of posets studied here. It would be interesting to find an instance of this phenomenon involving interval-closed sets. 

\bibliographystyle{plain}
\bibliography{master.bib}

\end{document}